\documentclass{agtart_a}
\pdfoutput=1

\usepackage{graphicx} 
\usepackage{pinlabel}

\title[Commensurability and separability of quasiconvex subgroups]
{Commensurability  and separability of quasiconvex subgroups}

\author[F Haglund]{Fr\'{e}d\'{e}ric Haglund}
\givenname{Fr\'{e}d\'{e}ric}
\surname{Haglund}
\address{Laboratoire de Math\'{e}matiques\\
Universit\'{e} de Paris XI (Paris-Sud)\\\newline
91405 Orsay\\
France}
\email{frederic.haglund@math.u-psud.fr}
\urladdr{}

\volumenumber{6}
\issuenumber{}
\publicationyear{2006}
\papernumber{36}
\startpage{949}
\endpage{1024}

\doi{}
\MR{}
\Zbl{}


\received{25 July 2005}
\revised{}
\accepted{23 November 2005}
\published{9 August 2006}
\publishedonline{9 August 2006}
\proposed{}
\seconded{}
\corresponding{}
\editor{}
\version{}

\arxivreference{}  

\AtBeginDocument{\let\bar\wbar\let\tilde\wtilde\let\hat\what}
\def\<{\langle}\def\>{\rangle}
\makeop{CAT}
\makeop{Is}
\def\fbar{{}\mskip4mu\overline{\mskip-4mu f\mskip-1mu}\mskip1mu}
\def\fibar{{}\mskip2mu\overline{\mskip-2mu f^{-{\smash[t] 1}}\mskip-1mu}\mskip1mu}
\def\fpbar{{}\mskip2mu\overline{\mskip-2mu f'\mskip-1mu}\mskip1mu}
\def\iperp{i^{{{}\smash[t]{\perpeq}}}}
\def\ovae{{}\mskip5mu\overrightarrow{\mskip-5mu e}}
\def\part#1{\section*{Part #1}}

\def\tsty{\textstyle}

\makeatletter
\def\cnewtheorem#1[#2]#3{\newtheorem{#1}{#3}[section]
\expandafter\let\csname c@#1\endcsname\c@thm}
\makeatother

\newtheorem{thm}{Theorem}[section]
\cnewtheorem{lem}[thm]{Lemma}
\cnewtheorem{cor}[thm]{Corollary}
\cnewtheorem{prop}[thm]{Proposition}

\theoremstyle{definition}
\cnewtheorem{defn}[thm]{Definition}
\cnewtheorem{rem}[thm]{Remark}
\cnewtheorem{exmp}[thm]{Example}
\cnewtheorem{prob}[thm]{Problem}

\newcommand{\field}[1]{\mathbb{#1}}
\newcommand{\integers}{\ensuremath{\field{Z}}}
\newcommand{\Euclidean}{\ensuremath{\field{E}}}
\newcommand{\size}[1]{\ensuremath{\vert #1 \vert}}

\newcommand{\perpeq}{^{\underline{\perp}}}

\begin{document}

\begin{asciiabstract}
We show that two uniform lattices of a regular right-angled Fuchsian
building  are commensurable, provided the chamber is a polygon with at
least six edges.  We show that in an arbitrary Gromov-hyperbolic regular
right-angled building associated to a graph  product of finite groups,
a uniform lattice is commensurable with the graph  product provided
all of its quasiconvex subgroups are separable. We obtain a similar
result for uniform lattices of the Davis complex of Gromov-hyperbolic
two-dimensional Coxeter groups. We also prove that every extension of a
uniform lattice of a CAT(0) square complex by a finite group is virtually
trivial, provided each quasiconvex subgroup of the lattice is separable.
\end{asciiabstract}

\begin{htmlabstract}
We show that two uniform lattices of a regular right-angled Fuchsian
building  are commensurable, provided the chamber is a polygon with at
least six edges.  We show that in an arbitrary Gromov-hyperbolic regular
right-angled building associated to a graph  product of finite groups,
a uniform lattice is commensurable with the graph  product provided
all of its quasiconvex subgroups are separable. We obtain a similar
result for uniform lattices of the Davis complex of Gromov-hyperbolic
two-dimensional Coxeter groups. We also prove that every extension of a
uniform lattice of a CAT(0) square complex by a finite group is virtually
trivial, provided each quasiconvex subgroup of the lattice is separable.
\end{htmlabstract}

\begin{abstract}
We show that two uniform lattices of a regular right-angled Fuchsian
building  are commensurable, provided the chamber is a polygon with at
least six edges.  We show that in an arbitrary Gromov-hyperbolic regular
right-angled building associated to a graph  product of finite groups,
a uniform lattice is commensurable with the graph  product provided
all of its quasiconvex subgroups are separable. We obtain a similar
result for uniform lattices of the Davis complex of Gromov-hyperbolic
two-dimensional Coxeter groups. We also prove that every extension of a
uniform lattice of a $\CAT(0)$ square complex by a finite group
is virtually trivial, provided each quasiconvex subgroup of the lattice
is separable.
\end{abstract}

\maketitle

\section{Introduction}
\label{intro}

 We consider a special instance of the
following (very general) question:

\noindent If $X$ and $Y$ are compact metric spaces with isometric
universal covers $\tilde X$ and $\tilde Y$, when do $X$ and $Y$ have
isometric compact covers?

Using the fundamental group leads to the following question:

When are two discrete cocompact subgroups of $\Is(\tilde X)$ commensurable
in $\Is(\tilde X)$?

Recall that two subgroups $\Gamma,\Gamma'$ of
a given group $G$ are said to be \textit{commensurable\/}
in $G$ when there is $g\in G$ such that
$g\Gamma g^{-1}\cap \Gamma'$ is of finite index in both $\Gamma'$ and
$g\Gamma g^{-1} $ (taking $G={\mathbb R}$, $\Gamma=\alpha{\mathbb Z}$ and 
$\Gamma'=\alpha'{\mathbb Z}$ explains the origin of the word
commensurable). Two abstract groups $\Gamma,\Gamma'$ are said to be 
\textit{abstractly commensurable\/} if they have finite index isomorphic
subgroups.

 There is a  lot of literature on this
subject when $X$ is a locally symmetric space
(of noncompact type). In this case $X$ is a
Riemannian manifold of nonpositive sectional
curvature and $G=\Is(\tilde X)$ is a Lie
group, the set of real points of some
algebraic group. It is always possible to
construct particular lattices in $G$ -- called
arithmetic because they arise from number
theory (see Borel \cite{Borel63}). When the real rank is at least
two, Margulis showed that all
(irreducible) lattices are necessarily
arithmetic \cite{Margulis84}. Yet all arithmetic
lattices need not be commensurable. When
the real rank is one, then $X$ is negatively curved. Again
there are  noncommensurable  lattices,
because there are nonarithmetic ones by Gromov and Piatetski-Shapiro \cite{GromovPShapiro}.

 We are interested in the case when $\tilde X$ is not a manifold
anymore, but some particular Tits building  or 
Davis
complex, both of which admit a simplicial subdivision. The automorphism groups of these locally compact simplicial complexes are locally compact, totally disconnected, and nondiscrete in many of the cases we are interested in. For us, a
\textit{lattice\/} of a locally compact simplicial complex
$X$ is a subgroup $\Gamma\subset {\rm Aut}(X)$ such that each vertex $v$
has a  finite stabilizer $\Gamma_v$ and  
$\sum \punfrac{1}{\size{\Gamma_v}}<\infty$, where the sum has one term for
each orbit of the $\Gamma$--action on the set of vertices. A lattice $\Gamma\subset
{\rm Aut}(X)$ is said to be \textit{uniform\/} whenever it is cocompact, that is, 
there are finitely many orbits of vertices.

 We first
explain our result concerning buildings.
\subsection{Right-angled buildings}

The  simplest examples of (nonspherical) 
buildings are trees. Here any two uniform
lattices of a  locally finite tree are
commensurable (see Bass and Kulkarni \cite{BassKulkarni} and Leighton \cite{Leighton82}).

In dimension 2,  a Euclidean building $\Delta$ is obtained
by taking the product of a
$q$--regular tree $T_q$ with itself.   Burger and
Mozes 
have produced uniform lattices of $\Delta$
which are simple groups \cite{BurgerMozes2000}. Their construction
begins with an irreducible uniform lattice
$\Gamma$  in $\Delta$ (of Hilbert modular
group nature); by definition, $\Gamma$  is
not commensurable with the product of two
uniform lattices of the factors. In 
the same year, Wise constructed non
residually finite uniform lattices in
$\Delta$, that cannot either be  
commensurable with a product lattice \cite{Wise96Thesis}.

Let us describe more precisely 
$\Delta=T_q\times T_q$. A face of
$\Delta$ is a square, the product of two edges.  An
edge of $\Delta$ incident to a vertex $v$ is one
of two types corresponding to an edge of the
first factor or of the second, and there are $q$ of each type. Furthermore
any two edges of distinct types coming out of a
vertex are contained in one and exactly one
square. In other words the link of $\Delta$ 
at each vertex is the complete bipartite 
graph on
$q+q$ vertices. Note that
$\Delta$ is full of copies of the Euclidean
plane tesselated by squares. Furthermore
$\Delta$ has a nice reducible uniform lattice
$\Gamma_{4, q}$ which we describe now.

The free product of two cyclic groups of order
$q$ acts on its Bass--Serre tree $T_q$: the
action is simply transitive on edges (without
inversion). So this free product $L_q$ is a
uniform lattice of $T_q$. We let $\Gamma_{4,q}$
denote the direct product $L_q\times L_q$.

The spaces  $\Delta$
we want to look
at are the hyperbolic analogues of 
$T_q\times T_q$, in which squares are replaced
by hyperbolic regular right-angled polygons. M
Bourdon was the first to consider these
right-angled Fuchsian buildings in order to
study their geometry \`a la Gromov. For each
integer $q\ge 2$ and each integer $p\ge 4$,
there is a connected, simply connected
CW--complex $I_{p,q}$ such that
\begin{itemize}
\item all attaching maps are injective
\item the boundary of a 2--cell is  cycle of
length $p$
\item the link of each vertex is a  complete
bipartite graph on $q+q$ vertices.
\end{itemize}
\noindent  Furthermore such a CW--complex is
unique up to cellular homeomorphism \cite[Proposition 2.2.1]{Bourdon97}. Bourdon also introduces a group
$\Gamma_{p,q}$ acting on $I_{p,q}$ as a 
uniform lattice, generalizing the case $p=4$
above. For $p\ge 5$ the group
$\Gamma_{p,q}$ is Gromov-hyperbolic, as
$I_{p,q}$ can be equipped with a $\CAT(-1)$
metric (see Meier \cite{Meier96}).

 Our main result answers a
question raised by Bourdon \cite[Section 1.E.2]{Bourdon00}.  It will follow
from a general statement on regular right-angled buildings.

\begin{thm}\label{thm:commBourdon}
For
$p\ge 6$ and
$q\ge 2$, all uniform lattices of $I_{p,q}$ 
are commensurable in ${\rm Aut}(I_{p,q})$.
\end{thm}

\begin{cor}\label{cor:linear}  For
$p\ge 6$ and $q\ge 2$, all uniform lattices of
$I_{p,q}$ are $\mathbb R$--linear.
\end{cor}

\begin{proof}[Proof of
\fullref{cor:linear}] A theorem of Hsu
and Wise shows that all
$\Gamma_{p,q}$ are linear \cite{HsuWiseGraphProducts}.
Every group abstractly commensurable to a
linear group is itself linear (see
Wehrfritz \cite{Wehrfritz73} for a proof).
\end{proof}

 Note that  B R\'emy recently exhibited
nonlinear (nonuniform) lattices in 
${\rm Aut} (I_{p,q})$ \cite{Remy04}.

The groups $\Gamma_{p,q}$ are special  kinds
of  \textit{graph products\/} of finite groups 
$(G_v)_{v\in V({\mathcal G})}$ above a finite
graph ${\mathcal G}$. To each such graph
product
$\Gamma=\Gamma({\mathcal G},(G_v)_{v\in
V({\mathcal G})})$ we associate a proper
$\CAT(0)$ cubical complex
$\Delta=\Delta({\mathcal G},(G_v)_{v\in
V({\mathcal G})})$ such that
$\Delta$ is a  right-angled building and
$\Gamma$ acts naturally on
$\Delta$ as a uniform lattice  (see \fullref{sec:building}
 for precise
definitions). When
$\mathcal G$ is a cyclic graph with $p$
vertices and all finite groups $G_v$ equal  
to
$\unfrac{\mathbb Z}{q\mathbb Z}$, then $\Gamma=
\Gamma_{p,q}$ and $\Delta=I_{p,q}$ (more
precisely $\Delta$ is  the first square
subdivision of the polygonal complex 
$I_{p,q}$).

 We will prove the following criterion:

\begin{thm}\label{thm:commsepar} Suppose
$\Delta({\mathcal G},(G_v)_{v\in V({\mathcal
G})})$ is Gromov-hyperbolic.  Then
a uniform lattice $\Gamma'$ is commensurable
to $\Gamma$ in ${\rm Aut} (\Delta)$ if  all of its
quasiconvex subgroups are separable.
\end{thm}

Here ${\rm Aut} (\Delta)$ is the group of
type-preserving automorphisms (see \fullref{sec:building}).
The hyperbolicity condition is equivalent to
the fact that the graph
$\mathcal G$ does not contain a square as a 
full subgraph (see Gromov \cite{Gromov87}).

Recall that a subgroup $\Lambda$ of a
Gromov-hyperbolic group is \textit{quasiconvex\/}
if a
$\Gamma$--geodesic between two elements of
$\Lambda$ stays at uniformly bounded  distance
of $\Lambda$.  Recall also that a subgroup $\Lambda$ of a  group
$\Gamma$ is \textit{separable\/} if it is the
intersection of all finite index subgroups
containing it. The residual finiteness of
$\Gamma$ is equivalent to the separability of
the trivial subgroup.

The property that all quasiconvex subgroups
be separable 
has been studied for Fuchsian and Kleinian
groups by Scott \cite{Scott78}, Gitik \cite{GitikDoubles}, and Agol, Long and Reid \cite{AgolLongReid2001}.
In \cite{WisePolygons} and \cite{WiseFigure8},
Wise was the first
to explore this property in the framework of
geometric group theory. He also
showed the importance of this property with
respect to the problem of residual finiteness
of Gromov-hyperbolic groups (see Kapovich and Wise \cite{KapovichWise}).

We proved in \cite{Haglundgraphpdct} the converse of
\fullref{thm:commsepar} by establishing
the separability of all
quasiconvex subgroups of graph products of
finite groups.
Hence in the world of
Gromov-hyperbolic right-angled buildings
the  commensurability class of the
``natural'' uniform lattice is characterized
by the separability of quasiconvex subgroups.

\fullref{thm:commBourdon}
is a consequence of
\fullref{thm:commsepar} since 
$I_{p,q}$ is hyperbolic for
$p\ge 5$. A remarkable theorem of Wise \cite[Theorem 10.1]{WisePolygons} insures that  for
$p\ge 6$ and
$q\ge 2$, all quasiconvex subgroups of any uniform
lattice of $I_{p,q}$ are  separable. Note that 
the group of type-preserving automorphisms
of $I_{p,q}$ is of finite index in the full
automorphism group. So
\fullref{thm:commBourdon} is valid even if
we understand ${\rm Aut}(I_{p,q})$ as the full
automorphism group. Observe that in general the
group of type-preserving automorphisms of a
right-angled building is \textit{not\/} of finite
index in the full automorphism group.

In the case $p=5$ the criterion of
\fullref{thm:commsepar} remains valid, 
but Wise's separability result is not
available. So the question remains open.

To establish
\fullref{thm:commsepar} we first introduce
a geometric invariant of the action of the
uniform lattice $\Gamma'$ called holonomy
which is trivial for a naturally defined finite
index subgroup of  $\Gamma'=\Gamma$.  In \fullref{sec:atlas}, 
we prove the following:

\begin{thm}\label{thm:commhol}
A uniform lattice
$\Gamma'$ is commensurable to
$\Gamma$ in ${\rm Aut} (\Delta)$ if and only  if
$\Gamma'$ has a finite index subgroup with
trivial holonomy.
\end{thm}

In \fullref{sec:killhol}, we conclude the proof of
\fullref{thm:commsepar}   by
progressively killing all the holonomy in a
convenient finite index subgroup of $\Gamma'$. This is where the separability
is used.

Note that \fullref{thm:commhol} can be used
to recover a previous result by Januszkiewicz
and Swiatkowski \cite{JanuszkiewiczSwiatkowskiGraphProducts} (see
\fullref{sec:atlas}).

\begin{cor}\label{cor:commJS}  The graph products of
$\Gamma({\mathcal G},(G_v)_{v\in V({G})})$
and 
$\Gamma({\mathcal G},(H_v)_{v\in V({G})})$ 
are commensurable if for each vertex $v$, the
two finite groups
${ G_v}$ and ${
H_v}$ have the same cardinality.
\end{cor}

Locally finite biregular trees are  the simplest infinite regular right-angled
buildings we can consider. But in the case of
a tree the holonomy is automatically trivial
for any  uniform lattice. Hence in \fullref{sec:atlas} we also
recover ``Leighton's lemma''
\cite{AngluinGardiner,Leighton82,BassKulkarni}.

\begin{cor}\label{cor:commL}
 Two
uniform lattices of a locally finite biregular
tree $T$ are commensurable in ${\rm Aut} (T)$.
\end{cor}

\subsection{Davis complexes} We now turn to
our second result linking commensurability and
separability of quasiconvex subgroups. For some
Fuchsian buildings which are not right-angled
we can still prove the
commensurability of all uniform lattices.
By Davis \cite{Davis83} 
and Haglund \cite{HaglundReseaux}, such a Fuchsian building is Davis' geometric realization of
a Coxeter system whose finite nerve $L$ is a
generalized $\mu$--gon with $\mu\ge 3$, and all
finite entries of the Coxeter matrix are 
equal to some integer $m\ge 2$.  
In fact our result is obtained in the framework
of the Davis complexes, the basic properties
of which we recall now.

Let $(m,L)$ be a pair where 
$m$ is any integer and $L$ is any graph
such that $m\ge 3$ or $m=2$ and the
girth of $L$ is $\ge 4$.  The Davis complex associated to $(m,L)$
 is a $\CAT(0)$
polygonal complex
$X(m,L)$ in which each polygon has $2m$ edges
and  the link of each vertex is isomorphic to
$L$. The automorphism group of
$X(m,L)$ contains a subgroup $W(m,L)$ generated
by reflections of $X(m,L)$ which acts
simply transitively on vertices. In this context
we define a reflection to be an order--2
automorphism exchanging the endpoints of some
edge $e$ of
$X(m,L)$, and fixing pointwise the totally
geodesic tree consisting of points equidistant
to the endpoints of $e$. Note that the  (locally
compact) automorphism group of the polygonal
complex
$X(m,L)$ is \textit{not discrete\/}
if and only if $L$ has nontrivial
automorphisms fixing pointwise the star of a
vertex \cite{HaglundReseaux,HaglundPaulin98}.

When $L$ is a generalized $\mu$--gon,
$X(m,L)$ is Davis' geometric realization of a
building whose apartments are tesselations of
$\mathbb H^2$ by regular $2m$--gons with vertex
angle $\unfrac{\pi}{\mu}$.  (In the terminology of
\cite{HaglundExUniHom} these buildings are \textit{locally reflexive without holonomy\/}). For
$\mu=2$ and $L$ finite, we recover regular
locally finite right-angled buildings for which \fullref{thm:commsepar} implies uniform lattices are
commensurable.

In \fullref{sec:killholcox} we
will prove the following generalization.

\begin{thm}\label{thm:commseparbis} Assume
$L$ is a finite graph and 
$X(m,L)$ is negatively curved (that is $m\ge
4$, or $m\ge 3$ and the girth of $L$ is at
least $4$, or $m\ge 2$ and the girth of $L$ is
at least $5$).

Then a uniform lattice
$\Gamma$ in  $X(m,L)$ is commensurable to the
Coxeter group $W(m,L)$ in ${\rm Aut}(X(m,L))$
if all of its quasiconvex subgroup are
separable.
\end{thm}

\begin{cor}
If $L$ is a finite bipartite graph and   $m\ge
3$ then all uniform lattices of $X(m,L)$ are
commensurable (hence they are linear). In
particular all uniform lattices of a given
locally reflexive even-gonal Fuchsian buildings
without holonomy are commensurable.
\end{cor}

\begin{proof}
The bipartite assumption implies that the first square subdivision of
$X(m,L)$ is a VH--complex of Wise (this argument is implicit in the
proof of Theorem 10.1 of
\cite{WisePolygons}). By Theorem 8.1 of
\cite{WisePolygons}, all uniform lattices of
$X(m,L)$ have separable quasiconvex
subgroups. Thus the corollary is a direct
application of \fullref{thm:commseparbis}. 
\end{proof}

The converse of \fullref{thm:commseparbis}
is true at least when $m=2$ \cite{Haglundgraphpdct} or more
generally when $m$ is even (the method we
developed in \cite{Haglundgraphpdct} for graphs of finite groups
also applies to this particular kind of Coxeter
groups). So once again in these complexes
$X(m,L)$ there is a preferred commensurability
class of uniform lattices, characterized by
important separability properties.

The  proof of
\fullref{thm:commseparbis} is parallel to
the proof of \fullref{thm:commsepar}, although
we were not able to find a common proof for
these results. Here we use the notion of
systems of local reflections and of their
holonomy which we introduced in
\cite{HagCras} and \cite{HaglundExUniHom}.
The basic system  of local reflections is given
by the restrictions of the reflections of
$W(m,L)$ to the neighbourhoods of edges. Its
holonomy is obtained by composing these
reflections along the boundary of a polygon:
this is trivial by definition of the relations
of a Coxeter group. In \fullref{sec:killholcox} we obtain an analogy to
\fullref{thm:commhol}.

\begin{thm}\label{thm:commholcox}
A uniform lattice
$\Gamma$ of  $X=X(m,L)$ is commensurable to
$W(m,L)$ in ${\rm Aut} (X)$ if and only if
$\Gamma$ has a finite index subgroup 
preserving a system of local reflections whose
holonomy is trivial.
\end{thm}

In order to deduce \fullref{thm:commseparbis}
from \fullref{thm:commholcox} it remains to
use the separability of certain quasiconvex
subgroups to produce a system of local
reflections without holonomy preserved by a
finite index subgroup of the uniform lattice
$\Gamma$.

Here is a sketch of the argument. Passing to a finite index subgroup, we may
assume that $\Gamma$ preserves a system
$\sigma$ of local reflections. An easy
computation shows that the holonomy of $\sigma$
splits into two parts. We then consider a
kind of walls inside the polygonal complex.
These walls are always quasiconvex under the
negative curvature assumption, and so are their
stabilizers (as well as the finite index
subgroups of these). This allows us to modify
$\sigma$ along the $\Gamma$--orbit of one wall
$M$ in such a way that the new system
of local reflections $\sigma'$ is invariant
under a finite index subgroup of $\Gamma$. Furthermore the holonomy of $\sigma'$ at
polygons transverse to a $\Gamma$--translate of
$M$ is ``half-trivial''. There are finitely many
orbits of walls, so if we iterate this process a finite
number of times we get a finite
index subgroup of $\Gamma$ preserving a system
of local reflections without holonomy.

\subsection{Extensions by finite groups}

In the two situations
above we noticed that quasiconvex subgroup
separability could be used to make certain conjugation obstructions
virtually vanish. We
apply this idea once more to answer a question
raised by Anna Erschler:

If $\Gamma$ is a group studied by Wise in \cite{WisePolygons} and $1\to F\to G\to \Gamma\to 1$
is an exact sequence of groups with $F$ finite,
is $G$ always residually finite? commensurable with $\Gamma$?

This is a natural question since Ragunathan proved in
\cite{Ragunathan84} that a uniform lattice of ${\rm
Spin}(2,n)$, $n$ odd, always admits a finite extension which is not
virtually torsion-free, hence not commensurable with the initial lattice.

The groups studied by Wise in \cite{WisePolygons} are
fundamental groups of certain nonpositively
curved square complexes. In the universal
covers  of such spaces there is a natural
notion of wall: an equivalence class
for the relation on edges generated by being
opposite inside one square. The stabilizer of a
wall is quasiconvex. In more general polygonal
complexes there are several natural notions of
walls.  We choose one and introduce a
combinatorial condition
$(\mathrm{C^2})$ on a polygonal complex $X$ so that when
$(\mathrm{C^2})$ is fulfilled our walls are
automatically convex for a piecewise Euclidean
nonpositively curved metric on $X$. We say that
a subgroup of the fundamental group of a $(\mathrm{C^2})$
polygonal complex $X$ is \textit{convex\/} whenever it
preserves a convex subset $Y\subset \widetilde
X$ of the universal cover and acts
cocompactly on $Y$.

\begin{thm}\label{thm:virttrivext}
Let $X$ denote a compact polygonal complex 
satisfying  the nonpositive curvature
condition $(\mathrm{C^2})$, for example, $X$ is a nonpositively curved square complex.
 Assume that
convex subgroups are
separable in
$\Gamma=\pi _1(X)$.

Then
for any extension $1\to G\to 
\wbar{\Gamma}\to \Gamma\to 1$ where $G$  
is finite, the group 
$\wbar{\Gamma}$ is commensurable with 
${\Gamma}$.
\end{thm}
Any square complex studied in \cite{WisePolygons} satisfies
the above assumptions and its fundamental
group is residually finite (in fact every
quasiconvex subgroup is separable). In
particular any extension of such a fundamental
group is also residually finite.

\section{Outline of paper}

 The paper is organized in three parts. Theses part are independent,
except that we will use in the last part some (classical) definitions
introduced in the previous parts.

 In the first part we deal with (regular) right-angled buildings. We
identify buildings with their Davis--Moussong realization. In \fullref{sec:cubcompl} we recall
basic definitions and facts on cube complexes. This is mostly dedicated to the reader more
familiar with chamber systems and ``abstract'' buildings than with
$\CAT(0)$ cube complexes.

 In \fullref{sec:building} we define the right-angled building of a
graph product as a (typed) $\CAT(0)$ cube complex, on which the graph
product acts as a uniform lattice. We also introduce residues, a kind
of typed subcomplex. We are particularly interested in a specific 
kind of residues that will play the role of ``walls''. We establish a
product structure for these wall-residues, where the first factor is
compact. All the results in this section are rather classical, especially
for the reader used to the langauge of chamber systems.

 In \fullref{sec:holonomy} we define the holonomy of a uniform
lattice at a wall-residue to be the projection of the residue
stabilizer on the automorphism group of the compact factor.  Thus it is a finite
subgroup. We note that a naturally defined finite index subgroup of the
graph product has trivial holonomy.

We define atlases on our building $\Delta$ in \fullref{sec:atlas}.
We prove that the automorphism group of an atlas is discrete cocompact,
and that any two atlases are conjugate in ${\rm Aut}_0(\Delta)$. We also
note that the graph product preserves a naturally defined  atlas. So we
have a strategy for commensurating a uniform lattice $\Gamma$ with the
graph product: find an atlas on the building which is
invariant under some finite index subgroup of $\Gamma$. Since in fact
groups without holonomy always preserve an atlas, it remains to find a
finite index subgroup without holonomy. We do this in \fullref{sec:killhol} using the separability of the
wall-residues stabilizers and their finite index subgroups.

 In the second part we study extensions by finite groups. We first
introduce many basic definitions on polygonal complexes
(\fullref{sec:polycompl}), which will be used also in the third part.

 Then we define the walls of a polygonal complex
(\fullref{sec:walls}). When the number of sides of each polygon is
even, our walls are essentially obtained as unions of segments joining
inside a polygon the midpoint of an edge to the midpoint of the opposite
edge. We check that these walls correspond to totally geodesic subtrees
when some nonpositive curvature condition is fulfilled.

 In \fullref{sec:killcocycle} we consider a 2--cocycle (with values in
a finite abelian group) on a compact polygonal complex satisfying our
particular nonpositive curvature condition. We show that if the
fundamental group has separable quasiconvex subgroups then there is a
finite cover on which the lifted 2--cocycle vanishes. This proves
\fullref{thm:virttrivext}.

 In the last part we study commensurability with Coxeter groups.  We use
definitions given in the second part and adapt the argument there
to the noncommutative context.

In \fullref{sec:coxdavcomplex} we recall the definition and the
properties of the Davis--Moussong realization of a Coxeter group. We will
work with two-dimensional Coxeter groups, so our complexes are polygonal,
and the Coxeter group acts on it as a uniform lattice. 
We introduce a new type of walls in polygonal
complexes, which we call $e$--walls, in \fullref{sec:ewalls}.

In \fullref{sec:localrefl} we define systems of local reflections and
their holonomy. Here, a system of local reflection without holonomy plays
the same role as an atlas on a regular right-angled building. We note
that any two systems of local reflections without holonomy are conjugate
in the full automorphism group, and the automorphism group of any such
system is a finite extension of a conjugate of the initial Coxeter group.
So in order to commensurate a uniform lattice $\Gamma$ with our initial
Coxeter group, it is enough to produce a system of local reflections
without holonomy which is preserved by a finite index subgroup of
$\Gamma$.  This is done in \fullref{sec:killholcox}, using a nonpositive
curvature assumption and the additional hypothesis that $e$--walls
stabilizers in $\Gamma$ and  their finite index
subgroups are separable.

In the three situations we deal with, we note that the obstructions for
commensurability lie along wall-like convex subcomplexes. Under strong
enough separability  assumptions it is then possible to kill these
obstructions in a finite index subgroup, thus getting the commensurability
result.

The results in this paper raise some questions.
\begin{prob}
Let $1\to G\to \overline\Gamma\to\Gamma\to 1$ denote a central extension
of $\Gamma$, a uniform lattice of a (Gromov-hyperbolic) $\CAT(0)$ cube
complex of dimension $n$ at least $3$. Does the separability of quasiconvex subgroups of $\Gamma$ imply
the virtual triviality of the extension?
\end{prob}
Note that our answer is for $n=2$. In higher dimensions, our method does
not apply without changes. Indeed we are not able to kill \textit{naturally\/}
a 2--cocycle along the walls, even on a single $3$--cube.
   
The same question arises for commensurability of uniform lattices with Coxeter groups inside Davis complexes of higher
dimension:
\begin{prob}
Let $(W,S)$ denote some right-angled Coxeter system, that is, a graph product of order two groups along some graph $\mathcal G$. Assume that the Davis complex $X$ of $(W,S)$ has dimension
$n\ge 3$, so that $\mathcal G$ contains a complete graph on $4$ vertices. Does the separability of quasiconvex subgroups of a uniform lattice $\Gamma$ imply
that $\Gamma$ is commensurable with $W$?
\end{prob}

\begin{prob}
Find a Gromov-hyperbolic $\CAT(0)$ square complex that admits two
noncommensurable uniform lattices.
\end{prob}
Of course it would be even more interesting if the square complex was the
Davis--Moussong realization of a right-angled Coxeter group. Note that the product of  two (regular nonelementary) trees
admits noncommensurable uniform lattices. That is why we insist here on Gromov-hyperbolicity.

With Wise \cite{HaglundWise05}, we introduced a notion of \textit{special\/} uniform
lattice $\Gamma$ of a $\CAT(0)$ cube complex $X$: we say that $\Gamma$ is
special if there exists an injective morphism from $\Gamma$ into a finitely generated right-angled Coxeter group $W$, and an equivariant
isomorphism of
$X$ onto a convex subcomplex of the Davis realization of $W$. 

\begin{prob}
Are two special uniform lattices of a locally compact $\CAT(0)$ cube
complex always commensurable ?
\end{prob}

For example, a uniform lattice of  a product of two trees is (virtually) special exactly when it is virtually a product of uniform lattices of the trees. So in this case all special uniform lattices are commensurable.

\part{I\qua Commensurability of uniform lattices in right-angled buildings}
\section{Cube complexes} \label{sec:cubcompl}

 In this section we recall basic definitions and facts about cube
complexes (see also \cite[page 111]{BridsonHaefliger}).

\begin{defn}\label{defn:cub}
 A \textit{cube complex of dimension $0$\/} is a discrete, nonempty set. Any
map $f\co X\to Y$ between cube complexes of dimension $ 0$ is said to be
\textit{combinatorial\/}. If $C$ is a singleton and $X$ is a cube complex of
dimension $ 0$,  then any map $f\co C\to X$ is called a \textit{$0$--cube of
$X$}. Observe that the boundary of any unit interval is a cube complex of
dimension $0$.

 Assume that cube complexes of dimension $m\le  n$, combinatorial maps between such spaces, and also
$k$--cubes of such a space for $k\le m$ have been defined. Then a cube complex of
dimension $ n+1$ is obtained by gluing $(n{+}1)$--cubes to some cube
complex  of dimension $  n$. More precisely, let $X^n$ denote some cube
complex of dimension $  n$, and let $(f_i)_{i\in I}$ denote some (nonempty) family of combinatorial maps $f_i\co \partial C_i\to X^n$, where
$C_i$ is some Euclidean cube of dimension $n+1$ whose edges have length
1 and $\partial C_i$ denotes the combinatorial boundary of $C_i$ endowed
with its natural cube complex structure of dimension $  n$.  In fact we
assume by induction that a cube complex structure of dimension $n$  is
defined on the union $C^n$ of faces of dimension $n$ of any unit
Euclidean cube $C$ of dimension $p\ge n$, in such a way that for any
$p$--face $D$ of a unit Euclidean cube $C$ with $p\ge n$, the natural
inclusion $D^n\to C^n$ is combinatorial.

 We obtain a cube complex $X$ of dimension $n+1$ if we glue each cube
$C_i$ to $X^n$ along $\partial C_i$ using the maps $f_i$. Then $X^n$ is
naturally a subspace of $X$, called its \textit{$n$--skeleton\/}. Observe that
the maps $f_i\co \partial C_i\to X^n$ naturally extend to maps $f_i\co C_i\to
X$. These maps are the \textit{$(n{+}1)$--cubes of $X$\/} for $k\le n$ the \textit{$k$--cubes of $X$\/} are the $k$--cubes of $X^n$.  Now we define a map
$f\co X\to Y$ between  cube complexes of dimension $  n+1$ to be \textit{combinatorial\/} whenever it maps $X^n$ into $Y^n$, the induced map between
the $n$--skeleta is combinatorial and for each $(n{+}1)$--cube $f_i\co C_i\to X$
there exists an $(n{+}1)$--cube $g_j\co D_j\to Y$ and an isometry $h\co C_i\to D_j$
such that $ff_i=g_jh$.

 For any topological space $X'$, a cube complex structure of dimension
$n+1$ is given  by a homeomorphism $h\co X\to X'$, where $X$ is a cube complex constructed as
above. Two homeomorphisms $h_1,h_2\co X\to X'_1,X'_2$ define the same
structure whenever $h_2\circ {h_1}^{-1}$ is combinatorial. Pushing by $h$
defines the $n$--skeleton of $X'$ and the $(n{+}1)$--cubes of $X'$. 
Combinatorial maps between topological spaces equipped with a cube complex structure of dimension $n+1$ are similarly defined. The open cubes of
$X'$ are the image of the interior of $C$ by some cube $C\to X'$. Clearly
$X'$ is the disjoint union of its open cubes: every point in $X'$ is
contained in the interior of a unique cube.

 It remains to define a natural cube complex structure of dimension
$n+1$  on the union $C^{n+1}$ of faces of dimension $n+1$ of any unit
Euclidean cube $C$ of dimension at least $n+1$. Let $C^{n}$ denote the union
of the $n$--faces of $C$. By induction this is a cube complex of
dimension $n$. For any $(n+1)$--face $F$ of $C$, we have $\partial F=F^n$
and we know by induction that the inclusion $F\to C$ induces a
combinatorial map $F^n\to C^n$. The cube complex obtained by gluing all
$(n+1)$--faces to the $n$--skeleton $C^n$ has a natural homeomorphism to
$C^{n+1}$. The naturality of this structure is straightforward.

 We give the usual names to low dimensional cubes: $0$--cubes are
vertices, $1$--cubes are edges, $2$--cubes are squares. We sometimes
identify $k$--cubes with their range (note that distinct cubes have
distinct ranges). Thus the $0$--skeleton may be viewed as the set of
vertices.

 A subcomplex of a cube complex  is a union of cubes. It inherits
naturally a cube complex structure for which the inclusion is
combinatorial.
\end{defn}

 \begin{defn}\label{defn:simplicial} A \textit{multisimplicial complex\/} $X$, a \textit{$k$--simplex\/}  of $X$,  
 a \textit{combinatorial map\/} $X\to Y$ between multisimplicial complexes and  
 a \textit{subcomplex\/} are obtained from \fullref{defn:cub} by replacing
unit Euclidean cubes by affine simplices (isometries of cubes have to be
replaced by affine isomorphism of simplices).  Observe that a
combinatorial map preserves the dimension of simplices.

 A \textit{simplicial complex\/} is a multisimplicial complex  $X$ such that
each simplex is injective and two distinct $k$--simplices  $f,f'\co \Sigma^k
\to X$ have distinct boundaries, ie, $f(\partial \Sigma^k)\neq
f'(\partial \Sigma^k)$. It is easy to check that these simplicial
complexes are in one-to-one correspondence with abstract simplicial
complexes, ie, collections of nonempty subsets, called simplices, of a
given set $I$ such that if $\sigma$ is a
simplex and if $\tau $ is a nonempty subset of $\sigma$, then $\tau$ is a
simplex too. For any simplicial complex $N$ we will denote by
$\wbar N$ the poset ordered by inclusion whose elements are the simplices of $N$, together
with the empty set.

The \textit{join\/} of two (abstract) simplicial complexes $X,Y$ is the
simplicial complex $X*Y$ with vertex set $X^0\sqcup Y^0$ and with
simplices the subsets $\sigma\cup \tau$, for $\sigma\in \wwbar X,\tau\in
\wwbar Y$. If
$\sigma,\tau$ are two simplices of a simplicial complex $X$ such that
$\sigma\cup \tau$ is a simplex, we say that $\sigma$ and $\tau$ are
\textit{joinable} in $X$.
The join is commutative and associative.

 A simplicial complex $L$ is said to be a \textit{flag\/} if any complete
subgraph of the $1$--skeleton is the 1--skeleton of a unique simplex of $
L$ (a complete subgraph of $L^1$ is a subcomplex $K$ such that any two
vertices of $K$ are contained in an edge of $K$).  
\end{defn}
For a reference on simplicial complexes see
Rourke and Sanderson \cite[Chapter 2]{RourkeSanderson}.

\begin{exmp}[Thickened octahedron]\label{exmp:thickocta}
 A \textit{thickened octahedron\/}  is a simplicial complex $L$ obtained by
iterated joins with discrete sets.

More precisely let $L_1,\cdots,L_n$ denote simplicial complexes of
dimension 0. Then $L=L_1*\cdots*L_n$ is a thickened octahedron.

Assume that a group $G$ acts simplicially on a simplicial complex $L$ and that there is a simplex $\sigma=\{x_0,x_1,\dots,x_n\}$ of $L$ such that  $\sigma$ is a strict fundamental domain for the action of $G$ on $L$, ie, the orbit of any simplex of $L$ under $G$ contains exactly one simplex of $\sigma$. Let $G_i$ denote the stabilizer of $\sigma_i=\{x_0,x_1,\dots,\hat{x_i},\dots,x_n\}$. Then $L$ is a thickened octahedron provided $G$ is generated by $\{G_i\}_{i\in I}$ and $[G_i,G_j]=1$ for $i\neq j$.

To check this, let $L(x_0)$ denote the subcomplex of $L$ consisting in those simplices not containing $x_0$ but joinable with $\{x_0\}$. Observe that the stabilizer $G(x_0)$ of $x_0$ in $G$  acts on $L(x_0)$ with $\sigma_0$ as a strict fundamental domain. The stabilizers of the codimension--one faces of $\sigma_0$ are precisely $G_1,\dots,G_n$, which commute by assumption. We claim that $G(x_0)$ is generated by $G_1,\dots,G_n$. Indeed for any $g\in G$ we have $g=g_0h=hg_0$ where $g_0\in \<G_1,\dots,G_n\>$ and $h\in G_0$. 
It follows that $g(\sigma)=g(x_0*\sigma_0)=k(x_0)*g_0(\sigma_0)$. 
Thus if $g\in G(x_0)$ then $k(x_0)=x_0$, whence $k\in\cap_{i=0}^{i=n}G_i$ and $g\in \<G_1,\dots,G_n\>$. 
Since $\sigma$ is a fundamental domain, the formula $g(\sigma)=k(x_0)*g_0(\sigma_0)$ also shows that $L$ is the join of $L(x_0)$ and the set of vertices joinable with $\sigma_0$. 
We conclude that $L$ is a thickened octahedron by induction on the dimension of $\sigma$.
\end{exmp}

 \begin{defn}[\rm Links]\label{defn:link} Let $C$ be a Euclidean cube
and let $p$ be one of its vertices. We define ${\rm link}(p,C)$ as the
(codimension 1) simplex generated by the midpoints of edges of $C$
containing $p$. Note that if $D$ is a strict face of $C$ containing $p$
then ${\rm link}(p,D)\subset\partial\,{\rm link}(p,C)$.

Let $X$ be a cube complex and let $v$ be a vertex of $X$.
For each  based $(k{+}1)$--cube $f\co (C,p)\to (X,v)$ we obtain by restriction a
map $\sigma_f\co {\rm link}(p,C)\to X$.  The \textit{link of $X$ at $v$\/} is
the multisimplicial complex ${\rm link}(v,X) \subset X$ such that
each $k$--simplex is $\sigma_f$ for some $f$. We say that $X$ is locally finite if
for each vertex $v$, the complex ${\rm link}(v,X)$ has finitely many vertices.
 
  We define similarly the link of a vertex in a multisimplicial complex.
The link of a vertex in a simplicial complex is still simplicial.

  We say that $X$ is a \textit{simple\/} cube complex if for each vertex $v$
the complex ${\rm link}(v,X)$ is a simplicial complex. This amounts to demanding
that the cubes of $X$ are locally injective, and any combinatorial map to
$X$ defined on the union of faces containing a given vertex of a cube
extends to at most one combinatorial map on the whole cube.
\end{defn}

 The following construction appears in Davis, Januszkiewicz and Scott
\cite[1.2 page 503]{DavisJanusScott98}.

\begin{lem}\label{lem:cubcone}
Let $N$ denote any simplicial complex.  Then there exists a cubical
complex $C(N)$ with a distinguished vertex $v$, called the {\rm center of
$C(N)$}, such that ${\rm link}(v,C(N))$ is isomorphic to $N$, and for every
cube $C$ of $C(N)$ there exists a unique minimal cube $Q(C)$ containing
$v\cup C$. Moreover such a cubical complex is unique up to combinatorial
isomorphism preserving the centers. The complex $C(N)$ is called the
{\rm cubical cone on $N$}. For every  combinatorial map $f\co N\to M$ there
is a unique combinatorial map $c(f)\co C(N)\to C(M)$ inducing $f$ on the
link of the center of $C(N)$. We will denote by $\partial C(N)$ the union
of cubes not containing $v$.

For any cube  $C$ of $C(N)$, $C\neq\{v\}$, the cube $Q(C)$ defines a
certain simplex $t(C)$ in ${\rm link}(v,C(N)) = N$. Thus we get a map $t$
assigning to each vertex $x$ of $C(N)$, $x\neq v$, the simplex $t(\{x\})$
of $N$. We extend the definition of $t$ by declaring $t(v)=\emptyset$.
On the set of vertices of each cube of $C(N)$ the map $t$ is a
bijection onto some interval of the poset $\wbar N$.

The greatest element of this interval is $t(C)$. We will denote by 
$\underline t(C)$ the least element of  this interval. For each
$p\in C(N)$ let $C(p)$ denote the unique cube of $C(N)$
containing $p$ in its interior. Then we define $t(p)$ and $\underline
t(p)$ as $t(C(p))$ and $
\underline t(C(p))$. The map $t$ is the \textit{typing map on $C(N)$\/}, and $t(p)$ is \textit{the type of the point $p$\/}.

 The map $t$
induces a bijection from   the set of vertices of the cubical cone $C(N)$
onto
$\wbar N$.

\end{lem}

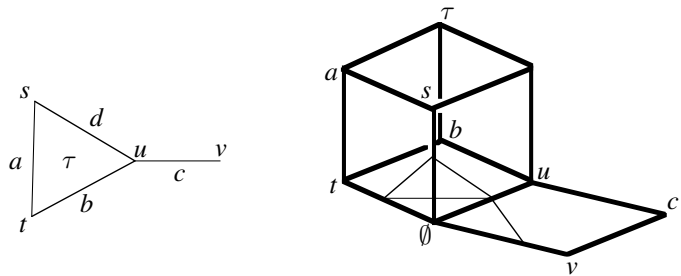
\begin{figure}[ht!]
\labellist
\small\hair 2pt
\pinlabel $t$ at 117 258
\pinlabel $s$ at 119 390
\pinlabel $a$ at 110 320
\pinlabel $\tau$ at 163 320
\pinlabel $d$ at 190 365
\pinlabel $b$ at 180 280
\pinlabel $u$ at 233 332
\pinlabel $c$ at 272 308
\pinlabel $v$ at 313 332
\pinlabel $a$ at 422 408
\pinlabel $t$ at 424 297
\pinlabel $\tau$ at 537 468
\pinlabel $s$ at 516 390
\pinlabel $b$ at 545 353
\pinlabel $\emptyset$ at 515 248
\pinlabel $u$ at 632 310
\pinlabel $v$ at 660 215
\pinlabel $c$ at 760 280
\endlabellist
\centering
\includegraphics[width=12cm]{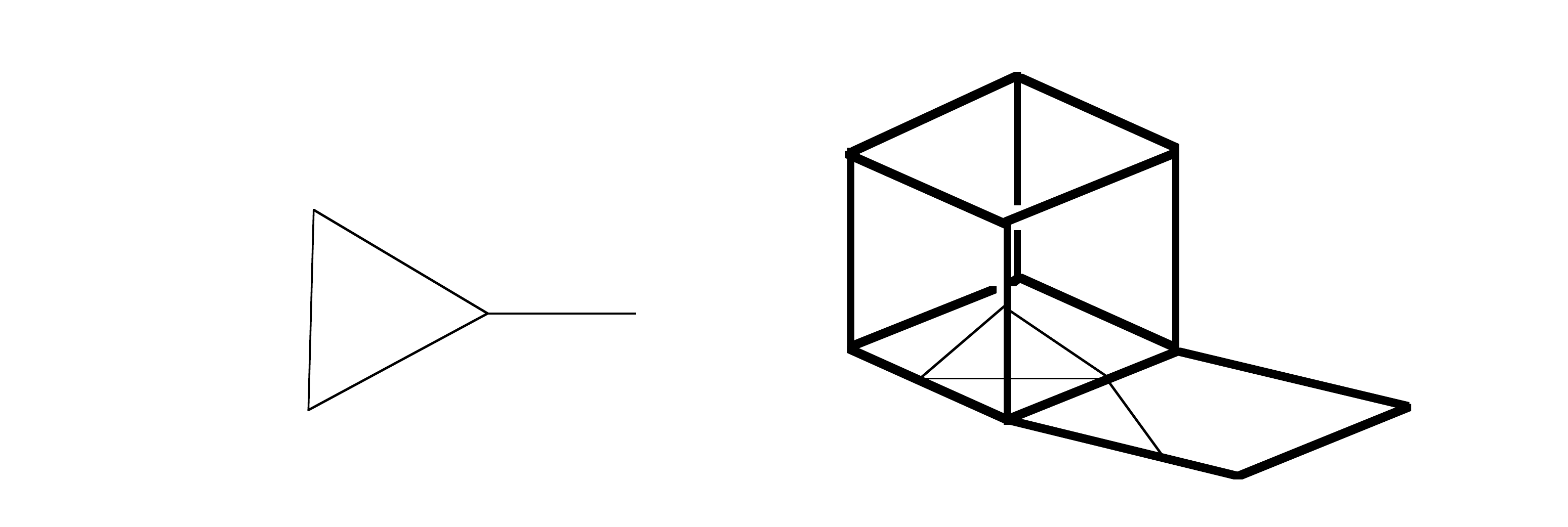}
\caption{$N$ with the cubical cone on $N$ and typing map. On the left,
$N$ has four vertices $s, t, u, v$, four edges $a, b, c, d$ and one triangle $\tau$.}
\end{figure}

\begin{proof}[Proof: Uniqueness]
Let $(C,v)$ be a cube complex as in the lemma, and let
$(C',v')$ denote any other based cube complex endowed with an
isomorphism $\varphi$ from ${\rm link}(v,C)$ to ${\rm link}(v',C')$. For any cube
$Q$ of $C$ containing $v$ there is a (unique) combinatorial map $f_Q\co Q\to
C'$ inducing $\varphi$ on ${\rm link}(v,Q)$. For any cube $\xi$ of $C$ we
define $f_{\xi}$ as the restriction of $f_{Q(\xi)}$ to $\xi$. These maps
fit together and produce a based combinatorial map $f\co (C,v)\to (C',v')$
inducing $\varphi$ on ${\rm link}(v,C)$. The uniqueness of such a map is
straightforward. The universal property we have just described
implies as usual the uniqueness of the based cube complex $(C,v)$ up to
isomorphism.

{\bf Existence}\qua In the Hilbert space $\ell^2(N^0)$ consider the union of
cubes whose vertices are vectors of the form: $\sum_{v\in \sigma}x_ve_v$,
where $(e_v)_{v\in N^0}$ denotes the canonical Hilbert basis, $x_v=0$ or
$1$, and $\sigma$ denotes some simplex of $N$.

If $f\co N\to M$ is a combinatorial map, it induces a linear map
$\vec f\co \ell^2(N^0)\to \ell^2(M^0)$.  Restricting this to $C(N)$ yields a combinatorial map 
$c(f)\co C(N)\to C(M)$ with the desired property. Two combinatorial maps $C(N)\to
C(M)$ which induce $f$ on ${\rm link}(v,C(N))$ have to agree on each cube
containing $v$. Uniqueness then follows because each cube of $C(N)$ is
contained in a cube containing $v$.

{\bf Properties of the map $t$}\qua Let $Q$ be a cube containing $v$. A
vertex $x$ of $C$ is opposite to $v$ in $Q$ if and only if $Q(\{x\})=Q$.
So for the vertex $v_Q$ opposite to $v$ in $Q$ we have $t(\{v_Q\})= {\rm
link}(v,Q)$. For any cube $C$ of $C(N)$ we set $v_C=v_{Q(C)}$. Note that
$v_C$ is a vertex of $C$. Finally define $w_C$ to be the vertex of $C$
opposite to $v_C$ in $C$. Observe  that the cube $Q(C)$ is the product of
$C$ and a unique face $q(C)$ containing $v$ and meeting $C$ at the single
vertex $w_C$. For example, $w_C=v$ if and only if $v\in C$.
Then for any vertex $x$ of $C$, the smallest cube containing 
$\{v\}\cup\{x\}$ in fact contains $w_C$ and is contained in
$Q(C)=Q(\{v_C\})$. Thus we get $t(q(C))\subset t(\{x\})\subset t(Q(C))$.
Clearly $t$ is injective on the set of $0$--cubes. So $t$ induces a
bijection between the set of vertices of $C$ and the interval 
$\{\tau\in
\wbar N, t(q(C))\subset \tau \subset t(Q(C))\}$. Since $t\co C^0\to
\wbar N$ is clearly onto it is in fact a bijection.
\end{proof}

\begin{lem}\label{lem:joinproduct}
The cubical cone of a join is the product of the cubical cones.

More precisely let  $N_1,N_2$ be two abstract simplicial complexes,
with vertex sets $I_1,I_2$. Let $N=N_1*N_2$ denote the join of $N_1$ and
$N_2$.

Then there is a  combinatorial isomorphism $j\co C(N_1)\times C(N_2)\to
C(N)$ such that $t(j(p_1,p_2))=t_1(p_1)*
t_2(p_2)\subset I=I_1\sqcup I_2$. Such an isomorphism is unique.
\end{lem}

\begin{proof}
For existence, consider the natural isometry 
$\ell^2(I_1)\times \ell^2(I_2) \to \ell^2(I)$ and restrict it to
$C(N_1)\times C(N_2)$. The type-preserving condition implies uniqueness.
\end{proof}

\begin{lem}[Links in the cone]\label{lem:linkscone}
Let $N$ be a simplicial complex, $w$ a vertex of the cubical cone $C(N)$, and
$t(w)=J\subset \wbar N$ the type of $w$.

Then a cube $C$ of $C(N)$ contains $w$ if and only if $t(w)\subset t(C)$.

In particular the link of $w$ in $C(N)$ is isomorphic to the union of
simplices of
$N$ containing $t(w)$ under the map assigning to each simplex $\sigma$ of
${\rm link}(w,C(N))$ the simplex $t(C(\sigma))$, where $C(\sigma)$
is the cube of $C(N)$ containing $w$ and corresponding to $\sigma$.
\end{lem}

\begin{proof}
Indeed $w\in C\iff Q(\{w\})\subset Q(C)$.
\end{proof}

 Note that cubical cones on the boundary of simplices are the building
blocks of a subdivision of simplicial complexes into cubical complexes.
Thus topologically, cube complexes and simplicial complexes are the same.

The specificity of cube complexes appears as soon as geometry is involved.
According to Gromov \cite{Gromov87},
there are very simple local conditions on a cube complex which make it  
a nonpositively curved space in the sense of Aleksandrov:
\begin{thm}\label{thm:notriangle}
Let $X$ denote any
locally finite cube complex. Endow $X$ with the length metric inducing on each cube 
the unit Euclidean metric. Then $X$ is locally $\CAT(0)$ if and only if
each vertex link of $X$ is a flag simplicial complex. 
\end{thm}

This justifies the following:

\begin{defn}\label{defn:NPC}
  We say that a cube complex $X$ is  \textit{nonpositively curved\/} if for
each vertex $v$ the complex ${\rm link}(v,X)$ is a flag simplicial complex. We
say that $X$ is $\CAT(0)$ whenever it is simply connected and nonpositively curved.
\end{defn}

For references on the $\CAT(0)$ inequality, see Bridson and Haefliger \cite{BridsonHaefliger}.
For local combinatorial conditions on simplicial complexes implying a nonpositively curved behaviour, see \cite{HaglundHighDimHyper} and
\cite{JanuszkiewiczSwiatkowskiSNPC}.

\section[The right-angled building of a graph product of finite groups]{The right-angled building of a graph product\\ of finite groups}\label{sec:building}

All the material in this section is well-known to people working on
graph products or right-angled buildings. For convenience we have written
explicitly most of the arguments.

\subsection{Graph products}

\begin{defn}[Graph products]\label{defn: graphproduct}
 Let $\mathcal G$ denote any simplicial graph with vertex set $I$.
Suppose that for each $i\in I$ we are given a group $G_i$. We define the
\textit{graph product\/} of $\{G_i\}_{i\in I}$ along $\mathcal G$ to be the group
$\Gamma=\Gamma({\mathcal G},(G_i)_{i\in I})$ defined as follows.

 $\Gamma$ is the quotient of the free product $\ast\,G_i$ by the
normal subgroup generated by all elements of the form
$g_ig_jg_i^{-1}g_j^{-1}$, where $g_i\in G_i,g_j\in G_j$ for distinct adjacent vertices $i,j$ in
$\mathcal G$.

 For any subset $J\subset I$ we denote by $\Gamma_J$ the subgroup of
$\Gamma$ generated by the image of $\{G_i\}_{i\in I}$ under the
natural morphism $G_j\to \Gamma$, and we denote by ${\mathcal G}_J$ the
graph on $J$ induced by $\mathcal G$.

 A subset $J\subset I$ is said to be \textit{spherical\/} whenever any two
distinct vertices of $J$ are adjacent in $\mathcal G$. In the sequel we
denote by $N({\mathcal G})$ the abstract simplicial complex on $I$ whose
simplices are the nonempty spherical subsets $J\subset I$.

 For any element $i\in I$ we will denote by $i^\perp$ the set of $j\in
I\setminus\{i\!\}$ such that $j$ is adjacent to $i$ in $\mathcal G$. We also
consider the subset $i\perpeq=i^\perp\cup\{i\!\}$.
\end{defn}

 For example if $\mathcal G$ is a complete graph then $\Gamma$ is the
direct product of the groups $G_i$, $i\in I$. In contrast, if
$\mathcal G$ is totally disconnected then  $\Gamma$ is the free product
$\ast_{i\in I}\,G_i$.

Some references on graph products (and their buildings) are Green
\cite{GreenGraphProd}, Januszkiewicz and {\'S}wi{\c{a}}tkowski \cite{JanuszkiewiczSwiatkowskiGraphProducts},
and Meier \cite{Meier96}.

In the rest of the first part we assume we are given a graph
$\mathcal G$ with vertex set $I$, together with a group $G_i$ for each
$i\in I$. We denote by $\Gamma$ the graph product of $\{G_i\}_{i\in I}$ along
$\mathcal G$.

The following result is proved in
\cite{JanuszkiewiczSwiatkowskiGraphProducts}, for example.

\begin{lem}[Naturality]\label{lem:naturality}
 For any $i\in I$ the natural map $G_i\to \Gamma$ is injective. So we
will identify $G_i$ and $\Gamma_i$. More generally for any subset
$J\subset I$ the natural map from the free product $\ast_{j\in J}\, G_j$ to the group $\Gamma$ induces an isomorphism $\Gamma({\mathcal
G}_J,(G_i)_{i\in J})\to \Gamma_J$.
\end{lem}

\begin{proof}
 Consider the morphism from the free product $\ast_{i\in I}\, G_i$
onto the free product $\ast_{j\in J}\, G_j$  that kills each $G_k$,
$k\not\in J$, and induces the identity on each $G_j$, $j\in J$. This
morphism induces a morphism $\rho_J$ from $\Gamma({\mathcal G},(G_i)_{i\in
I})$ to $\Gamma({\mathcal G}_J,(G_i)_{i\in J})$. Now the natural map
$\Gamma({\mathcal G}_J,(G_i)_{i\in J})\to\Gamma({\mathcal G},(G_i)_{i\in
I})$ postcomposed with $\rho_J$ is the identity.
 \end{proof}

\subsection{The building of a graph product}

 \begin{defn}\label{defn:cubgraph}
 Let $C=C(N({\mathcal G}))$ denote the cubical cone over $N({\mathcal
G})$ (see \fullref{lem:cubcone}). Consider the equivalence
relation on $\Gamma \times C$ defined by 
$$(\gamma,p)\sim (\gamma',p')
\iff p=p'\in\partial C \text{ and } \gamma^{-1}\gamma'\in \Gamma_J for
J=\underline t(p).
$$ Set $\Delta=\Delta({\mathcal G},(G_i)_{i\in I}) =
{\Gamma}\times C /\sim$, and denote by $[\gamma,p]$ the equivalence
class of $(\gamma,p)$.

For $\gamma \in \Gamma$, the image under the natural map $\pi\co {\Gamma}\times C\to \Delta$ of the
$\{\gamma\}\times C$  is called a \textit{chamber of $\Delta$\/}, denoted by $C_{\gamma}$.  We 
denote the base chamber $C_{1_{\Gamma}}$ by $C_*$.
\end{defn}

\begin{lem} Let $C$ be as in \fullref{defn:cubgraph}.  Then the following observations easily hold.
\begin{enumerate}
\item The left action of $\Gamma$ on ${\Gamma}\times C$ induces an
action by homeomorphisms on $\Delta$.

\item The second projection ${\Gamma}\times C\to \{1_{\Gamma}\}\times C$
induces a topological retraction map $\rho\co \Delta\to C_*$. The natural
map $C\to C_*$ given by the inclusion $C\to \{1_{\Gamma}\}\times C$ postcomposed
with $\pi$ is a homeomorphism. Thus we may identify $C$ with $C_*$.

\item The map $\rho$ is invariant under $\Gamma$.
 In fact two points of $\Delta$ belong to the same $\Gamma$--orbit if and only if they have the same image under $\rho$.

\item The base chamber $C_*$ is a strict fundamental domain for the action of $\Gamma$ on $\Delta$. So every $\Gamma$--orbit in $\Delta$ meets $C_*$ at a one and only one point.
\end{enumerate}
\end{lem}
The following is straightforward.

\begin{lem}
Each building $\Delta$ admits a unique cube complex structure for which the natural map $\pi\co {\Gamma}\times C\to \Delta$  is combinatorial.
\end{lem}

We will always endow $\Delta$ with this structure. 

\begin{rem} Note that (1) chambers are subcomplexes, (2) $\Gamma$ acts as a group of combinatorial automorphisms and 
(3) the retraction $\rho\co \Delta\to C$ is combinatorial.
\end{rem}

\subsection{The typing map}

\begin{defn}\label{defn:typ}
 Recall from \fullref{lem:cubcone} that we have already defined a map
$t$ assigning to each vertex of  the cubical cone $C=C(N({\mathcal G}))$
an element of $\overline{N({\mathcal G})}$, ie,  a simplex of
$N({\mathcal G})$ or the empty subset of $I$. This map is also defined
on the set of vertices of ${\Gamma}\times C$ and induces a map
$t\co \Delta^0\to \overline{N({\mathcal G})}$, called the \textit{typing map}
on $\Delta$.  On the set of vertices of the  base chamber $C_*$ identified
with $C$, the map $t$ is the usual one. We define similarly for any cube
$Q$ of $\Delta$ the elements $t(Q)$ or $\underline t(Q)$ by first taking
the image $\rho(Q)$ inside $C_*\simeq C$ and then applying the usual maps
$t$ or $\underline t$. Equivalently $t(Q)$ is the maximum of all $t(v)$
for $v$ a vertex of $Q$ and $\underline t (Q)$ is the minimum of  all
$t(v)$ for $v$ a vertex of $Q$.

 The \textit{rank of a vertex $v$\/} is the cardinality ${\rm rk}(v)$ of 
$t(v)$. The rank--0 vertices will be called \textit{centers of chambers\/}. We
will denote the center of $C_*$ by $v_*$.

 We will denote by ${\rm Aut}_0(\Delta)$ the subgroup of the group of
combinatorial automorphisms of $\Delta$ whose elements preserve the
typing map $t$. Note that $t$ is invariant under $\Gamma$, ie, $\Gamma\subset {\rm Aut}_0(\Delta)$) and $\rho$.
\end{defn}

\begin{lem}\label{lem:stabtyp}
 For any $J\in \overline{N({\mathcal G})} $,  each chamber contains a
single vertex $v$ of type $t(v)=J$. The stabilizer of the unique vertex
of type $J$ in $C_*$  is the  group $\Gamma_J$. The stabilizer of a vertex
acts simply transitively on the set of chambers containing this vertex.

 In particular
$\Gamma$ is simply transitive on centers of chambers and on chambers.
\end{lem}
This implies that the action of $\Gamma$ on $\Delta$ is faithful. Hence
we will identify $\Gamma$ with its image in ${\rm Aut}_0(\Delta)$.

\begin{proof}
 Using the $\Gamma$ invariance of $t$, it suffices to prove the first
assertion for the base chamber $C_*$.

Let $v_J$ denote the unique vertex of $C$ with $t(v_J)=J$ (see
\fullref{lem:cubcone}). Then
$[1,v_J]$ is the unique vertex of $C_*$ with type $J$. Now
$\gamma[1,v_J]=[1,v_J]$ if and only if $\gamma\in \Gamma_J$, by definition of $\Delta$.

For $J=\emptyset$, we have $\Gamma_J=\{1\}$,
thus $\Gamma$ acts freely on centers of chambers (by construction it acts
transitively).

In particular if $\gamma C_*=C_*$ then $\gamma$ fixes the center of $C_*$
(we have just seen that there is a unique rank--0 vertex in each chamber).
Thus $\gamma=1$. The simple transitivity of $\Gamma$ on the set of
chambers follows.

 Let us check now that $\Gamma_J$ acts simply
transitively on the set of chambers containing $v_J$. The action is free
because $\Gamma$ acts freely on the set of chambers. Let $C=\gamma C_*$
denote a chamber containing $v_J$. Then $\gamma v_J$ is a vertex of type
$J$ inside $C$. By the first part of the lemma, we have $\gamma v_J=v_J$.
Thus $\gamma\in \Gamma_J$ as claimed.
\end{proof}

For each vertex $v$ of $\Delta$ we denote by ${\mathcal C}(v)$ the set of
chambers of $\Delta$ containing $v$. This set is in one-to-one
correspondence with the stabilizer of $v$ in $\Gamma$ by the previous
lemma. As a consequence we obtain the following:

\begin{cor}\label{cor:chamberneighbor} For each vertex $v$ of $\Delta$ we
have ${\rm link } (v,\Delta)=\bigcup_{C\in{\mathcal C}(v)}{\rm link }
(v,C)$.

In particular,  the building
$\Delta$ is locally compact if and only if the graph $\mathcal G$ is
finite and each generating group $G_i$ is finite.
\end{cor}

\subsection{Residues}

\begin{defn}[\rm Chamber system and residues]\label{defn:res}
 Let  $\mathcal C$ denote the set of chambers of $\Delta\!$. Note that 
$\mathcal C\!$ is identified under $\gamma\!\mapsto \!C_{\gamma}$ with $\Gamma$
(see \fullref{lem:stabtyp}). We follow Ronan
\cite{Ronan89} to define on the set
$\mathcal C$ a structure of \textit{chamber system over $I$\/}. For each $i\in I$ we say that two chambers $C_1,C_2$ of
$\Delta$ are
$i$--adjacent (written
$C_1\sim_i C_2$) whenever  the intersection $C_1\cap C_2$ contains a
vertex $v$ with
$t(v)=\{i\!\}$.

We then define for any subset $J\subset I$ the equivalence relation
$\sim _J$ on $\mathcal C$ as the equivalence relation generated by 
$\sim_j$, for $j\in J$ (by convention $\sim_{\emptyset}$ is the equivalence
relation whose orbits are the singletons). We define \textit{the extended
$J$--residue of a chamber
$C$} to be the union of the chambers $C'$ for which $C'\sim_J C$. We will
denote this subcomplex by $\bar{R}(J,C)$ (or sometimes
$\bar{R}_{\Delta}(J,C)$). For example $\bar{R}(\emptyset,C) =C$.

Now we define \textit{the $J$--residue of a chamber $C$\/} as the union of
cubes $Q$ of $\bar{R}(J,C)$ such that $t(Q)\subset J$ ; we denote it
by ${R}(J,C)$ (or sometimes ${R}_{\Delta}(J,C)$). For example, 
${R}(\emptyset,C)=\{v\}$ where $v$ is the center of the chamber $C$.

 For any residue $R$ of $\Delta$ we will denote by ${\mathcal C}(R)$ the
set of chambers of $\Delta$ whose center is contained in $R$.

Clearly for
any $g\in{\rm Aut}_0(\Delta)$ each $i$--adjacency is  invariant under $g$.
Thus for any subset $J\subset  I$ the relation $\sim_J$ is also invariant
under $g$, and for any chamber $C$ we have $gR(J,C)=R(J,gC)$. For example
we have $R(J,C_{\gamma})=\gamma R(J,C_*)$.

Note that extended residues are unions of chambers, while for $J\neq I$,
no $J$--residue contains a chamber.
\end{defn}

\begin{figure}[ht!]
\labellist
\tiny\hair 2pt
\pinlabel $\{3,4\}$ at 0 246
\pinlabel $4$ at -5 201
\pinlabel $\{4,5\}$ at 0 147
\pinlabel $\{4,5\}$ at -3 93
\pinlabel $4$ at -5 45
\pinlabel $\{3,4\}$ at -3 -3
\pinlabel $3$ at 57 249
\pinlabel $C$ at 51 205
\pinlabel $5$ at 36 135
\pinlabel $5$ at 39 111
\pinlabel $3$ at 52 -4
\pinlabel $\{2,3\}$ at 96 250
\pinlabel $\emptyset$ at 85 197
\pinlabel $\{5,6\}$ [bl] at 61 125
\pinlabel $\emptyset$ at 85 50
\pinlabel $\{2,3\}$ at 92 -4
\pinlabel $2$ at 132 226
\pinlabel $6$ at 120 113
\pinlabel $2$ at 132 20
\pinlabel $\{1,2\}$ at 159 213 
\pinlabel $1$ at 169 151
\pinlabel $\{1,6\}$ at 136 133 
\pinlabel $1$ at 169 93
\pinlabel $\{1,2\}$ at 159 29 
\pinlabel $2$ at 195 229
\pinlabel $6$ at 200 110
\pinlabel $2$ at 193 16
\pinlabel $\{2,3\}$ at 230 251
\pinlabel $\emptyset$ at 235 199
\pinlabel $\{5,6\}$ [br] at 252 128
\pinlabel $\emptyset$ at 233 50
\pinlabel $\{2,3\}$ at 228 -6
\pinlabel $3$ at 274 249
\pinlabel $5$ at 280 135
\pinlabel $5$ at 284 111
\pinlabel $3$ at 274 -3
\pinlabel $\{3,4\}$ at 327 246
\pinlabel $4$ at 323 206
\pinlabel $\{4,5\}$ at 330 163
\pinlabel $\{4,5\}$ at 329 89
\pinlabel $4$ at 323 45
\pinlabel $\{3,4\}$ at 326 0
\pinlabel $\{3,4\}$ at 388 245
\pinlabel $4$ at 378 205
\pinlabel $\{4,5\}$ at 385 147
\pinlabel $\{4,5\}$ at 385 91
\pinlabel $4$ at 378 44
\pinlabel $\{3,4\}$ at 385 -5
\pinlabel $3$ at 445 248
\pinlabel $C$ at 440 203
\pinlabel $5$ at 424 133
\pinlabel $5$ at 427 110
\pinlabel $3$ at 444 -4
\pinlabel $\{2,3\}$ at 484 249
\pinlabel $\emptyset$ at 471 196
\pinlabel $\{5,6\}$ [bl] at 447 124
\pinlabel $\emptyset$ at 471 47
\pinlabel $\{2,3\}$ at 478 -5
\pinlabel $2$ at 520 227
\pinlabel $6$ at 506 111
\pinlabel $2$ at 520 17
\pinlabel $\{1,2\}$ at 545 212
\pinlabel $1$ at 558 151
\pinlabel $\{1,6\}$ at 522 131
\pinlabel $1$ at 559 92
\pinlabel $\{1,2\}$ at 545 28
\pinlabel $2$ at 583 226
\pinlabel $6$ at 586 109
\pinlabel $2$ at 581 15
\pinlabel $\{2,3\}$ at 618 250
\pinlabel $\emptyset$ at 622 192
\pinlabel $\{5,6\}$ [br] at 640 125
\pinlabel $\emptyset$ at 620 49
\pinlabel $\{2,3\}$ at 616 -7
\pinlabel $3$ at 661 248
\pinlabel $5$ at 672 139
\pinlabel $5$ at 667 110
\pinlabel $3$ at 662 -4
\pinlabel $\{3,4\}$ at 715 245
\pinlabel $4$ at 711 205
\pinlabel $\{4,5\}$ at 716 162
\pinlabel $\{4,5\}$ at 717 88
\pinlabel $4$ at 711 44
\pinlabel $\{3,4\}$ at 716 -1
\endlabellist
\centering
\includegraphics[width=12cm]{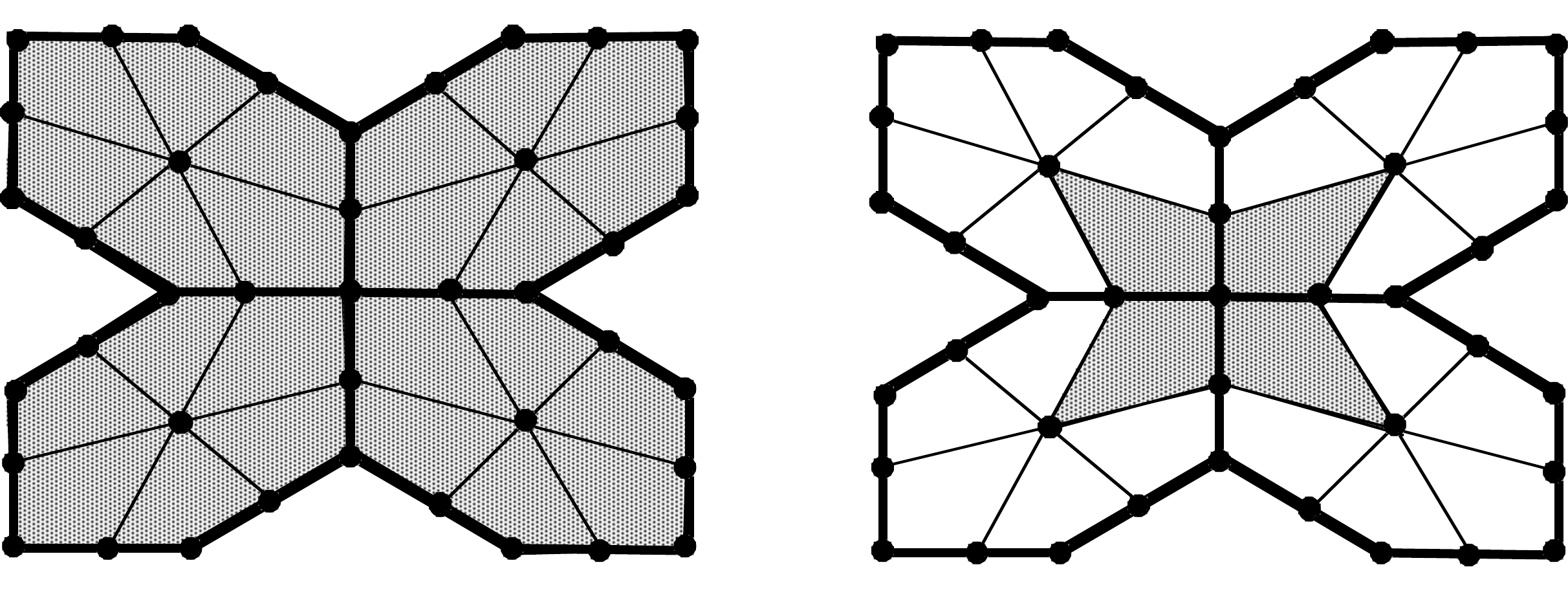}
\caption{Left: The extended $\{1,6\}$-residue
of a chamber C. Right: The $\{1,6\}$-residue
of the same chamber. In both cases the building is Fuchsian: it corresponds to the circular
graph on $I=\{1,2,3,4,5,6\}$ when the groups $G_{i}$ are of order two.}
\end{figure}
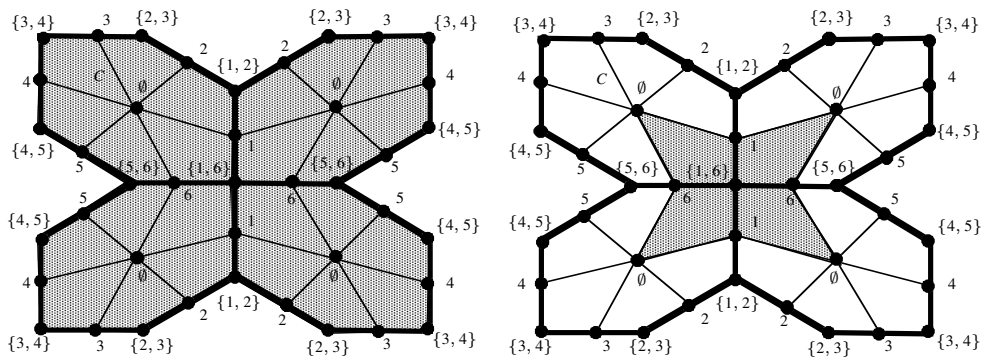

\begin{lem}\label{lem:definedtyp}
Let $i$ denote some element of  $I$ and
let $C$ and $C'$ be any pair of distinct
$i$--adjacent chambers. Then either $C=C'$ or $C\cap C'$ contains a single
rank--1 vertex.
\end{lem}

\begin{proof}
It is enough to argue when $C=C_*$. By \fullref{lem:stabtyp}, $C\cap C'$
contains a single vertex $w$ of type $\{i\!\}$. Assume that $C\cap C'$
contains a vertex $w'$ of type
$\{j\}$ with
$j\neq i$.

Then by \fullref{lem:stabtyp}, there are $\gamma\in\Gamma_i$ and
$\gamma'\in\Gamma_j$ such that $C'=\gamma C$ and $C'=\gamma' C$. Since
$\Gamma$ acts simply transitively on chambers we have $\gamma=\gamma'$.

So the lemma will follow if we prove that $\Gamma_i\cap\Gamma_j=\{1\}$
whenever $i\neq j$. But this relation is true when the set of vertices of $\mathcal G$ is
$\{i,j\!\}$, hence it is true in general because $\Gamma_{\{i,j\!\}}$ embeds
in $\Gamma$ by \fullref{lem:naturality}.
\end{proof}

\begin{lem}\label{lem:reschamber}
 Let $R=R(J,C)$ denote some residue of $\Delta$. Let $C'$ denote a
chamber  of center $w'$. Then $w'\in R$ if and only if $C\sim_J C'$.
Equivalently $C'\in {\mathcal C}(R)\iff C'\sim_J C$.
\end{lem}

\begin{proof}
 The ``if'' part follows by definition. Conversely assume $w'\in R$.
Then $w'$ lies in some chamber $C''$ such that $C''\sim_J C$. By
\fullref{lem:stabtyp}, $w'$ has to be the center of $C''$, hence
$C''=C'$ and we are done.
\end{proof}

 \begin{defn}[\rm Galleries]\label{defn:gal} \textit{A gallery of $\Delta$ of
length $n$ joining $C$ and $C'$} is a sequence of chambers
$(C_0,C_1,\dots,C_n)$ such that $C_0=C$, $C_n=C'$ and for each integer
$1\le k\le n$ the chambers $C_{k-1}$ and $C_k$ are $i_k$--adjacent for
some $i_k\in I$. We say that $(C_0,C_1,\dots,C_n)$ is a $J$--gallery
whenever for each integer $1\le k\le n$ the chambers $C_{k-1}$ and $C_k$
are $j_k$--adjacent for some $j_k\in J$. A gallery $(C_0,C_1,\dots,C_n)$ is \textit{closed\/}
whenever $C_0=C_n$.

The \textit{product\/} of two galleries $G\cdot G'$ is defined by concatenation.

  \textit{The combinatorial distance
between two chambers} is the minimal length of a gallery joining them. A
gallery is said to be \textit{geodesic\/} whenever its length is the distance
between its endpoints.
\end{defn}

\begin{lem}[Residues are subbuildings]\label{lem:res}
 Let $J$ denote any subset of $I$. The natural maps $\Gamma({\mathcal
G}_J,(G_i)_{i\in J})\to \Gamma_J\subset \Gamma=\Gamma({\mathcal
G},(G_i)_{i\in I})$ and $C(N({\mathcal G}_J))\to C(N({\mathcal G}))$
induce an equivariant type-preserving combinatorial embedding 
 $\delta_J$ from $\Delta({\mathcal G}_J,(G_i)_{i\in J})$ to $\Delta({\mathcal
G},(G_i)_{i\in I})$ whose image is the residue $R(J,C_*)$.
\end{lem}

\begin{proof}
First, the inclusion ${\mathcal G}_J\subset {\mathcal G}$ yields  a
natural combinatorial inclusion of $N({\mathcal G}_J)$ into $N({\mathcal G})$.
Then there is an induced combinatorial map
$i_J\co C(N({\mathcal G}_J))\to C(N({\mathcal G}))$, by \fullref{lem:cubcone}. If two vertices $x$ and $y$
are identified under this map then the unique smallest cubes $Q_x$ and $Q_y$
containing $\{x,v\}$ and $\{y,v\}$ are mapped to the same cube. By injectivity
of the map $N({\mathcal G}_J)\to N({\mathcal G})$ this means that
$Q_x=Q_y$, whence $x=y$. Now the injectivity of $C(N({\mathcal G}_J))\to
C(N({\mathcal G}))$ on the set of vertices implies injectivity because
cubical cones are simple cube complexes.

 Note that the inclusion $C(N({\mathcal G}_J))\to C(N({\mathcal G}))$
preserves the typing maps.

 We already know that the natural morphism $\Gamma({\mathcal
G}_J,(G_i)_{i\in J})\to \Gamma_J$ is an isomorphism.

 Consider the map $\Gamma({\mathcal
G}_J,(G_i)_{i\in J})\times C(N({\mathcal G}_J))\to\Gamma\times
C(N({\mathcal G}))$. It is a type-preserv\-ing, equivariant inclusion, and it induces a type-preserving, equivariant
combinatorial map $\delta_J \co \Delta({\mathcal G}_J,(G_i)_{i\in J})
\to\Delta $. Now assume two pairs $(\gamma,p)$ and $(\gamma',p')$ of
$\Gamma({\mathcal G}_J,(G_i)_{i\in J})\times C(N({\mathcal G}_J))$ are identified in $\Delta$. This means  that $i_J(p)=i_J(p')$ and that $\gamma^{-1}\gamma'\in \Gamma_{\underline t(i_J(p))}$. Using injectivity of the natural maps we see that in fact $(\gamma,p)$
and $(\gamma',p')$ are identified in $\Delta({\mathcal G}_J,(G_i)_{i\in J})$.

 It remains to show that the image of $\delta_J$ is $R(J,C_*)$.  First,
the image of the base chamber $C_*^J$ of $\Delta({\mathcal
G}_J,(G_i)_{i\in J})$ is the union of cubes $Q$ of $C_*$ containing the
center $v_*$ of $C_*$ and such that  (either $Q=\{v_*\}$ or) ${\rm link}(v_*,Q)$ is  a simplex of $N({\mathcal G}_J)$. The last condition
means that the types of the vertices of $Q$ are simplices of $N$ whose
vertices are in $J$, thus $t(Q)\subset J$. So we have shown that  $C_*^J$
maps into $C_*\cap R(J,C_*)$.

 In fact any cube $Q$ of $C_*\cap R(J,C_*)$ corresponds to the image of a
cube of $C_*^J$.  For $t(Q)\subset J$ and $t(Q)$ is a simplex of
$N({\mathcal G})$ since its vertices are pairwise adjacent in ${\mathcal G}$.
So the vertices of $t(Q)$ are pairwise adjacent in ${\mathcal G}_J$, and
$t(Q)\in N({\mathcal G}_J)$. This implies that $Q\in\delta_J(C_*^J)$.

 Let us check that $R(J,C_*)$ is invariant under $\Gamma_J$. As
$\Gamma_J$ is generated by $\Gamma_j=G_j$ for $j\in J$, it enough to prove
that $R(J,C_*)$ is invariant under $\Gamma_j, j\in J$. Now for
$\gamma\in\Gamma_j$ we have $\gamma \bar R(J,C_*)=\bar R(J,\gamma C_*)$, and $\gamma C_*\cap C_*$ contains the unique vertex of
$C_*$ of type $\{j\}$. Thus $\gamma C_*\sim_j C_*$ and $\gamma \bar R(J,C_*)=\bar R(J,C_*)$. It follows that $\gamma  R(J,C_*)=
R(J,C_*)$.

At this stage  we know the following: $C_*^J$ is mapped into $R(J,C_*)$, the
building $\Delta({\mathcal G}_J,(G_i)_{i\in J})$ is the union of the
$\Gamma({\mathcal G}_J,(G_i)_{i\in J})$--translates of $C_*^J$, the
combinatorial map $\delta_J$ is equivariant, and $R(J,C_*)$ is invariant
under $\Gamma_J$. This implies that the image of $\delta_J$ is contained
in $R(J,C_*)$. It remains to prove the reverse inclusion.

 Let $C$ denote any chamber  such that $C\sim_JC_*$. So there is a
sequence of elements $\gamma_1,\dots,\gamma_n$ of $\Gamma$ such that
$(C_*,C_{\gamma_1},\dots,C_{\gamma_n})$ is a $J$--gallery and
$C_{\gamma_n}=C$. By induction on $n$ we may assume that
$\gamma_i\in\Gamma_J$ for each $i<n$. Now by definition $C_{\gamma_n}$ is
$j$--adjacent to $C_{\gamma_{n-1}}$ for some $j\in J$. Applying
$\gamma_{n-1}^{-1}$ shows that $C_{\gamma_{n-1}^{-1}\gamma_n}$ is
$j$--adjacent to $C_*$. 

So $\gamma_{n-1}^{-1}\gamma_nC_*\cap C_*$ contains
the unique vertex $v_*^j$ of $C_*$ such that  $t(v_*^j)=j$. As $C_*$ is a
strict fundamental domain this implies $\gamma_{n-1}^{-1}\gamma_n
v_*^j=v_*^j$. By \fullref{lem:stabtyp}, we then have 
$\gamma_{n-1}^{-1}\gamma_n \in \Gamma_j$, and thus $\gamma_n\in \Gamma_J$.

This implies that $\bar R(J,C_*)=\Gamma_J C_*$.  
Thus  $ R(J,C_*)=\Gamma_J (C_*\cap R(J,C_*))$. As we have already seen,
$C_*\cap R(J,C_*)=\delta_J(C_*^J)$. We deduce that the image of
$\delta_J$ is the full residue $R(J,C_*)$.
\end{proof}

 Each residue $R$ of $\Delta$ is isomorphic to a right-angled building,
thus we will denote by ${\rm Aut}_0(R)$ the group of those combinatorial
automorphisms of the cube complex $R$ preserving the restriction of the
typing map $t$ on the set of vertices of $R$.

\begin{cor}\label{cor:stabres}
Let $\gamma\in\Gamma$ and let $J\subset I$.
The chambers $C$ such that $C\sim_J C_{\gamma}$ are the chambers
$C_{\gamma'}$ with $\gamma'\in\gamma\Gamma_J$.
 The stabilizer of $R(J,C_{\gamma})$ in $\Gamma$ is the conjugate $\gamma
\Gamma_J\gamma^{-1}$. It acts simply transitively on  ${\mathcal
C}(R(J,C_{\gamma}))$.
\end{cor}

\begin{proof}
By transitivity of $\Gamma$ on the set of chambers of $\Delta$ it
suffices to argue when $\gamma=1$.

By \fullref{lem:stabtyp},  the graph product $\Gamma({\mathcal
G}_J,(G_i)_{i\in J})$ acts simply transitively on the rank--0 vertices of 
$\Delta({\mathcal G}_J,(G_i)_{i\in J})$. By \fullref{lem:res},  we deduce that $\Gamma_J$ preserves $R(J,C_*)$ and
acts simply transitively on the rank--0 vertices of $R(J,C_*)$.

Let $C$ denote a chamber with center $\gamma'v_*$. Then by
\fullref{lem:reschamber} we have
$C\sim_J C_*$ if and only if $\gamma'v_*\in R(J,C_*)$.
 This is equivalent to the fact that there exists a
$\gamma''\in\Gamma_J$ such that $\gamma' v_*=\gamma'' v_*$. Applying
\fullref{lem:stabtyp} now in $\Delta$, we see that
$\gamma'v_*\in R(J,C_*)$ is equivalent to $\gamma'\in \Gamma_J$.

This proves that the chambers $C$ such that $C\sim_J C_*$ are the chambers
$C_{\gamma'}$ with $\gamma'\in\Gamma_J$. The rest of the corollary
follows easily.
\end{proof}

\subsection{Local structure and simple connectedness of the building}

\begin{defn}[\rm Lower link]
For any vertex $v$ of $\Delta$ let $\underline{\rm link}(v,\Delta)$ denote the
subcomplex of ${\rm link}(v,\Delta)$ whose simplices correspond to cubes
$Q$ containing $v$ and such that $t(Q)\subset t(v)$. We call this
subcomplex of the link \textit{the lower link of $v$\/}.
\end{defn}

\begin{lem}\label{lem:lowerlink}
Let $v$ denote any vertex of $\Delta$. Then
$\underline{\rm link}(v,\Delta)$  is a thickened octahedron. Each
maximal simplex of the lower link is a strict fundamental domain for the
action of the stabilizer of $v$ in $\Gamma$.
\end{lem}

\begin{proof}
As usual we may assume that  $v=v_*$. Let $J\in\wbar N$ denote the
type of $v$. Then $\underline{\rm
link}(v,\Delta)={\rm link}(v,R(J,C_*))$. Now the base chamber is
always a strict fundamental domain for the action of the graph product.  Thus the cube $Q_v=C_*\cap R(J,C_*)$ is a strict fundamental domain for
the action of $\Gamma_J$, and the simplex ${\rm link}(v,Q_v)$ is a strict
fundamental domain for the action of $\Gamma_J$ on $\underline{\rm
link}(v,\Delta)$. Now \fullref{exmp:thickocta} concludes the proof.
\end{proof}

\begin{lem}[Links in $\Delta$]\label{lem:linksdelta}
Let $v$ denote a vertex of $C_*$ of type $t(v)=J\subset \wbar N$. Let
$\overline{\rm link}(v,\Delta)\subset {\rm link}(v,\Delta)$ denote the union of simplices $\sigma$ corresponding to cubes $Q$  such
that the vertex $v\in Q\subset C_*$, the intersection $t(Q)\cap J$ is empty and the union $t(Q)
 J$ is a simplex of $N$. (In
other words, $\overline{\rm link}(v,\Delta)$ corresponds to the full subcomplex of $N$ on vertices which are joinable in $\mathcal G$ to each
vertex of $J$.)
Then the link of $v$ in $\Delta$  is the join of $\overline{\rm
link}(v,\Delta)$ with
$\underline{\rm link}(v,\Delta)$.
\end{lem}

\begin{proof}
 By \fullref{lem:stabtyp}, the group $\Gamma_J$ acts simply transitively on
the set ${\mathcal C}(v)$  of chambers containing $v$.
We have ${\rm link }
(v,\Delta)=\bigcup_{C\in{\mathcal C}(v)}{\rm link } (v,C)$ by \fullref{cor:chamberneighbor}. Observe that
$\Gamma_J$ fixes pointwise $\overline{\rm link}(v,\Delta)$. Thus the
result follows because by \fullref{lem:linkscone}, we have already  ${\rm
link}(v,C_*)=\overline{\rm link}(v,C_*)*\underline{\rm link}(v,C_*)$.
\end{proof}

\begin{prop}\label{prop:cat0} The building $\Delta$ is a $\CAT(0)$ cube
complex.  Residues of $\Delta$ are convex subcomplexes.
\end{prop}
The fact that $\Delta$ is $\CAT(0)$ is proved by Davis \cite{Davis98} and
Meier \cite{Meier96}.
\begin{proof}
By \fullref{lem:linksdelta}, we have a description of the link of each
vertex $v$ as the join of two simplicial complexes $\underline{\rm
link}(v,\Delta)$ and $\overline{\rm link}(v,\Delta)$. By
\fullref{lem:lowerlink},
$\underline{\rm link}(v,\Delta)$ is a thickened octahedron, hence is
a flag.  By construction $\overline{\rm link}(v,\Delta)$ is the link of the
simplex $t(v)$ in the flag complex $N$, thus it is a flag. Consequently
each vertex link is a flag and $\Delta$ is nonpositively curved.

Let $\tilde p\co \widetilde \Delta\to \Delta$ denote the universal cover.
This induces a typing map and a notion of rank on the set of vertices of
$\widetilde \Delta $. Let
$\widetilde {C_*}$ denote some connected component of
$p^{-1}(C_*)$. Since $C_*$ is a cone it is contractible, hence
$p\co \widetilde {C_*}\to C_*$ is an isomorphism.

Let $\widetilde \Gamma$ denote the group of automorphisms $\tilde \gamma$
of $\widetilde \Delta$ such that there exists $\gamma\in \Gamma$
satisfying $p\circ \tilde\gamma=\gamma\circ p$. The map
$\varphi:\tilde\gamma\mapsto \gamma$ is a morphism onto $\Gamma$; the
kernel is the fundamental group of $\Delta$.

The subgroup $\Gamma_i=G_i$ is the stabilizer of the unique vertex
$v_*^i$ of $C_*$ of type $\{i\!\}$. Let us denote by
$\widetilde{\Gamma_i}$ the set of lifts of elements of
$\Gamma_i$ fixing the unique vertex of $\widetilde{C_*}$ projecting to
$v_*^i$. Then  $\widetilde{\Gamma_i}$ is a subgroup of
$\widetilde \Gamma$ for each $i\in I$, and the restriction
$\varphi_i\co \widetilde{\Gamma_i}\to {\Gamma_i}$ is an isomorphism.

Assume that $i$ and $j$ are linked in $\mathcal G$. Let $\tilde v$ denote the
unique vertex of  $\widetilde{C_*}$ projecting to the vertex $v$ of $C_*$
such that $t(v)=\{i,j\!\}$. Then for any $\tilde
\gamma_i\in\widetilde{\Gamma_i}$ and $\tilde
\gamma_j\in\widetilde{\Gamma_j}$, the commutator $[\tilde
\gamma_i,\tilde
\gamma_j]$ maps to $1\in\Gamma$. Thus $[\tilde
\gamma_i,\tilde
\gamma_j]$ is in the deck transformation group. But it fixes $\tilde v$.
Thus $[\tilde
\gamma_i,\tilde
\gamma_j]=1$. We have proved that
$[\widetilde{\Gamma_i},\widetilde{\Gamma_j}]=1$.

For each rank--0 vertex $v$ of $\Delta$, there is a path
$(v_0=v_*,v_1,\dots,v_n=v)$ using only vertices of rank at most $1$. The lift
of this path at the center of $\widetilde {C_*}$ is contained in the
orbit of $\widetilde {C_*}$ under the subgroup generated by
$\{\widetilde{\Gamma_i}\}_{i\in I}$. The endpoint of this path is an arbitrary vertex
projecting onto a rank--0 vertex. This shows that $\widetilde{\Gamma}$ is 
generated by
$\{\widetilde{\Gamma_i}\}_{i\in I}$, because $\widetilde{\Gamma}$ acts freely on
rank--0 vertices of $\widetilde{\Delta}$.

Now the inverse isomorphisms 
${\varphi_i}^{-1}\co {\Gamma_i}\to\widetilde{\Gamma_i}$ define a morphism
$\psi\co \Gamma\to
\widetilde{\Gamma}$. We have $\psi\varphi=1$ on $\widetilde{\Gamma}$
because this is true on each $\widetilde{\Gamma_i}$, and these latter
subgroups generate $\widetilde{\Gamma}$.

Thus ${\rm Ker}\, \varphi=1$ and $\Delta$ is simply connected, therefore
$\CAT(0)$.

Now let $R$ denote some $J$--residue of $\Delta$, and let us prove that
$R$ is $\CAT(0)$--convex. We may assume that
$R=R(J,C_*)$.  Recall that  there is an injective combinatorial map
$i_J\co C(N({\mathcal G}_J))\to C(N({\mathcal G}))$ (see
\fullref{lem:res}). But in fact there is also a cellular map $r_J: 
C(N({\mathcal G}))\to C(N({\mathcal
G}_J))$ such that $r_J\circ i_J=1$: this is just the restriction to 
$C(N({\mathcal G}))\subset \ell^2(I)$ of the orthogonal
projection $\ell^2(I)\to \ell^2(J)$. Note that for any segment $[x,y]$
inside a cube of $C(N({\mathcal G}))$ the length  of $[r_J(x),r_J(y)]$ is
not larger than that of $[x,y]$.

Recall that we also have a retraction morphism  $\rho_J
\co \Gamma\to\Gamma({\mathcal G}_J,(G_i)_{i\in J})$, such that $\rho_J$
composed with the natural inclusion $\Gamma({\mathcal G}_J,(G_i)_{i\in
J})\to \Gamma_J$ is the identity.

We leave it to the reader to verify that
$(\rho_J,r_J)\co \Gamma\times C\to \Gamma({\mathcal
G}_J,(G_i)_{i\in J})\times C(N({\mathcal G}_J))$ induces a cellular map 
$\pi _J\co \Delta\to\Delta({\mathcal
G}_J,(G_i)_{i\in J})$ such that $\pi _J\circ \delta_J=1$. In fact $\pi
_J$ sends piecewise geodesics of length $L$ to piecewise geodesics of
length $\le L$.

Now take two points $p$ and $q$ of
$R=R(J,C_*)=\delta_J(\Delta({\mathcal
G}_J,(G_i)_{i\in J}))$, and consider the geodesic $g$ between them. It follows that $g'=\delta_J(\pi _J(g))$ is a
piecewise geodesic with endpoints $p$ and $q$ contained in $R$, and the length
of $g'$ is at most the length of $g$. Thus $g'=g$ and $R$ is convex.
\end{proof}

\begin{defn}\label{defn:homotopies} Two galleries $G=(C_0,\dots,C_n)$ and $G'=(C'_0,\dots,C'_m)$ differ by
an \textit{elementary (rank--$r$)-homotopy\/} whenever there exists integers $i, j,k$ such that
\begin{enumerate}
\item $(C_0,\dots,C_i)=(C'_0,\dots,C'_i)$
\item the centers of $C_i,\dots,C_j,C'_i,\dots,C'_k$ belong to some residue $R(J,C)$ with $\size{J}\le r$
\item $(C_j,\dots,C_n)=(C'_k,\dots,C'_m)$.
\end{enumerate}
The  \textit{(rank--$r$)-homotopy\/} is the equivalence relation  (on galleries) generated by elementary (rank--$r$)-homotopies.

A  \textit{lassoe\/} is a closed gallery $G$ of the form
$G'\cdot(C_0,C_1,$ $C_2,C_3,C_0)\cdot \overline{G'}$, where $G'$ denotes  any
gallery from the origin of $G$ to the chamber $C_0$, and  $C_0, C_1,C_2, C_3$ are pairwise distinct chambers in some  residue $R(J,C)$ with ${J}=\{i,j\!\}$ an edge of $\mathcal G$.
\end{defn}

\begin{lem}\label{lem:2simplyconnected}
Every closed gallery of $\Delta$ is (rank--1)--homotopic  to a product of
lassoes.
\end{lem}

\begin{proof}
 Let $G=(C_0,C_1,\dots,C_n)$ denote a closed gallery. Using transitivity
we may assume $C_0=C_n=C_*$. Let $\gamma_j\in\Gamma$ be such that
$C_j=\gamma_j C_*$. Set $g_j=\gamma_j^{-1}\gamma_{j+1}$. Since $C_{j-1}$
and $C_j$ are adjacent chambers, either $C_{j-1}=C_j$ in which case we
have $g_j=1$, or there exists a unique $i_j\in I$ such that $g_j\in
G_{i_j}$. Thus we obtain a sequence $g_1,\dots,g_n$ of $\bigcup_i G_i\subset
\Gamma$ such that $C_j=g_1\dots g_j C_*$ and $g_1\dots g_n=1$.

This means that in the free product $\ast_{i \in I}\, G_i$, the element $g_1\dots
g_n$ is a product of conjugate of commutators $[h_k,h_{\ell} ]$, with
$h_k\in G_k\setminus\{1\}$ and $h_{\ell}\in G_{\ell}\setminus\{1\},k\neq\ell$. Note that two products $a_1\dots
a_p$ and $b_1\dots b_q$ for $a_i,b_j \in \bigcup_k G_k$ are equal
in the free product $\ast\, G_k$ if and only if the corresponding
galleries $(C_*,a_1C_*,\dots,a_1\dots a_p C_*)$ and
$(C_*,b_1C_*,\dots,b_1\dots b_q C_*)$ are (rank--1)--homotopic.

 This shows that our initial gallery $G$ is (rank--1)--homotopic to a
product of lassoes, since the gallery with initial chamber $C_*$
corresponding to a conjugate of a commutator $[h_k,h_{\ell} ]$ as above  is a lassoe.
\end{proof}

\subsection{Product structure of certain buildings and residues}

\begin{lem}\label{lem:prodbuild}
 Let $\mathcal G$ denote a graph with vertex set $I$. Let $(G_i)_{i\in 
I}$ denote a family of groups. Assume that we are given a partition
$I=I_1\sqcup I_2$ in such a way that each $i_1\in I_1$ is adjacent to
each $i_2\in I_2$. Consider the three buildings 
$$\Delta=\Delta({\mathcal
G},(G_i)_{i\in  I}),\quad\Delta_1=\Delta({\mathcal G}_{I_1},(G_i)_{i\in 
I_1}),\quad \text{and} \quad \Delta_2=\Delta({\mathcal G}_{I_2},(G_i)_{i\in  I_2}).$$

 Then there exists an equivariant isomorphism of cube complexes $\delta:
\Delta_1\times \Delta_2\to \Delta$ such that
$\delta(p,v_*^{I_2})=\delta_{I_1}(p)$,
$\delta(v_*^{I_1},p)=\delta_{I_2}(p)$ and $t(\delta (p,q)) =
t_1(p)*t_2(q)$, where we denote by $\tau_1*\tau_2$ the simplex generated
by $\tau_1\cup \tau_2$. Such an isomorphism is unique.
\end{lem}

\begin{proof}
Denote by $\Gamma,\Gamma_1$ and $\Gamma_2$ the graph products
associated to the graphs ${\mathcal G},{\mathcal G}_{I_1}$ and ${\mathcal
G}_{I_2}$, respectively. By definition of graph product, there is a natural
isomorphism $\Gamma_1\times \Gamma_2\to \Gamma$ (on the generating groups
$G_{i_1}\times  \{1_{\Gamma_2}\}$ it forgets the second component and
acts similarly on the generating groups $ \{1_{\Gamma_1}\}\times
G_{i_2}$).

 Let us check now that the chamber $C_*$ is naturally the product
$C_*^{I_1}\times C_*^{I_2}$. To see this note that by assumption the
simplicial complex $N=N({\mathcal G})$ is the join of its subcomplexes
$N_1=N({\mathcal G}_{I_1})$ and $N_2=N({\mathcal G}_{I_2})$. By
\fullref{lem:joinproduct} there is a type-preserving isomorphism
$j\co C(N_1)\times C(N_2) \to C(N)$.

 Now combining the two isomorphisms we get an equivariant type-preserving
isomorphism $(\Gamma_1\times \Gamma_2)\times (C(N_1)\times C(N_2))\to 
\Gamma\times C(N)$. If we interpret the domain $(\Gamma_1\times \Gamma_2)\times
(C(N_1)\times C(N_2))$ as $(\Gamma_1\times C(N_1))\times (\Gamma_2\times
C(N_2))$ it is routine to check that the isomorphism induces an
equivariant type-preserving isomorphism $\delta: \Delta_1\times
\Delta_2\to \Delta$ sending the base point $(v_*^{I_1},v_*^{I_2})$ to
the base point  $v_*$.

 Two isomorphisms with the  properties above would differ by a
type-preserving automorphism $\varphi$ of $\Delta$ commuting with
$\Gamma$ and fixing  $v_*$.  We know that $\varphi(v_*)=v_*$,
$\varphi(C_*)=C_*$ and $\varphi$ preserves the type. But by
\fullref{lem:stabtyp} the typing map $t$ realizes a bijection from the
set of vertices of $C_*$ onto $\wbar N$. This implies that $\varphi$
is the identity on the whole chamber $C_*$. Now on any other chamber
$C_{\gamma}=\gamma C_*$ we have by the commutation assumption that
$\varphi\vert_{C_{\gamma}} = \gamma \varphi\vert_{C_*} \gamma^{-1}$, and
thus $\varphi\vert_{C_{\gamma}} ={\rm id}\vert_{C_{\gamma}}$. Uniqueness
follows.
\end{proof}

Since residues are subbuildings, by \fullref{lem:res} we obtain the
following:

 \begin{cor}[Product structure of
$\iperp$--residues]\label{cor:perpeqres} Fix $i\in I$ and a chamber
$C$ of $\Delta$ with center $w$. There exists a type-preserving
combinatorial isomorphism $\delta_{i,C}$ from $R(\{i\!\},C)\times R(i^\perp,C)$ to $
R(\iperp,C)$ which is equivariant under $\Gamma_i\times \Gamma_{i^\perp}\to
\Gamma_{\iperp}$ and sends $(w,w)$ to $w$ and both $(p,w)$ and $(w,p)$ to
$p$. Such an isomorphism is unique.
\end{cor}

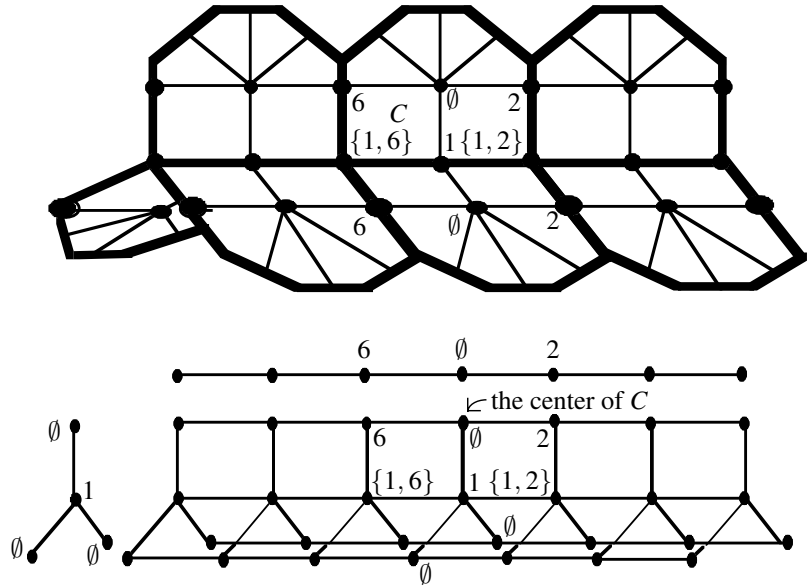
\begin{figure}[ht!]
\labellist
\small\hair 2pt
\pinlabel $6$ at 302 453
\pinlabel $\emptyset$ at 368 455
\pinlabel $2$ at 411 453
\pinlabel $C$ at 330 445
\pinlabel $\{1,6\}$ [bl] at 293 413
\pinlabel $1$ at 366 425
\pinlabel $\{1,2\}$ [br] at 420 413
\pinlabel $6$ at 303 366
\pinlabel $\emptyset$ at 369 367 
\pinlabel $2$ at 438 370
\pinlabel $6$ at 306 280
\pinlabel $\emptyset$ at 374 280 
\pinlabel $2$ at 438 280
\pinlabel $6$ at 317 218
\pinlabel $\emptyset$ at 384 218
\pinlabel $2$ at 431 218
\pinlabel {$\text{the center of } C$} [l] at 391 244
\pinlabel $\{1,6\}$ [bl] at 308 175
\pinlabel $1$ at 383 187
\pinlabel $\{1,2\}$ [br] at 440 175
\pinlabel $\emptyset$ at 408 156
\pinlabel $\emptyset$ at 349 123
\pinlabel $\emptyset$ at 90 223
\pinlabel $\emptyset$ at 64 141
\pinlabel $\emptyset$ at 117 136
\pinlabel $1$ at 114 181
\endlabellist
\centering
\includegraphics[width=13cm]{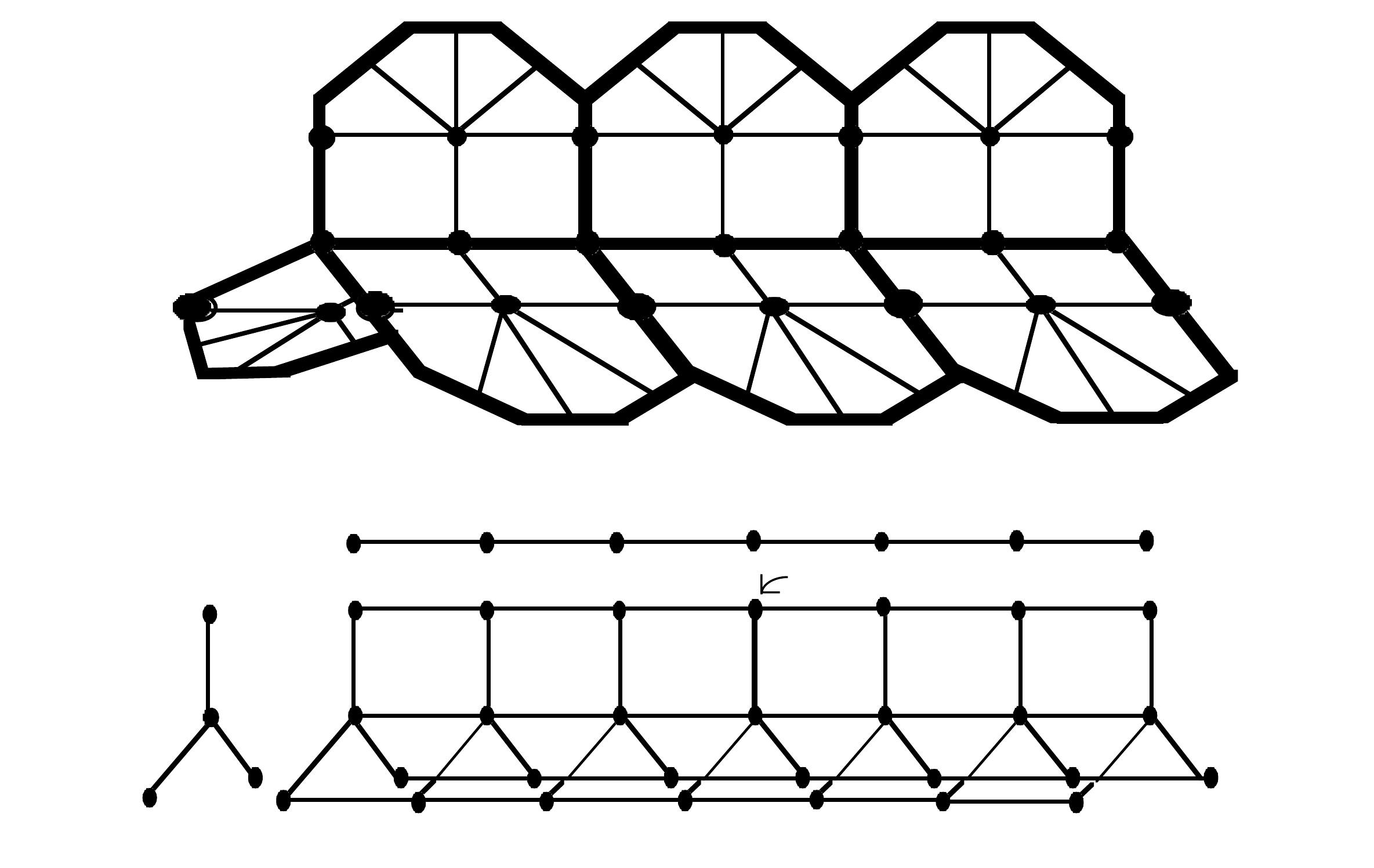}
\caption{The building is associated to a graph product of groups along a circular graph on $\{1,2,3,4,5,6\}$.
The group $G_{1}$ is of order $3$. The groups $G_2$ and $G_6$ are of order $2$. We have $1^{\perp} = \{2,6\}$ and
$1^{\perpeq} = \{1,2,6\}$. Above we show the extended $\{1,2,6\}$--residue of a chamber $C$. Below we show
the corresponding $\{1,2,6\}$--residue, product of the (finite) $\{1\}$--residue and of the $\{2,6\}$--residue, a subdivided line in this case.}
\end{figure}

\section{Holonomy of a group of type-preserving automorphisms}\label{sec:holonomy}

\begin{defn}[\rm $i$--boundary of residues]\label{defn:boundaryres}
 Let $R$ denote a residue. The \textit{$i$--boundary of $R$\/}, denoted by
$\partial _iR$, is the union of cubes $Q$ of $R$ such that $i\not\in
t(Q)$.
\end{defn}

\begin{exmp}\label{exmp:ibound}
 Let $i$ denote some element of $I$. Let us determine the $\{i\!\}$--residue
of $C_*$. By \fullref{cor:stabres}, we have $C_{\gamma}\sim_{\{i\!\}}
C_*\iff
\gamma\in \Gamma_i$. Thus $\bar R(\{i\!\},C_*)$ is the union of the
chambers $\gamma C_*$ for $\gamma\in\Gamma_i$, and  $R(\{i\!\},C_*)$ is the
union of all edges $\gamma [v_*,v_*^i]$, for $\gamma\in \Gamma_i$, where
$v_*^i$ denotes the unique vertex of $C_*$ with type $\{i\!\}$. So
topologically  $R(\{i\!\},C_*)$ is a cone with vertex $v_*^i$ 
The boundary $\partial _iR(\{i\!\},C_*)$ is just the base of the cone, that is, the
set of $q_i$ rank--0 vertices contained in $R(\{i\!\},C_*)$.

In this example, we observe a general fact: the stabilizer of a residue
$R$ in ${\rm Aut}_0(\Delta)$ acts on the $i$--boundary of $R$.
\end{exmp}

\begin{lem}\label{lem:iboundrank2}
Let $i,j$ denote distinct adjacent vertices of $\mathcal G$. Let $R$
denote the $\{i,j\!\}$--residue of some chamber $C$. Then the center of a
chamber $C'$ of $R$ is in the same connected component of $\partial_iR$
 as the center of $C$ if and only if
$C$ and $C'$ are
$j$--adjacent.
\end{lem}

\begin{proof}
The ``if'' part is obvious. Conversely let $w,w'$ denote the centers of
$C$ and $C'$ and assume that there exists a sequence
$(v_0=w,v_1,\cdots,v_n=w')$ of vertices such that each $\{v_i,v_{i+1}\}$
is contained in an edge of $\partial_iR$. Necessarily
$t(v_k)\subset \{j\}$. But by
\fullref{lem:stabtyp},  each rank--0 vertex is adjacent to a unique
vertex of type $\{j\}$ and to no other rank--0 vertex. Thus in the family
$\{v_0,v_1,\cdots,v_n\}$ there is at most one vertex of type $\{j\}$, and
all other vertices are of rank 0. We deduce that $w'$ and $w$ are
$j$--adjacent.
\end{proof}

We may reformulate the lemma above in the following way: the connected
components of $\partial_iR$ are the $\{j\}$--residues contained in $R$.

\begin{lem}\label{lem:iboundperpeq}
Let $C$ be a chamber and $i$ an element of $I$. Then 
$\delta_{i,C}$ induces a type-preserving
$\Gamma_{\iperp}$--equivariant isomorphism 
$\partial_i R(\{i\!\},C)\times R(i^\perp,C)\to \partial_i R(\iperp,C)$.
The inclusion $\partial_i R(\{i\!\},C)\to \partial_i R(\iperp,C)$
induces a bijective map from $\partial_i R(\{i\!\},C)$ to $\pi _0(\partial_i
R(\iperp,C))$. Composing this with the identification ${\mathcal
C}(R(\{i\!\},C))\to \partial_i R(\{i\!\},C)$, we get a bijection ${\mathcal
C}(R(\{i\!\},C))\to \pi _0(\partial_i R(\iperp,C))$.
\end{lem}

\begin{proof}
The inclusion $R(\{i\!\},C)\subset R(\iperp,C)$ leads the
injective map from $\partial_i R(\{i\!\},C)$ to $\partial_i R(\iperp,C)$.
Remember from \fullref{exmp:ibound} that $\partial_i R(\{i\!\},C)$ consists of isolated rank--0
vertices, so the injection induces a map $\partial_i
R(\{i\!\},C)\to \pi _0(\partial_i R(\iperp,C))$, and we have to check
that this map is a bijection.

By \fullref{cor:perpeqres}, we have an isomorphism $\delta_{i,C}
\co R(\{i\!\},C)\times R(i^\perp,C)\to R(\iperp,C)$ sending a vertex
$(v_1,v_2)$ to a vertex of type
$t(v_1)*t(v_2)$.
In other words $\partial_i R(\iperp,C)$ is the image under $\delta_{i,C}$ of
the cubes $\{w'\}\times Q$, where $w'$ is a rank--0 vertex of $R(\{i\!\},C)$
and $Q$ is any cube of $ R(i^\perp,C)$ (recall that $i\not\in i^\perp$).
Thus $\delta_{i,C}$ induces an isomorphism 
$\partial_i R(\{i\!\},C)\times R(i^\perp,C)\to \partial_i R(\iperp,C)$ which is  type-preserving.

Residues are connected, hence sending $w'$ to $(w',w)$ identifies $\pi
_0(\partial_i R(\{i\!\},C))$ with $\pi _0(\partial_i
R(\{i\!\},C)\times R(i^\perp,C))$, where $w$ denotes the center of the
chamber $C$. Composing with
$\delta_{i,C}$ yields the desired bijection.
\end{proof}

\begin{cor}\label{cor:iboundrank2}
Let $i$ and $j$ denote distinct adjacent vertices of $\mathcal G$. Let $R$
denote the $\{i,j\!\}$--residue of some chamber $C$. Then the center of a
chamber $C'$ of $R$ is in the same connected component of
$\partial_iR(\iperp,C)$
 as the center of $C$ if and only if
$C$ and $C'$ are
$j$--adjacent.
\end{cor}

\begin{proof}
It is enough to prove that the
inclusion $\partial_iR(\{i,j\!\},C)\subset \partial_iR(\iperp,C)$ induces a bijection
in $\pi _0$ by \fullref{lem:iboundrank2}.

Now each inclusion $R(\{i\!\},C)\to\pi_0(\partial_iR(\{i,j\!\},C))$ and
$R(\{i\!\},C)\to\pi_0(\partial_iR(\iperp,C))$ is bijective by
\fullref{lem:iboundperpeq}.  To deduce the first bijectivity remember
that  $R(\{i,j\!\},C)$ is itself a right-angled building, in which $i\perpeq=\{i,j\!\}$.
\end{proof}

 Let $G$ denote any subgroup of ${\rm Aut}_0(\Delta)$. For any subcomplex
$Y\subset \Delta$ we denote by $G_Y$ the stabilizer of $Y$ in $G$, that
is, the set of $g\in G$ such that $gY=Y$.

\begin{defn}\label{defn:holbuilding}
 Let $G\subset {\rm Aut}_0(\Delta)$ denote a subgroup. Let $i$ denote some
element of   $I$ and let $R$ denote some $\iperp$--residue of
$\Delta$. \textit{The
$i$--holonomy of
$G$ at $R$} is the morphism $G_{R}\to {\mathfrak
S}(\pi _0(\partial_iR))$.

 We say that \textit{the group $G$ has no holonomy\/} if for each $i\in I$ and
each  $\iperp$--residue $R$, the
$i$--holonomy of
$G$ at $R$ is  trivial.

 Observe that if $G_i$ is finite of cardinality $q_i$ then  ${\mathfrak
S}(\pi _0(\partial_iR))$ is the finite group ${\mathfrak S}(q_i) $. This
is because for any chamber $C$ whose center is inside $R$ we know by
\fullref{lem:iboundperpeq} that we have an identification ${\mathcal
C}(R(\{i\!\},C))\to \pi _0(\partial _iR)$, and we have already seen that ${\mathcal
C}(R(\{i\!\},C))$ admits a simply transitive action of a conjugate of
$\Gamma_i=G_i$.
\end{defn}

\begin{exmp}\label{exmp:gamma0}
For each graph product $\Gamma=\Gamma({\mathcal G},(G_i)_{i\in I})$ we
define $\Gamma_0$ as the kernel of the surjective morphism
$\Gamma\to\prod_iG_i$ inducing the identity on each group $G_i$.

If  each group $G_i$ has cardinality $q_i$ then
$\Gamma_0$ is a subgroup of $\Gamma$ of index $\tsty \prod_iq_i$.

The group $\Gamma_0$ has no holonomy.
To check this let us first compute the $i$--holonomy of $\Gamma_0$ at
$C_*$. The stabilizer of $R(\iperp,C_*)$ in $\Gamma$ is
$\Gamma_{\iperp}$. Thus the stabilizer of $R(\iperp,C_*)$ in 
$\Gamma_0$ is the kernel of
$\Gamma_{\iperp}\mapsto \prod_{j\in \iperp}G_j$. In particular
this stabilizer is contained in $\Gamma_{i^{\perp}}$.

But by \fullref{lem:iboundperpeq}, it is clear that the action of
$\Gamma_{i^{\perp}}$ on $\pi _0(\partial_iR(\iperp,C_*))$ is
trivial. Thus the $i$--holonomy of $\Gamma_0$ at $C_*$ is trivial.

Now $\Gamma_0$ is normal in $\Gamma$ and $\Gamma$ is transitive on
chambers. It follows that the $i$--holonomy of $\Gamma_0$ at any
chamber is trivial.
\end{exmp}

\section{Atlases on a regular right-angled building}\label{sec:atlas}

\begin{defn}\label{defn:atlas}
\textit{An atlas $\mathcal A$ on $\Delta$\/} consists in the following data:
for each $i\in I$ and each $\iperp$--residue $R$ we are given a simply
transitive action of $G_i$ on the set $\pi _0(\partial_iR)$ of connected
components of the $i$--boundary of $R$. We will denote the corresponding
representation by ${\mathcal
A}_{i,R}\co G_i\to {\mathfrak S}(\pi _0(\partial_iR))$.

We say that \textit{two atlases ${\mathcal
A},{\mathcal
A}'$ are equivalent} (written ${\mathcal
A}\sim {\mathcal
A}'$) whenever for each $i,R$, there is an element $g_{i,R}\in G_i$ such
that ${\mathcal
A}'_{i,R}(g)={\mathcal
A}_{i,R}(g_{i,R}g{g_{i,R}}^{-1})$ for all $g\in G_i$. We denote by  $[
{\mathcal A}]$ the equivalence class of the atlas ${\mathcal A}$.

Assume $\varphi$ is a type-preserving automorphism of $\Delta$.  Then we may set ${\mathcal
A}'_{i,R}(g)$ equal to $\varphi\circ {\mathcal
A}_{i,\varphi^{-1}(R)}(g) \circ\varphi^{{-}1}$. We thus obtain a new atlas
${\mathcal A}'=\varphi_*({\mathcal A})$. This defines an action of ${\rm
Aut}_0(\Delta)$ on the set of all atlases on $\Delta$. Clearly this
action preserves equivalence of atlases. The stabilizer of the atlas
$\mathcal A$ in this action will be denoted by
${\rm Aut}_0(\Delta,{\mathcal A})$. Similarly we denote by  ${\rm
Aut}_0(\Delta,[{\mathcal A}])$ the stabilizer of the equivalence class
of an atlas $\mathcal A$.
\end{defn}

\begin{exmp}[The standard atlas on $\Delta({\mathcal G},(G_i)_{i\in
I})$]\label{exmp:standardatlas} We may use the group $\Gamma$ 
 to define  atlases on
$\Delta$.

\medskip \noindent (1)\qua More precisely for $i\in I$ and for some $\iperp$--residue
$R$, choose any chamber $C$ in $R$. By \fullref{cor:stabres} and
\fullref{lem:iboundperpeq}, the stabilizer
$\Gamma_{R(\{i\!\},C)}$ is conjugate to $\Gamma_i$ in $\Gamma$ and acts
simply transitively on
$\pi _0(\partial_iR)$. Using one of the possible conjugations we get a
simply transitive  action of
$\Gamma_i=G_i$ on $\pi _0(\partial_iR)$. This defines an atlas. Using the
other possible conjugations yields equivalent atlases, and in fact a
whole equivalence class of atlases. The atlases obtained that way we call
\textit{$\Gamma$--atlases\/}.

Now $\Gamma$ preserves the set of $\Gamma$--atlases. So if
${\mathcal A}$ is a $\Gamma$--atlas then $\Gamma\subset {\rm
Aut}_0(\Delta,[{\mathcal A}])$.

Note that if for all $i$ the group $G_i$ is abelian then the equivalence relation
is trivial, and there is only one $\Gamma$--atlas. In this case we have $\Gamma\subset {\rm
Aut}_0(\Delta,{\mathcal A})$.

\medskip \noindent (2)\qua In fact it is always possible to define more carefully a $\Gamma$--atlas ${\mathcal A}$
 in such a way that $\Gamma\subset {\rm
Aut}_0(\Delta,{\mathcal A})$.

 Consider an element $i\in I$. By \fullref{cor:stabres}, the action
of
$\Gamma_i=G_i$ on ${\mathcal C}(R(\{i\!\},C_*))$ is simply transitive. We
define ${G_i}^*$ as the subgroup of ${\mathfrak S}({\mathcal
C}(R(\{i\!\},C_*)))$ consisting in those permutations $\sigma$ which
commute with each element $g\in G_i$ (acting on ${\mathcal
C}(R(\{i\!\},C_*))$ through $\Gamma_i$). So if $\sigma\in {G_i}^*$ then for
any $g\in G_i$ we must have $\sigma(g C_*)=g\sigma(C_*)$. If we define an
element ${h_{\sigma}}\in G_i$ by the relation 
$\sigma(C_*)={h_{\sigma}}^{-1}C_*$ then we have $\sigma(g
C_*)=g{h_{\sigma}}^{-1}C_*$.

Abbreviating the pair $(i,R(\iperp,C_*))$ by $(i,C_*)$, let us now define for every $h\in  G_i$ a permutation 
${\mathcal
A}^\Gamma_{i,C_*} (h)$ of ${\mathcal C}(R(\{i\!\},C_*))$ by
the rule ${\mathcal A}^\Gamma_{i,C_*}
(h)(gC_*)=gh^{-1}C_*$. We then obtain a map ${\mathcal
A}^\Gamma_{i,C_*}\co G_i\to {\mathfrak S}({\mathcal
C}(R(\{i\!\},C_*)))$. It is easy to check that ${\mathcal
A}^\Gamma_{i,C_*}$ is an injective morphism with
simply transitive image. In fact the image of this morphism is exactly
the subgroup ${G_i}^*$.

 For each $\iperp$--residue  $R$ choose an element
$\gamma_R\in\Gamma$ such that $\gamma_R R(\iperp,C_*)=R$. Then
define for any $h\in G_i$ a permutation of $\pi_0(\partial _iR)$ by
${\mathcal A}^\Gamma_{i,R}(h) =\gamma_R\circ {\mathcal
A}^\Gamma_{i,C_*} (h)\circ {\gamma_R}^{-1}$. Doing this
for each $i$ yields an atlas ${\mathcal A}^\Gamma$, which we call \textit{the standard atlas\/}. Apparently this atlas ${\mathcal A}^\Gamma$ depends
on the choices of the $\gamma_R$, but the following argument shows that
it is not the case.

 In fact let us show that  ${\mathcal A}^\Gamma$ is invariant under
$\Gamma$. To see this choose $i\in I$, $R$ a $\iperp$--residue and
$\gamma\in \Gamma$. Set $R'=\gamma R$. Then we have
$$\gamma_*({\mathcal A}^\Gamma) _{i,R'}(h)=\gamma\circ  {\mathcal
A}^\Gamma  _{i,R}(h)\circ \gamma^{-1} = (\gamma\gamma_R)\circ {\mathcal
A}^\Gamma_{i,C_*}(h) \circ (\gamma\gamma_R)^{-1}.$$ 
So to
prove invariance it suffices to prove independence relative to the choice
of the element $\gamma_{R'}$.

 In fact all this reduces to checking that the action of $\gamma$ on
$\pi_0(\partial_iR(\iperp,C_*))$ commutes with the permutation group
${\mathcal A}^\Gamma_{i,C_*}(G_i)={G_i}^*$ if $\gamma
R(\iperp,C_*)=R(\iperp,C_*)$. By
\fullref{cor:stabres}, the product subgroup $\Gamma_{\{i\!\}}\times
\Gamma_{i^{\perp}}$ is the stabilizer of $R(\iperp,C_*) $ in
$\Gamma$. But $\Gamma_{i^{\perp}}$ acts trivially on 
$\pi_0(\partial_iR(\iperp,C_*))$, and ${G_i}^*$ was defined to
commute with $\Gamma_{\{i\!\}}$.
\end{exmp}

\begin{defn}[\rm ${\mathcal A}$--word of a gallery]\label{defn:galword}
Let ${\mathcal
A}$ denote some atlas on $\Delta$. First assume we are given a pair
$(C,C')$ of adjacent chambers: we want to associate to it a letter $s$ in
the alphabet ${\mathfrak A}=\{\emptyset\}\sqcup (\sqcup_i\{i\!\}\times
(G_i\setminus\{1\}))$. If
$C'=C$ we set
$s=\emptyset$. Else by \fullref{lem:definedtyp}, there is one and
only one $i$ in
$I$ such that $C$ and $C'$ are $i$--adjacent. Furthermore using the
identification ${\mathcal C}(R(\{i\!\},C))\to \pi _0(\partial_i
R(\iperp,C))$ of \fullref{lem:iboundperpeq} together with the
simply transitive representation ${\mathcal
A}_{i,R(\iperp,C)}\co G_i\to \pi _0(\partial_i
R(\iperp,C))$, we obtain a simply transitive action of $G_i$ on
${\mathcal C}(R(\{i\!\},C))$ (the corresponding morphism will be still denoted
by
${\mathcal A}_{i,R(\iperp,C)}$). So there exists a unique
$g_i\in G_i$ sending
$C$ to $C'$ (and $g_i\neq 1$ because $C'\neq C$). We then set $s=(i,g_i)$.

Now to any gallery $(C_0,C_1,\dots,C_n)$ we associate the
word $(s_1,\dots,s_n)$ in ${\mathfrak A}$ with $s_i$ the letter
associated as above to the length one gallery $(C_{i-1},C_i)$.
\end{defn}

\begin{lem}\label{lem:wordgal}
Let $(s_1,\dots,s_n)$ denote some word in ${\mathfrak A}$. Fix a chamber
$C$ and an atlas ${\mathcal
A}$. Then there exists  a unique gallery $(C_0,C_1,\dots,C_n)$ with
$C_0=C$ and whose associated ${\mathcal
A}$--word is $(s_1,\dots,s_n)$.
\end{lem}

\begin{proof}
Using induction on length, it is enough to handle the case $n=1$. If
$s_1=\emptyset$ then the ${\mathcal A}$--word associated to  $(C_0,C_1)$
is $s_1$ if and only if $C_1=C_0$. If $s_1=(i,g_i)$ then necessarily $C_1\in
R(\{i\!\},C_0)$ and in fact using the simply transitive action of $G_i$ on
${\mathcal C}(R(\{i\!\},C))$ described in \fullref{defn:galword} we must
have $C_1={\mathcal A}_{i,R(\iperp,C_0)}(g_i)C_0$. Conversely
the gallery $(C_0,{\mathcal A}_{i,R(\iperp,C_0)}(g_i)C_0)$ has the
desired property.
\end{proof}

\begin{prop}\label{prop:conjugatlas}
Let ${\mathcal
A},{\mathcal
A}'$ denote two atlases on $\Delta$. Let $C$ and $C'$ denote two chambers of
$\Delta$, and let $f\co C\to C'$ denote the unique type-preserving
combinatorial isomorphism between the two chambers. Then $f$ admits one
and only one extension $\fbar\in {\rm
Aut}_0(\Delta)$ such that $\fbar_*({\mathcal
A})={\mathcal
A}'$.
\end{prop}

\begin{proof}
We first give the proof of uniqueness, which then will lead to
existence.

To get uniqueness we show that if  a type-preserving extension is already defined
on some chamber, then it is uniquely defined on each adjacent chamber. In
order to give the argument, we introduce the notion of a germ.
\begin{defn}[\rm Germs]\label{defn:germ}
A \textit{germ\/} of ${\rm
Aut}_0(\Delta)$ is a type-preserving
combinatorial isomorphism between two chambers.
\end{defn}

 Remember that by 
\fullref{lem:stabtyp} the type $t$ realizes a bijection between the set
of vertices of a given chamber and $\wbar N$. So in this context the
pseudo-group of germs is the same as the pseudo-group of pairs of
chambers. Observe that the proposition we are proving just says that any
germ has a unique type-preserving extension sending a given atlas to an
other given atlas.

For any gallery $G=(C_0,C_1,\dots,C_n)$ with origin $C_0=C$ we define
a gallery $f(G)$ as follows: $f(G)$ is the unique gallery with origin
$C'$ and whose ${\mathcal A}'$--word is  the ${\mathcal
A}$--word of $G$ (see \fullref{defn:galword} and
\fullref{lem:wordgal}).

We then define a sequence of germs
$G(f)=(f_0,f_1,\dots,f_n)$ as follows. Write
$f(G)=(C'_0,C'_1,\dots,C'_n)$ and define $f_i$ as the unique germ sending
$C_i$ to $C'_i$.

We claim that we have
$f(G)=\fbar(G)$ for any type-preserving extension $\fbar$ of $f$ and
any gallery $G=(C_0,C_1,\dots,C_n)$ with origin $C_0=C$ , and hence the restriction of $\fbar$ to
$C_n$ is the germ $f_n$. Uniqueness follows.

To check the claim it is enough to handle the case when
$n=1$ and $C_1\neq C_0$ by induction. Write $s_1=(i,g_i)$. By assumption we have
$\fbar(C_0)=C'_0$, and hence
$\fbar (R(\iperp,C_0))=R(\iperp,C'_0)$. To simplify
notation set
$R=R(\iperp,C_0),R'=R(\iperp,C'_0)$. The condition
$\fbar_*({\mathcal A})={\mathcal
A}'$ says that ${\mathcal
A}'_{i,R'}=\fbar\circ {\mathcal
A}_{i,R}\circ\fibar$. Now we may determine the chamber
$\fbar(C_1)$. We have
$$\fbar(C_1)=\fbar({\mathcal
A}_{i,R}(g_i)C_0)=\fbar({\mathcal
A}_{i,R}(g_i)(\fibar(C'_0)))={\mathcal
A}'_{i,R'}(g_i)C'_0.$$ 
This latter chamber is $C'_1$ because the ${\mathcal
A}'$--word associated to $(C'_0,C'_1)$ is $(s_1)$.

We now
prove the existence of a type-preserving extension
$\fbar$ sending $\mathcal A$ to ${\mathcal A}'$, still using the
action of the germ $f$ on galleries with origin $C$.

For any gallery $G=(C_0,C_1,\dots,C_n)$ with $C_0=C$ and $C_n=D$ we
consider the germ
$\omega(G(f))=f_n$ (from $C_n$ to the last chamber $C'_n$ of $f(G)$). We
first prove that this germ is independent of the gallery $G$ with
endpoints $C,D$.

Observe the product rule $\omega(G_1G_2(f))=\omega(G_2(\omega(G_1(f))))$.
If we denote by $\overline G$ the inverse gallery of $G$ we also have
$\omega(G\overline G(f))=f$. This is because the letter associated to the
gallery $(C_1,C_0)$ is $s^{-1}$, where $s$ is the letter associated to
$(C_0,C_1)$, and the inverse of a letter $s$ is $\emptyset$ if
$s=\emptyset$ and $(i,{g_i}^{-1})$ if $s=(i,{g_i})$. In other words
$\omega(G(f))=\omega(G'(f))$ whenever $G$ and $G'$ are (rank--1)--homotopic 
(see \fullref{defn:homotopies}).

If $G,G'$ are galleries from $C$ to $D$ we  have $\omega(G\overline{G'}G'(f))=\omega(G'(\omega(G\overline{G'}(f))))$
 and also
$\omega(G\overline{G'}G'(f))=\omega(\overline{G'}G'(\omega(G(f))))=
\omega(G(f))$. So
to conclude that $\omega(G(f))=\omega(G'(f))$ it is enough to prove that
for every closed gallery $G$ with origin $C$ and any germ $f\co C\to \Delta$
we have $\omega(G(f))=f$.

By \fullref{lem:2simplyconnected}, every closed gallery of $\Delta$ is a
product of lassoes, up to rank--1 homotopies.
So it is enough to check that if $G=(C_0=C,C_1,C_2,C_3,C_4=C)$ then $\omega(G(f))=f$
where $C_0\sim_iC_1,C_0\neq C_1$, $C_1\sim_jC_2,C_1\neq C_2$, $C_2\sim_iC_3,C_2\neq C_3$, $C_3\sim_jC_4,C_3\neq C_4$, for distinct adjacent vertices $i$ and $j$ of $\mathcal G$.

The $\mathcal A$--word of $G$ is then 
$$(s_1,s_2,s_3,s_4) = ((i,g_i),(j,g_j),(i,{g'}_i),(j,{g'}_j).
$$ Note that the
four chambers belong to $R(\iperp,C)$. Since
$C_1$ and
$C_2$ are $j$--adjacent they define the same connected component in
$\partial_iR(\iperp,C)$ (see \fullref{lem:iboundrank2}). The
same is true for
$C_0,C_3$. We deduce that ${g'}_i={g_i}^{-1}$. Similarly we have
${g'}_j={g_j}^{-1}$.

Now let us prove that the gallery $f(G)$ is closed: this will conclude
because $\omega(G(f))$ will be the unique germ from $C_4=C_0$ to
$C'_4=C'_0$, so $\omega(G(f))=f$.

First we note that the centers of $C'_3$ and $C'_0$ define the same
connected component in $\partial_iR(\iperp,C'_0)$. Because so do
$C'_2$ and $C'_1$ (they are $j$--adjacent), and furthermore
${g'}_i={g_i}^{-1}$. This implies that $C'_3$ and $C'_0$ are
$j$--adjacent by \fullref{lem:iboundrank2}. Note that $C'_0\neq
C'_3$, else ${g'}_i={g_i}^{-1}$ would imply $C'_2=C'_1$ thus $g_j=1$. So
there is an element $h_j\in G_j\setminus\{1\}$ such that ${\mathcal
A}'_{j,R(j^{\perpeq},C'_3)}(h_j)C'_3=C'_0$. The same argument as above
shows that ${h}_j={g_j}^{-1}$. So in fact $h_j={g'}_j$ and $C'_4=C'_0$.

So for any chamber $D$ there is a germ $f_D\co D\to\Delta$ such that for any
gallery $G$ from $C$ to $D$, we have
$\omega(G(f))=f_D$. Let us check that all the germs $f_D$ fit
together and define  a
type-preserving combinatorial map $\fbar$. So let $Q$ denote any
cube of $\Delta$, and let $D,D'$ denote any two chambers that contain
$Q$. We must check that $f_D=f_{D'}$ on $Q$. Let $v$ denote the vertex of
$Q$ with type $J=\underline{t}(Q)$.  By
\fullref{lem:stabtyp} the group $\Gamma_v$ acts simply transitively on
the set ${\mathcal C}(v)$  of chambers containing $v$. Let
$\gamma\in\Gamma$ be such that $D'=\gamma D$, and let $g\in\Gamma$ be such that
$D=gC_*$. The conjugate $\gamma'=g^{-1}\gamma g$ belongs to $\Gamma_J$ by
\fullref{lem:stabtyp}. Consider any decomposition  $\gamma'=g_1\dots
g_n$ as a product of elements of $G_i,i\in J$: this yields a gallery
$(C_*,g_1C_*,\dots,\gamma'C_*)$ from $C_*$ to $g^{-1}D'$, all of whose
chambers contain $g^{-1}Q$. Applying $g$ yields a gallery from  $D$ to
$D'$, all of whose chambers contain $Q$. Thus we see that it suffices to
argue in the  case when $D,D'$ are $i$--adjacent for some $i\in J$. Let
$w$ denote the unique vertex in $D\cap D'$ of type $\{i\!\}$. There is a
unique smallest cube $\bar Q$ in $D$ containing $Q$ and $w$. By
\fullref{lem:linksdelta}, $\bar Q$ is also in $D'$.
 We
note that the image chambers
$E$ and $E'$ of the germs
$f_D$ and $f_{D'}$ respectively are $i$--adjacent, because $f_{D'}$ is
obtained by extending $f_D$ along the gallery $(D,D')$. Since $E\cap E'$
contains a unique vertex of type $\{i\!\}$, the germs
$f_D$ and
$f_{D'}$ are equal on $w$. In fact by
\fullref{lem:linksdelta}, the germs $f_D$ and
$f_{D'}$ are equal on $\bar Q$, hence on $Q$.

Now assume ${\mathcal A}''$
is a second atlas,  $C''$ is a second chamber and  $f'$ denotes the
germ sending $C'$ to $C''$. Then it is straightforward that
$(f'f)(G)=f'(f(G))$. This implies that the $i$--th germ of $G(f'f)$ is
    $G(f'f)_i=(f(G))(f')_i\circ G(f)_i$. So
$\overline{f'f}=\fpbar \fbar$. In particular we have
$\fibar \fbar={\rm id}_{\Delta}$. So  $\fbar$
is in fact a type-preserving automorphism extending $f$.

It remains to check that $\fbar_*({\mathcal A}')={\mathcal A}$.
So let $i\in I$ and let $R$ denote some $\iperp$--residue.  Fix a
chamber $D$ in ${\mathcal C}(R)$ and a gallery $G=(C_0=C,\dots,C_n=D)$.
Set $R'=\fbar(R)$. We must show that for every $g\in G_i$ we have
${\mathcal A}'_{i,R'}(g)\circ \fbar= \fbar\circ {\mathcal
A}_{i,R}(g)$ on $\pi_0(\partial_i R)$, or equivalently on ${\mathcal
C}(R(\{i\!\},D))$ using the identification ${\mathcal C}(R(\{i\!\},D))\to
\pi_0(\partial_i R)$ of \fullref{lem:iboundperpeq}. Since the relation
is obvious if $g=1$ we will assume $g\neq 1$.

Set $E={\mathcal A}_{i,R}(g)D$. Consider the gallery  $H=G\cdot
(D,E)=(C_0,\dots,C_n,E)$. Let $H'=(C'_0,\dots,C'_n,E')$ denote its image
under $\fbar$. By construction the last letter in the ${\mathcal
A}'$--word of $H'$ is $(i,g)$. By definition it means that $E'$ is
$i$--adjacent to $D'$ with $E'={\mathcal A}'_{i,R'}(g)D'$. So we have
checked the relations ${\mathcal A}'_{i,R'}(g)(\fbar(D))= \fbar({\mathcal A}_{i,R}(g)D)$ for every $g\in G_i$. This was done
without assumption on $D$.
\end{proof}

\begin{cor}\label{cor:conjugatlas}
If $\mathcal A$ is an  atlas on $\Delta$ then the automorphism
group ${\rm Aut}_0(\Delta,{\mathcal A})$ acts simply transitively on
chambers. For example ${\rm Aut}_0(\Delta,{\mathcal A}^\Gamma)=\Gamma$.
In general ${\rm Aut}_0(\Delta,{\mathcal A})$ is conjugate to $\Gamma$
inside 
${\rm Aut}_0(\Delta)$.
\end{cor}
\begin{proof}
 The first assertion follows from \fullref{prop:conjugatlas}
applied with ${\mathcal A}'={\mathcal A}$; for the case of ${\mathcal
A}^\Gamma$ see \fullref{exmp:standardatlas}. 
 The second assertion then follows from \fullref{prop:conjugatlas}
applied with ${\mathcal A}'={\mathcal A}^\Gamma$. 
\end{proof}

\begin{prop}\label{prop:noholatlas}
 Let $\Gamma'$ denote  a subgroup of ${\rm Aut}_0(\Delta)$ with no
holonomy. Then there exists an atlas $\mathcal A$ on $\Delta$ such that
$\Gamma\subset {\rm Aut}_0(\Delta,{\mathcal A})$.
\end{prop}
\begin{proof}
 Let $i\in I$ denote a fixed element. Choose a family $(R_t)_{t\in T}$ of
$\iperp$--residues such that every $\iperp$--residue is in the
$\Gamma'$--orbit of $R_t$ for one and only one $t$. Choose arbitrarily a 
simply transitive representation ${\mathcal
A}_{i,R_t}\co G_i\to {\mathfrak S}(\pi_0(\partial _iR_t))$ for each residue
$R_t$.

 Now for any $\iperp$--residue $R$ there is some $\gamma'\in\Gamma'$ and a unique $t\in T$ and
 such that $R=\gamma'R_t$. The simply transitive
representation  ${\mathcal A}_{i,R}\co G_i\to {\mathfrak S}(\pi_0(\partial
_iR))$ defined by ${\mathcal A}_{i,R}(g)=\gamma'\circ {\mathcal
A}_{i,R_t}\circ {\gamma'}^{-1}$ is in fact independent of the choice of
$\gamma'$. Indeed for another $\gamma''\in\Gamma'$ sending $R_t$ to $R$
we have $$\gamma'\circ {\mathcal A}_{i,R_t}\circ
{\gamma'}^{-1}=(\gamma'{\gamma''}^{-1})\circ \gamma''\circ {\mathcal
A}_{i,R_t}\circ {\gamma''}^{-1}\circ (\gamma'{\gamma''}^{-1})^{-1}.$$ But
$\Gamma'$ has no holonomy, so the element $\gamma'{\gamma''}^{-1}\in
\Gamma'_R$ acts trivially on $\pi_0(\partial_i R)$, and we are done.

 If we do this for each $i\in I$ we get an atlas $\mathcal A$ on
$\Delta$. We now check that this atlas is invariant under $\Gamma'$.

Let $i\in I$ be a vertex of $\mathcal G$, let $R$ denote some
$\iperp$--residue and let $\varphi$ belong to $\Gamma'$. Set
$R'=\varphi(R)$. For $g\in G_i$ we have
 $$\varphi \circ {\mathcal A}_{i,R}(g)\circ\varphi^{-1} =( \varphi 
\circ\gamma')\circ {\mathcal A}_{i,R_t}(g)\circ(\varphi\circ\gamma'
)^{-1}$$
  with $\gamma'\in\Gamma'$ such that $\gamma' R_t=R$. The element
$\varphi\circ\gamma' \in\Gamma'$ sends $R_t$ to $R'$, so by the
independence property above we have
$$
(\varphi \circ\gamma'
)\circ{\mathcal A}_{i,R_t}(g)\circ(\varphi\circ\gamma' )^{-1} ={\mathcal
A}_{i,R'}(g).\proved$$
\end{proof}

\begin{thm}\label{thm:noholcomm}
 Let $\Gamma'$ denote  a subgroup of ${\rm Aut}_0(\Delta)$  with a finite
number of $\Gamma'$--orbits on the set of chambers of $\Delta$. Assume
that $\Gamma'$ has a finite index subgroup  with no holonomy. Then
$\Gamma'$ is commensurable to $\Gamma$ in ${\rm Aut}_0(\Delta)$.
\end{thm}

\begin{proof}
We may assume that $\Gamma'$ itself has no holonomy.
 By \fullref{prop:noholatlas}, there exists an atlas $\mathcal A$
on $\Delta$ which is invariant under $\Gamma'$. By
\fullref{cor:conjugatlas}, ${\rm Aut}_0(\Delta,{\mathcal A})$ is
simply transitive on chambers. The index of $\Gamma'$ in ${\rm
Aut}_0(\Delta,{\mathcal A})$ is thus precisely the number of $\Gamma'$--orbits
in the set of chambers of $\Delta$. Hence by assumption, $\Gamma'$ is of
finite index in ${\rm Aut}_0(\Delta,{\mathcal A})$.

 Again by \fullref{cor:conjugatlas}, the groups ${\rm
Aut}_0(\Delta,{\mathcal A})$ and $\Gamma$ are conjugate in ${\rm
Aut}_0(\Delta)$.
\end{proof}

\begin{proof}[Proof of \fullref{thm:commhol}]
Let $\Gamma'$ denote a uniform lattice of ${\rm Aut}_0(\Delta)$. If
$\Gamma'$ has a finite index subgroup with no holonomy then by
\fullref{thm:noholcomm}, it is commensurable to the graph
product $\Gamma$ in ${\rm Aut}_0(\Delta)$.

Conversely if $\Gamma'$ is commensurable with $\Gamma$, consider $f\in {\rm
Aut}_0(\Delta)$ where $\Gamma_1\subset \Gamma$ and $\Gamma'_1\subset\Gamma'$ are such
that $f\Gamma_1 f^{-1}=\Gamma'_1$. Then $\Gamma_0\cap \Gamma_1$ is a
finite index subgroup of $\Gamma$ without holonomy, thus $f(\Gamma_0\cap
\Gamma_1)f^{-1}$ is a finite index subgroup of $\Gamma'$ without holonomy.
\end{proof}

\begin{proof}[Proof of \fullref{cor:commL}] (Leighton's Lemma for biregular
graphs)\qua
Let ${X}_1$ and $X_2$ denote two biregular graphs of valencies $p$ and $q$, respectively. Then $\pi
_1X_k$ are uniform lattices of the biregular tree $T$ of valencies
$p$ and $q$. In fact the first barycentric subdivision of $T$ is the
right-angled building associated to the free product of
$G_i=\unfrac{\integers}{p\integers}$ and  $G_j=\unfrac{\integers}{q\integers}$. Observe that the stabilizer of a $\iperp$-- or a
$j^{\perpeq}$--residue in $\pi
_1X_k$ is trivial, because in this case $\iperp=\{i\!\}$ and 
$j^{\perpeq}=\{j\}$. In particular $\pi
_1X_k$ has no holonomy and we may apply \fullref{thm:noholcomm}:
either of the fundamental groups are commensurable with the free product
$\Gamma=\unfrac{\integers}{p\integers}*\unfrac{\integers}{q\integers}$.
\end{proof}

\begin{proof}[Proof of \fullref{cor:commJS}]
Fix a graph $\mathcal G$ with vertex set $T$. Choose two families of groups
$(G_i)_{i\in I},(G'_i)_{i\in I}$, such that $\size{G_i}=\size{G'_i}=q_i$ for each $i\in I$.
Consider the graph-product $\Gamma$ of $\{G_i\}_{i \in I}$ along $\mathcal G$, 
and set similarly $\Gamma'=\Gamma({\mathcal G},(G'_i)_{i\in I})$. Let
$\Delta$ and $\Delta'$ denote the right-angled building associated to the graph
products $\Gamma$ and $\Gamma'$.

Then $\Delta$ and $\Delta'$ are both  ``regular with parameters  $(q_i)_{i\in
I}$'', in the sense that each vertex of type $\{i\!\}$ is contained in
$q_i$  chambers. Thus by Proposition 5.1 of Haglund and Paulin \cite{HaglundPaulin03}, the
buildings
$\Delta$ and $\Delta'$ are isomorphic.

This provides an embedding of $\Gamma'$ inside  ${\rm Aut}_0(\Delta)$.
Now $\Gamma'$ is virtually without holonomy by \fullref{exmp:gamma0},
thus by \fullref{thm:noholcomm}, the group $\Gamma' \subset {\rm
Aut}_0(\Delta)$ is  commensurable with $\Gamma$.
\end{proof}

\begin{rem}
In the proof above we have in fact $\Gamma_0\simeq\Gamma'_0$ because one
can check directly that the cube complexes $\Gamma_0\backslash\Delta $ and
$\Gamma'_0\backslash\Delta
$ are  isomorphic.
\end{rem}

\section{Killing the holonomy}\label{sec:killhol}

In this section we  assume that $\mathcal G$ is finite and all groups
$G_i$ are finite, so that $\Delta$ is locally compact (see
\fullref{cor:chamberneighbor}). Let
$q_i$ denote the cardinality of $G_i$.

\begin{defn}\label{defn:separable}
 Let $G$ denote any group and let $H$ denote some subgroup of $G$. Then
$H$ is said to be \textit{separable in $G$\/} whenever the intersection of
finite index subgroups of $G$ containing $H$ reduces to $H$.
\end{defn}

\begin{thm}\label{thm:separcommens}
 Let $\Gamma'$ denote some subgroup of ${\rm Aut}_0(\Delta)$, such that
there are only finitely many $\Gamma'$--orbits in the set of
$\iperp$--residues for each $i\in I$. Assume that for each
$\iperp$--residue $R$ any  finite index subgroup of $\Gamma'_R$ is
separable in $\Gamma'$. Then $\Gamma'$ has a finite index subgroup with
no holonomy.
\end{thm}

\begin{proof}
 Suppose we have  fixed $i\in I$ and chosen some $\iperp$--residue
$R$. Consider the following two properties:

\begin{enumerate}
 \item[$(P_1)$] There exists a finite index subgroup $\Gamma''\subset
\Gamma'$ such that the $i$--holonomy of $\Gamma''$ at $R$ is trivial.

 \item[$(P_2)$] There exists a finite index subgroup $\Gamma''\subset
\Gamma'$ such that for every $\gamma'\in\Gamma'$ the $i$--holonomy of
$\Gamma''$ at $\gamma' R$ is trivial.
\end{enumerate}

 We first show  $(P_1)\Rightarrow (P_2)$. 
 Consider first a finite
index subgroup $\Gamma''$ such that  the $i$--holonomy of $\Gamma''$ at
$R$ is trivial. Up to replacing $\Gamma''$ by the intersection of its
conjugates in $\Gamma'$ we may assume that $\Gamma''$ is normal. Then its
$i$--holonomy at $\gamma' R$ is trivial for any $\gamma'\in\Gamma'$.

 This is because by the normality assumption $\gamma'$ conjugates
$\Gamma''_R$ onto $\Gamma''_{\gamma'\!R}$, and equivariantly conjugates
the $i$--holonomy representations at $R$ and $\gamma'R$.

 Now property $(P_2)$ for each pair $(i,R)$ implies the theorem. Indeed
choose a finite set $R_1,\dots,R_n$ of representative of
$\iperp$--residues, $i\in I$, modulo $\Gamma'$. Applying $(P_2)$  we
find for each $1\le j\le n$ a finite index subgroup ${\Gamma''}_j$ such
that the $i_j$--holonomy of ${\Gamma''}_j$ at any translate $\gamma' R_j$
is trivial. Then the finite intersection $\bigcap_j{\Gamma''}_j$ is a finite
index subgroup with no holonomy.

So it remains to prove $(P_1)$.
 Let $K_R$ denote the kernel of the $i$--holonomy representation
${\Gamma'}_R\to {\mathfrak S}(\pi_0(\partial_i R))$. Note that $K_R$ is a
finite index subgroup of ${\Gamma'}_R$ because $G_i$ is finite. Thus
there is a family $({\Gamma'}_s)_{s\in S}$ of finite index subgroups of
$\Gamma'$ whose intersection is $K_R$. Let $h_0=1,h_1,\dots,h_n$ denote a
set representatives of ${\Gamma'}_R$ modulo $K_R$. Then there exists
$s_1,\dots,s_n$ such that $h_j\not\in {\Gamma'}_{s_j}$.

 Set $\Gamma''=\cap _j {\Gamma'}_{s_j}$. This a finite index subgroup
because the intersection is finite. Furthermore $K_R\subset \Gamma''$. We
claim that in fact ${\Gamma'}_R\cap \Gamma'' =K_R$. For  $h\in
{\Gamma'}_R\setminus K_R$ we have $h=ks_j$ for some $1\le j\le n$ and
some $k\in K_R$. We cannot have $h\in  \Gamma''$, else as $k\in 
\Gamma''$ we would have $s_j\in  \Gamma''$, absurd.

 Now the $i$--holonomy of $\Gamma''$ at $R$ is trivial because
${\Gamma''}_R=K_R$.
\end{proof}

\begin{proof}[Proof of \fullref{thm:commsepar}]
By \fullref{prop:cat0}, residues are convex subcomplexes, thus
their stabilizers are quasiconvex subgroups, as well as their finite
index subgroups. So \fullref{thm:separcommens} applies, and to finish
we apply \fullref{thm:noholcomm}.
\end{proof}

\part{II\qua Virtual triviality of some extensions by finite groups}

\section{Polygonal complexes}\label{sec:polycompl}

\begin{defn}[Polygonal complexes]\label{defn:polycomplex}
Consider a discrete set $V$. Choose a family
of maps $\{e_{\alpha}\co \{0,1\}\to V\}_{\alpha
\in A}$ and use these maps to glue copies of
$[0,1]$ along its boundary (one copy for
each $e_{\alpha}$). This produces  a \textit{graph\/}
$\mathcal G$. The maps $e_{\alpha}$ naturally
extend to maps $e_{\alpha}\co [0,1]\to {\mathcal
G}$, called \textit{the basic oriented edges of
$\mathcal G$}. A map
$f\co {\mathcal G}\to {\mathcal
G}'$ between two graphs is \textit{combinatorial\/}
whenever for each $\alpha$ there is an
$\alpha'$ such that $f\circ
e_{\alpha}={e'}_{\alpha'}\circ i$ where $i$ is
an isometry of $[0,1]$.

Note that each simple polygonal line in
Euclidean space has a natural structure of
graph for which the basic oriented edges are
(restrictions of) affine maps.

Now consider a family $\{\pi_{\beta} \}
_{\beta\in B}$ of convex polygons of
$\Euclidean ^2$, together with a family of
combinatorial maps $f_{\beta}:\partial\pi
{_\beta}\to {\mathcal G}$. We may glue the
polygons $\pi
{_\beta}$ to the graph $\mathcal G$ along
their boundaries using the maps $f_{\beta}$. The
resulting space $X$ is \textit{a polygonal
complex}.  Then each map $f_{\beta}:\partial\pi
{_\beta}\to {\mathcal G}$ naturally   extends
to a map
$f{_\beta}\co  \pi{_\beta}\to X$ (\textit{the
basic oriented polygons of
$X$}).

The set \textit{$V$ is the set of vertices\/} (of
$\mathcal G$ and of $X$). The graph $\mathcal G$ is the \textit{1--skeleton of
$X$}.
A map $f\co X\to X'$ is \textit{combinatorial\/} if
$f({\mathcal G})\subset {\mathcal G}'$, the restriction
$f\vert_{{\mathcal G}}$ is a combinatorial
map of graphs and for each oriented polygon
$f_{\beta}\co \pi_{\beta}\to X$ we have $f\circ
f_{\beta}={f'}_{\beta'}\circ \varphi$, where
$\varphi\co \pi_{\beta}\to {\pi'}_{\beta'}$ is the
restriction of an affine isomorphism.
We will denote by ${\rm
Aut}(X)$ the group of bijective combinatorial maps $X\to X$.

\textit{An oriented edge of
   $X$} is a combinatorial map
$\overrightarrow a\co [0,1]\to X$. The
associated edge
$a$ is the image of $\overrightarrow a$. The
initial point of an oriented edge
$\overrightarrow a$ is $\overrightarrow
a(0)$, denoted by $i(\overrightarrow{a})$;
its terminal point is
$\overrightarrow a(1)$, denoted by
$t(\overrightarrow{a})$. The inverse of an
oriented edge  $\overrightarrow
a\co [0,1]\to X$ is the map $\overleftarrow
a\co [0,1]\to X$ defined by $\overleftarrow
a(u)=\overrightarrow
a(1-u)$. \textit{An oriented polygon of
   $X$} is a combinatorial map
$f\co \pi \to X$ where $\pi $ denotes some
Euclidean convex polygon. The associated
polygon is the image of $f$.
    Note
that the edges and polygons of
$X$ are the images of the basic oriented edges
and polygons.
   For each
   polygon $\pi $ of $X$ parametrized by the
basic oriented polygon
$f_{\beta}\co \pi_{\beta}\to X$ we denote by
$k_{\pi}
$ the number of edges in $\partial\pi_{\beta}$.

We say that a polygonal complex is \textit{regular\/} if for each $\beta\in B$ there is an
affine isomorphism  sending the
polygon $\pi_{\beta}$ to a regular Euclidean
polygon.
\end{defn}

\begin{defn}[Paths and cycles]
\textit{A (combinatorial) path\/} of a polygonal
complex
$X$ is a sequence
$c=(\overrightarrow{a_1},
\dots,\overrightarrow{a_n})$ of oriented edges
such that for each integer $0\le i<n$ we have
$t(\overrightarrow{a_i})=
i(\overrightarrow{a_{i+1}})$. {\rm A cycle of
$X$} is a path $(\overrightarrow{a_1},
\dots,\overrightarrow{a_n})$, $n>0$, such
that
$t(\overrightarrow{a_n})=
i(\overrightarrow{a_{1}})$ and
$\overleftarrow{a_i}\neq
\overrightarrow{a_{i+1}}$.
\end{defn}

\begin{defn}[Barycentric subdivision]
Let $X$ denote a polygonal complex. For any
oriented edge $\overrightarrow a$ of $X$
\textit{the center of $\overrightarrow a$\/} is the
point
$\overrightarrow a(\unfrac{1}{2})$. For any
oriented polygon $f_{\beta}\co \pi_{\beta}\to X$
\textit{the center of $f_{\beta\/}$} is the image
under
$f_{\beta}$ of the barycenter of $\pi_{\beta}$.
Note that these points depend only on the
images of the oriented edges or polygons.

We set $V_0=V$, the set of vertices of $X$.
We denote by $V_1$ the set of centers of
oriented edges and by $V_2$ the set of centers of
oriented polygons. We denote by $E_{01}$
the set of restrictions of basic oriented
edges of
$X$ to
$[0,\unfrac{1}{2}]$ or $[\unfrac{1}{2},1]$. We
denote by $F$
the set of restrictions of  basic oriented
polygons
$f_{\beta}$ to triangle with vertices $v,m,p$,
with $p$ the barycenter of $\pi_{\beta}$, $m$
the barycenter of an edge $e$ of $\pi_{\beta}$
and $v$ a vertex of the edge $e$ (these maps
will be called \textit{basic triangles\/}). With
the same notation we denote by
$E_{02}$ the set of restrictions of
$f_{\beta}$ to segments of the form $[vp]$,
and  we denote by $E_{12}$ the set of
restrictions of
$f_{\beta}$ to segments of the form $[mp]$.

The barycentric subdivision of $X$ is the
multisimplicial complex with set of vertices
$V_0\sqcup V_1\sqcup V_2$, set of edges $E_{01}
\sqcup E_{02}\sqcup E_{12}$ and set of
(basic) triangles $F$. This multisimplicial
complex will be denoted by $X'$. It is
naturally homeomorphic to $X$.
\end{defn}

\begin{defn}[Rank of vertices]
For each vertex $v$ of the barycentric
subdivision $X'$ of a polygonal complex $X$
we set ${\rm rk}(v)=i\iff v\in V_i$ for
$i\in \{0,1,2\}$. This integer we call
\textit{the rank of $v$\/}. For every vertex $p$
of rank 2 we will set $k_p=k_{\pi} $, where
$\pi $ is the polygon whose center is $p$.

   We
will denote by
${\rm Aut}_{\rm rk}(X')$ the group of
automorphisms of $X'$ which preserve the rank.
\end{defn}

\begin{rem}\label{rem:autreg}
Let $X$ denote a regular polygonal complex.
Then the natural inclusion ${\rm
Aut}(X)\to {\rm
Aut}_{\rm rk}(X')$ is an isomorphism.
\end{rem}

\begin{defn}[Links]
Let $X$ denote a polygonal complex, and let
$v$ denote a vertex of $X$. Then ${\rm link}(v,X')$
is the barycentric subdivision of a unique
graph, denoted by ${\rm link}(v,X)$. Vertices of
${\rm link}(v,X)$ correspond to oriented edges of $X$
with initial point $v$, and edges of
${\rm link}(v,X)$ correspond to polygons of $X$
containing $v$. For each edge $e$ of
${\rm link}(v,X)$ we will set $k_e =k_{\pi} $, where
$\pi $ is the polygon of $X$ containing $v$
and corresponding to $e$.
\end{defn}

\begin{defn}[Nonpositive curvature conditions]
We consider four kinds of nonpositive
curvature conditions on a polygonal complex
$X$ with vertex set $V$.
\begin{enumerate}\leftskip 2pt
\item[(Q)]
For all $v\in V$, if $c=(\overrightarrow{e_1},\dots,
\overrightarrow{e_n})$ is a cycle of ${\rm link}(v,X)$
then $n\ge 4$, and if 
$\pi$ is a polygon then $k_{\pi} \ge 4$.

\item[($\mathrm{C^4}$)] For all $v\in V$, if $
c=(\overrightarrow{e_1},\dots,\overrightarrow{e_n})$ is a cycle
of ${\rm link}(v,X)$ then $\sum_i (\frac{1}{2}-\frac{1}{4f_{e_i}})\ge 1$
where $f_{e_i}$ is the integer such that
$k_{e_i}-4f_{e_i}\in\{0,1,2,3\}$.

\item[($\mathrm{C^2}$)] For all $v\in V$, if 
$c=(\overrightarrow{e_1},\dots,\overrightarrow{e_n})$ is a cycle
of  ${\rm link}(v,X)$ then 
$\sum_i (\frac{1}{2}-\frac{1}{2t_{e_i}})\ge 1$
where $t_{e_i}$ is the integer such that
$k_{e_i}-2t_{e_i}\in\{0,1\}$.

\item[(C)] For all $v\in V$, if $c=(\overrightarrow{e_1},\dots,
\overrightarrow{e_n})$ is a cycle of ${\rm link}(v,X)$ then
$\sum_i (\frac{1}{2}-\frac{1}{k_{e_i}})\ge 1$.
\end{enumerate}
There are also analogous conditions
($\mathrm{Q}'$), ($\mathrm{C}'$), ($\mathrm{C^2}'$), ($\mathrm{C^{4}}'$) obtained by
replacing
$n\ge 4$ or $\sum \ge 1$ in (Q), (C), ($\mathrm{C^2}$), ($\mathrm{C^4}$) by  $n> 4$ or
$\sum> 1$.

Clearly (Q) $\Rightarrow (\mathrm{C^4})
\Rightarrow (\mathrm{C^2}) \Rightarrow$ (C), and $(\mathrm{Q}') \Rightarrow 
(\mathrm{C^4}')
\Rightarrow (\mathrm{C^2}') \Rightarrow (\mathrm{C}')$.
\end{defn}

\begin{defn}[Piecewise Euclidean metrics]
We define four types of piecewise Euclidean
metrics on (the barycentric
subdivision of) our polygonal complexes.

Let $X$ denote a polygonal complex. For each
basic triangle $f\co T\to X'$ of the barycentric
subdivision let $v,m$ and $p$ denote the vertices
of $f(T)$ with ranks $0, 1$ and $2$. There exist a unique
affine automorphism $\varphi\co \Euclidean ^2\to
\Euclidean ^2$ such that the triangle
$T_0=\varphi^{-1}(T)$ has the following
characteristics: the angle of $T_0$ at
$p_0=\varphi^{-1}(p)$ is $\unfrac{\pi}{k_p}$,
the angle  of $T_0$ at
$m_0=\varphi^{-1}(m)$ is $\unfrac{\pi}{2}$ the
length of the edge $[v_0m_0]$ is $\unfrac{1}{2}$ where
$v_0=\varphi^{-1}(v)$. We then consider the new triangles $f\circ\varphi$
and observe that two such triangles differ on
an edge of
$X'$ by an isometry. Thus we have defined a
structure of $M_0$--polyhedral complex on $X'$
in the sense of
\cite[Definition 7.37]{BridsonHaefliger}. We
will call the associated length metric on
$X$ the (C)--metric.

If no polygon of $X$ is a triangle, we may
similarly define on $X$ a
($\mathrm{C^4}$)--metric, a $(\mathrm{C^2})$--metric and (Q)--metric.
For the ($\mathrm{C^4}$)--metric, we
demand that the angle of the Euclidean
triangle $T_0$ at $p_0$ be $\unpfrac{\pi}{4f_p}$
with $f_p$  
the integer defined by $k_p-4f_p\in\{0,1,2,3\}$.
For the $(\mathrm{C^2})$--metric, we
demand that the angle of the Euclidean
triangle $T_0$ at $p_0$ be $\unpfrac{\pi}{2t_p}$
with $t_p$ the integer defined by
$k_p-2t_p\in\{0,1\}$.
For the (Q)--metric, we
demand that the angle of the Euclidean
triangle $T_0$ at $p_0$ be $\unfrac{\pi}{4}$.
\end{defn}

\begin{lem}\label{lem:condcat0}
If a locally compact polygonal complex $X$ satisfies
condition $(\mathrm C)$ then the $(\mathrm C)$--metric on $X$
is locally $\CAT(0)$. The same is true if we replace $(\mathrm C)$ by
$(\mathrm{C^4})$, $(\mathrm{C^2})$ or $(\mathrm Q)$.

If a compact polygonal complex $X$ satisfies 
condition $(\mathrm Q')$, $(\mathrm C')$, $(\mathrm{C^2}')$ or $(\mathrm{C^4}')$ then $\pi_1(X)$ is 
Gromov-hyperbolic.
\end{lem}

\begin{proof}
In each case $X$ has (locally) finitely many isomorphisms type of
polygons and the  systole of the  metric  link of vertices is at least $2\pi$
(computations are left to the reader). Thus  the associated
length metric is locally $\CAT(0)$ by 
\cite[Theorem 5.2 and Lemma 5.6, pages 206--207]{BridsonHaefliger}.

Under conditions $(\mathrm Q')$, $(\mathrm C')$, $(\mathrm{C^2}')$ 
or $(\mathrm{C^4}')$,  the 
systole of the  metric link of rank--0 vertices is  greater than $2\pi$. So we may
replace our Euclidean triangles by small enough hyperbolic triangles with
the same angles at rank--1 and rank--2 vertices. This defines a structure
of piecewise hyperbolic simplicial complex in which the link condition is
still fulfilled.  Thus  the associated
length metric is locally $\CAT(-1)$ \cite{BridsonHaefliger}, and so $\pi _1(X)$ is Gromov-hyperbolic by
Proposition 1.2 of \cite{BridsonHaefliger}.
\end{proof}

\begin{defn}
Let $X$ denote a polygonal complex. Let $\Pi=\Pi(X)$ denote
the set of 2-cells of $X$, and denote by $V=V(X)$ the set of vertices
of $X$ and $E=E(X)$ the set of edges.

An \textit{orientation\/} of an edge
$e$ is the choice of an oriented edge with image $e$. Note that each edge
has two orientations. We choose once and for all an orientation for each
edge of
$X$, using the unique basic oriented edge with image $e$. If an oriented
edge $\ovae$ with image $e$ is the chosen oriented edge we
set
$\eta(\!\ovae\!)=1$, and if it is the opposite edge we set
$\eta(\!\ovae\!)=-1$.

  We also  choose
once and for all for  each  basic oriented polygon $f\co \pi_{\beta}\to X$ an
orientation of each edges of
$\partial
\pi_{\beta} $, in such a way that the terminal point $v$ of an oriented edge
of
$\partial
\pi_{\beta} $ is the initial point of the second oriented edge of $\partial
\pi_{\beta} $ whose image contains $v$. For each vertex
$v\in\partial\pi_{\beta}$ there is a unique cycle $\delta_v\pi_{\beta}=(\overrightarrow{e_1},\dots,\overrightarrow{e_n})$ using the
chosen oriented edges, with $i(\overrightarrow{e_1})=v$, and such that $n$ is the number of sides of $\pi_{\beta}$.  For  each polygon
$\pi
$ of
$X$, $\pi =f_{\beta}(\pi_{\beta})$, and each vertex $v$ of $\pi_{\beta}$,  we
set
$\delta_v\pi =
(f_{\beta}\circ
\overrightarrow{e_1},\dots,f_{\beta}\circ\overrightarrow{e_n})$, where
$(\overrightarrow{e_1},\dots,\overrightarrow{e_n})=\delta_v\pi_{\beta}$.

Let
$\field A$ denote some finite abelian group. A \textit{2--cocycle\/} on $X$ (with
values in
$\field A$) is just a map $c\co \Pi\to \field A$ assigning to each 2-cell $\pi $ of  $X$
an element
$c(\pi )\in \field A$. A \textit{1--cochain\/} on
$X$ (with values in $\field A$) is a map $u\co E\to \field A$. The coboundary of a
1--cochain $u$ is the 2--cocycle $\delta u$ assigning to each
2-cell $\pi $ the sum $\tsty \sum \eta(\!\ovae\!)u(e)$ over $\ovae\in\delta_v\pi
$. Note that this element of $\field A$ is
well-defined independently of the choice of $v$ on $\partial \pi_{\beta}$,
because $\field A$ is abelian (this is why we are using additive notation).
The \textit{support\/} of a 1--cochain $u$ is the set of edges on which $u$
does not vanish; it will be denoted by ${\rm Supp}(u)$.

Let $\varphi$ denote an automorphism of $X$. 
For any edge $e$ we set $\eta(\varphi,e)=1$
whenever $\varphi$ sends the basic oriented
edge at $e$ onto the basic oriented edge at 
$\varphi(e)$. Otherwise $\varphi$ sends the
basic oriented edge at $e$ to the inverse of
the basic oriented edge at $\varphi(e)$ and we
set $\eta(\varphi,e)=-1$.  For $u$  a 1--cochain
and $e$ an edge  we set
$\varphi(u)(e) =\eta(\varphi,
\varphi^{-1}(e))u(\varphi^{-1}(e))$. We thus
obtain an action of ${\rm Aut}(X)$ on the
additive group of 1--cochains.

Similarly for any polygon $\pi$ of $X$ we set $\eta(\varphi,\pi)=1$ whenever $\varphi$ sends the basic oriented polygon at $\pi$ onto the basic oriented polygon at $\varphi(\pi)$. Otherwise we set $\eta(\varphi,\pi)=-1$. Then for any 2--cocycle $c$ and any polygon $\pi$ we set 
$\varphi(c)(\pi)$ equal to $\eta(\varphi,\varphi^{-1}(\pi))c(\varphi^{-1}(\pi))$. We thus obtain an action of ${\rm Aut}(X)$ on the additive group of 2--cocycles.
\end{defn}

\begin{exmp}\label{exmp:calculor}
Let $\varphi$ denote an automorphism of $X$.  Let $\pi$ be a polygon of $X$ and let $\ovae$ denote some oriented edge of $\delta_v\pi$ for any vertex $v\in\partial \pi$. Set $\pi '=\varphi(\pi)$, $v'=\varphi(v)$ and let $\ovae\!'$ be the oriented edge of $\delta_{v'}{\pi'}$ which has the same image  as  $\varphi\circ \ovae$. Then 
$$\eta(\!\ovae\!')=\eta(\varphi,\pi) \eta(\varphi,e) \eta(\!\ovae\!).$$
In particular we have
$(\varphi(c)+\delta(\varphi u))(\pi')
=\eta(\varphi,\pi)((c+\delta u)(\pi))$ for any 2--cocycle  $c$ and any
1--cochain $u$. Also
$\delta(\varphi u) =\varphi(\delta u)$.
\end{exmp}

\section{Walls}\label{sec:walls}

\begin{defn}
Let $X$ denote some  polygonal complex,  and
let $\pi$ denote some polygon in $X$,  not a
triangle. We say
that two distinct oriented edges
$\overrightarrow{a},\overrightarrow{a}'$
contained in $\pi$ are \textit{parallel in
$\pi$\/} whenever there are  disjoint connected
subgraphs
$\sigma,\sigma'$  of
$\partial\pi$ satisfying the following:
\begin{enumerate}
\item the boundary $\partial\pi$ is  the
union $a\cup \sigma\cup a'\cup \sigma'$
\item if $\ell,\ell'$ denote the numbers of edges of
$\sigma,\sigma'$, then
$\ell\le\ell'< \ell+2$
\item the terminal vertices of  the oriented
edges
$\overrightarrow{a},\overrightarrow{a}'$
belong to $\sigma$, so their initial vertices
belong to $\sigma'$.
\end{enumerate}

If such subgraphs exist then they
are unique. Hence any oriented  edge
$\overrightarrow{a}$ of a polygon $\pi$ is
parallel to exactly one other oriented
edge $\overrightarrow{a}'   $ of $\pi$. We denote
by $\parallel$ the equivalence relation on
oriented edges generated by parallelism
inside polygons. Note that the group ${\rm
Aut}(X)$  preserves the relation
$\parallel$.\end{defn}

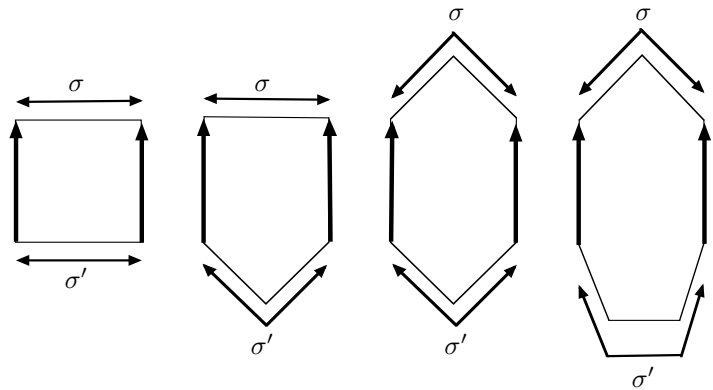
\begin{figure}[ht!]
\labellist
\small\hair 2pt
\pinlabel $\sigma$ at 135 450
\pinlabel $\sigma'$ at 135 275
\pinlabel $\sigma$ at 304 450
\pinlabel $\sigma'$ at 304 215
\pinlabel $\sigma$ at 480 515
\pinlabel $\sigma'$ at 480 215
\pinlabel $\sigma$ at 650 515
\pinlabel $\sigma'$ at 650 185
\endlabellist
\centerline{\includegraphics[width=11cm]{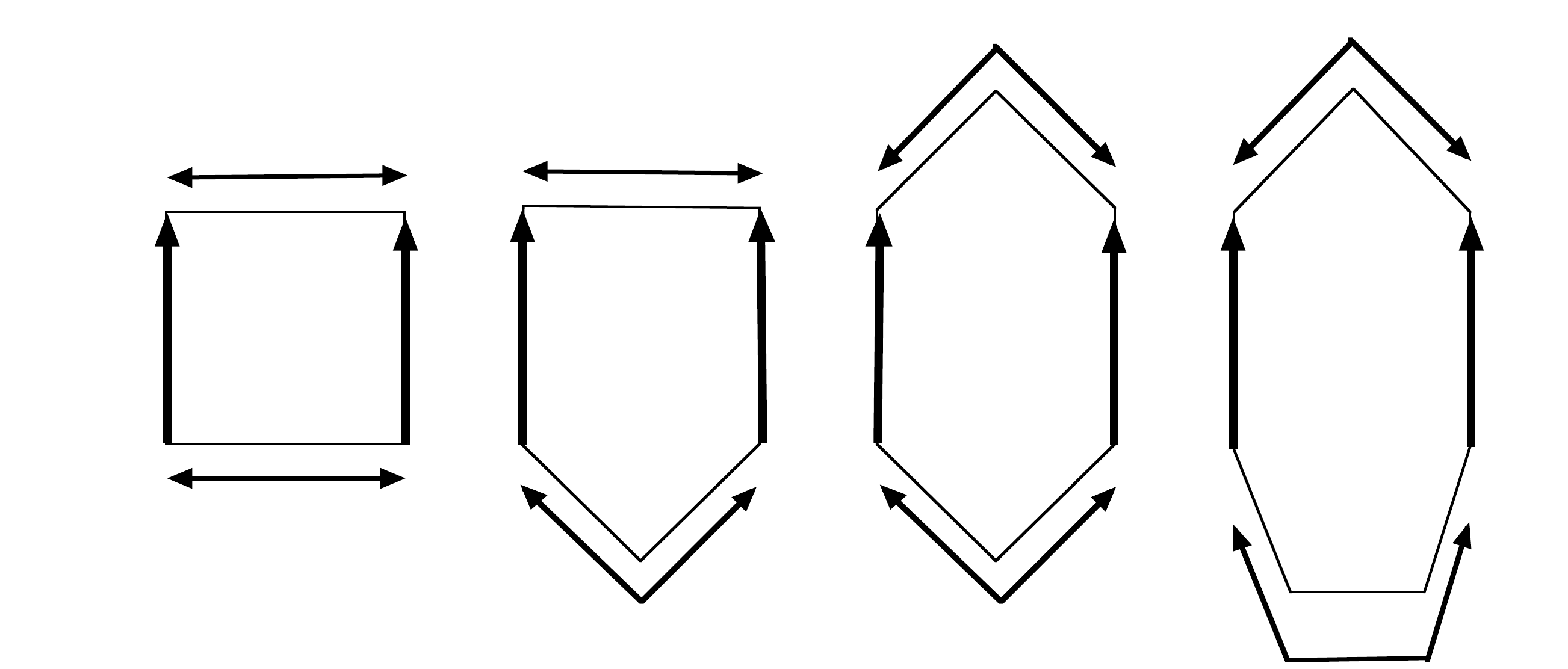}}
\caption{Two parallel oriented edges inside a polygon}
\end{figure}

\begin{exmp}
   Let $k$ denote the number of edges of
the polygon $\pi$, and assume $k\ge 4$.
We consider parallelism inside $\pi$.

If $k$ is even then two oriented edges
$\overrightarrow{a},\overrightarrow{a}'$
contained in $\pi$ are { parallel in
$\pi$} whenever they induce opposite
orientations on $\partial\pi$ and the
associated geometric edges $a,a'$ are opposite
on $\partial\pi$. Note that this is the 
parallelism relation defined in \cite{HaglundWise05}
for square complexes.  See also  Haglund and Paulin \cite{HaglundPaulin98}
and Wise \cite{WiseSmallCanCube04}.

If $k$ is odd, and so $k=2t+1$, then $\ell=t-1$ and 
$\ell'=t$. Thus two oriented edges
$\overrightarrow{a},\overrightarrow{a}'$
contained in $\pi$ are { parallel in
$\pi$} whenever they induce opposite
orientations on $\partial\pi$ and 
$i(\overrightarrow{a}')$ is the   vertex of
$\pi$ opposite to the center of
$\overrightarrow{a}$.
\end{exmp}

\begin{defn}[\rm Walls]\label{defn:wall}
Let $X$ denote a polygonal complex. A \textit{
wall\/} of $X$ is an equivalence class of
$\parallel $.

We say that \textit{a wall $M$
passes through an oriented edge
$\overrightarrow a$\/}  whenever
$\overrightarrow a\in M$. We say that \textit{a
wall
$M$ separates a polygon $\pi $ of $X$\/}
whenever $M$ passes through distinct oriented edges
$\overrightarrow a,\overrightarrow{a}'$ with
$a\subset \partial\pi,a'\subset \partial\pi$ and $\overrightarrow
a,\overrightarrow{a}'$ are parallel inside $\pi $.

A \textit{diameter\/} of a polygon $\pi $ of $X$ is the union of two edges
$[v,p],[v',p]$ of $X'$ such that $p$ is the center of $\pi $, and there
are distinct oriented edges $\overrightarrow a,\overrightarrow{a}'\subset
\partial\pi$ which are parallel in $\pi $, and whose centers are $v,v'$.
The diameter is then said to be \textit{dual to $\overrightarrow
a,\overrightarrow{a}'$\/}.

\textit{The geometric wall associated to a wall
$M$\/} is the subcomplex $\size{M\!}$ of the
first barycentric subdivision $X'$, union of those diameters of polygons
which are dual to $\overrightarrow
a,\overrightarrow{a}'$, with $\overrightarrow
a,\overrightarrow{a}'\in M$.

  If some wall $M$ passes through an oriented
edge
$\overrightarrow a$, we will also  say that
\textit{$M$ is dual to
$\overrightarrow a$\/}, and that \textit{$M$ is dual to
$a$\/}. For any wall $M$ we will
denote by $V(M)$ the set of vertices $v$
belonging to an edge $a$  to which $M$ is dual.

  Note that a combinatorial map $f\co X\to Y$ 
maps a wall of $X$ inside a wall of $Y$. The
same is true for the associated
geometric wall.
\end{defn}

\begin{figure}[ht!]
\centering
\includegraphics[width=11cm]{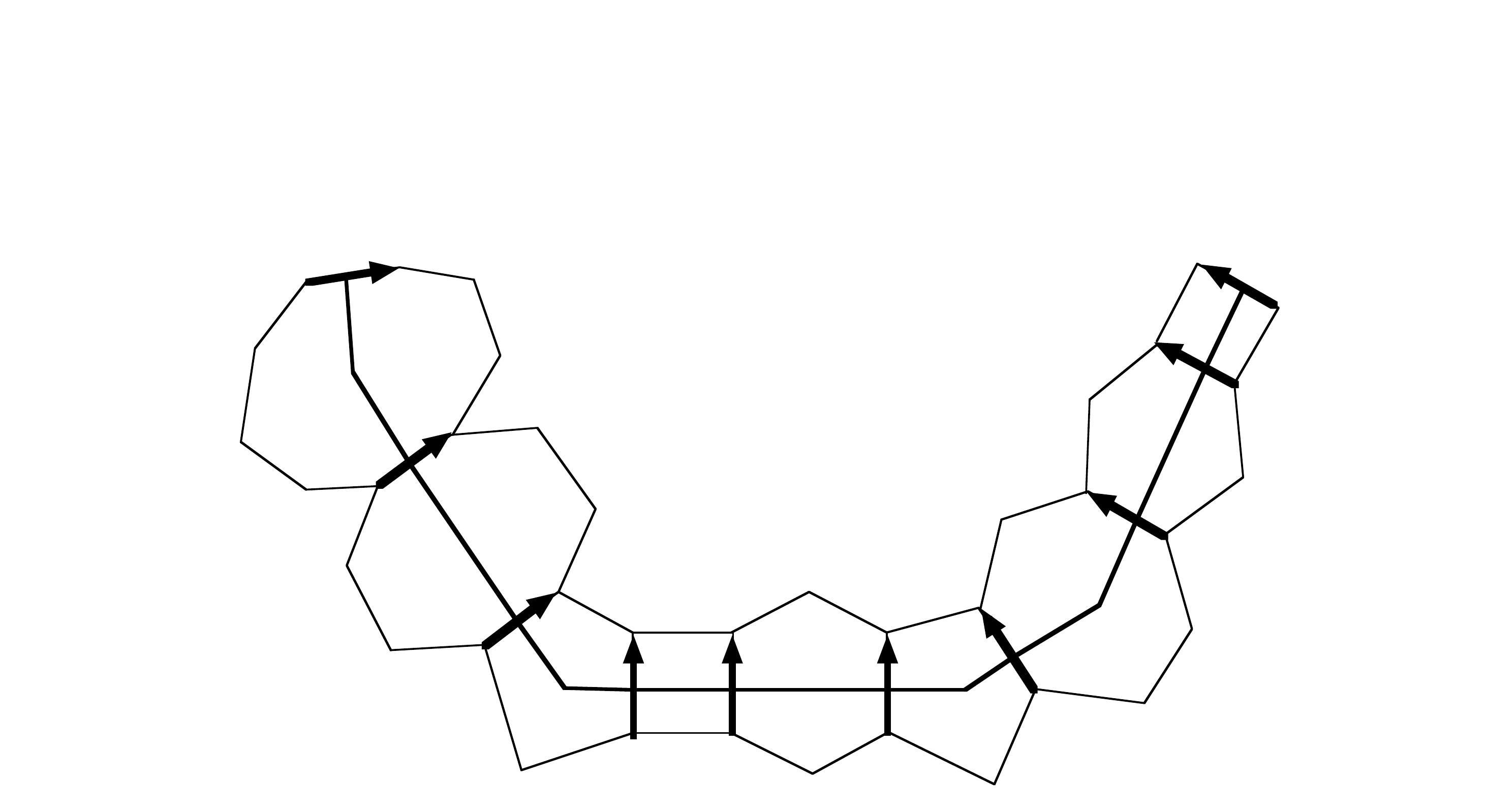}
\caption{A wall and its associated geometric wall}
\end{figure}

\begin{lem}\label{lem:walltree}
 Assume that $X$ is locally compact, satisfies condition  $(\mathrm{C^2})$  and is
simply connected. Equip $X$ with the
$(\mathrm{C^2})$--metric. Then:
\begin{enumerate}
\item Each geometric wall is a totally geodesic subtree.

\item Each geometric wall separates $X$ into two connected components.

\item Two opposite oriented edges are not parallel.

\item The intersection of a geometric wall and a polygon is either empty
or a single diameter. In particular a geometric wall $\size{M\!}$ is the
first barycentric subdivision of a tree whose vertices correspond to rank--1 
vertices of $\size{M\!}$, and whose centers of edges  correspond to rank--2 vertices 
of $\size{M\!}$.
\end{enumerate}
\end{lem}

\begin{proof}
Consider any diameter $d=[v,p]\cup [p,v']$ of some polygon of $X$. By
definition of the $(\mathrm{C^2})$ metric the angle of the geodesic segments
$[p,v],[p,v']$ at $p$ is $\pi $. Thus diameters of polygons are local
geodesics.

Let $M$ denote some wall of $X$. By definition of $\parallel$, for any
two points $x,y\in\size{M\!}$ there is a sequence of diameters
$d_0,d_1,\dots,d_n$ and a sequence of polygons $\pi _0,\pi _1,\dots,\pi
_n$ such that $d_i\subset \pi _i$, $x\in d_0,y\in d_n$, and $d_i\cap d_{i+1}$ contains a rank--1 vertex. After
simplifications, we may assume that $d_i\neq d_{i+1}$. Then $d_i\cap
d_{i+1}=\{v_{i+1}\}$, where $v_{i+1}$ is the center of an edge
$\overrightarrow a_{i+1}$ to which $M$ is dual. Observe that the angle of
$d_i$ and $d_{i+1}$  at $v_{i+1}$ is ${\pi}{}$, so that
$(d_i,d_{i+1})$  is a local geodesic.

It follows that $(d_0,d_1,\dots,d_n)$ is a local geodesic. By
\fullref{lem:condcat0} the $(\mathrm{C^2})$ length metric on $X$ is $\CAT(0)$, so
a local geodesic is a global one \cite[Proposition
4.14, page 201]{BridsonHaefliger}.

We have shown that any two points of $\size{M\!}$ are contained in a
geodesic segment of $X$ sitting in $\size{M\!}$. Thus $\size{M\!}$ is totally
geodesic, and in particular it is contractible. Since $\size{M\!}$ is a
subgraph of
$X'$, it follows that $\size{M\!}$ is a tree.

Since each diameter of a polygon $\pi $ is separating $\pi $, we deduce
that $\size{M\!}$ is locally separating in two connected components. Since
$\size{M\!}$ is a tree we see that $\size{M\!}$ separates a neighborhood in
two connected components. But $X$ is simply connected, thus in fact
$\size{M\!}$ separates $X$ in
two connected components.

Assume now that $d,d'$ are  two distinct diameters of the same polygon
$\pi $ of $X$. Then $d,d'$ contain distinct edges $e,e'$ of $X'$ meeting
at $p$ with angle less than $\pi$. In particular the convex hull of $d\cup d'$ is
not contained in the 1--skeleton of $X'$. This shows that a totally
geodesic union of diameters of $X$ (such as $\size{M\!}$) does not contain
more than one diameter of each polygon.

Note that a parallelism between
$\overrightarrow a$ and its opposite $\overleftarrow a$ creates a curve
disjoint from $\size{M\!}$ and joining the endpoints of $\overrightarrow a$.
This contradicts the fact that $\size{M\!}$ separates a neighborhood.
\end{proof}

\section{Killing 2--cocycles}\label{sec:killcocycle}

\subsection{Groups acting without self-intersection}

\begin{defn}[Self-intersecting wall] Let
$X$ denote some simply connected polygonal
complex and let $\Gamma$ denote some subgroup
of ${\rm Aut}(X)$. We say that \textit{a wall $M$
is self-intersecting under $\Gamma$\/} whenever
there exists $\gamma\in \Gamma$ such that
$M,\gamma{M}$ are distinct,  but the  geometric
walls associated to $M,\gamma{M}$ have nonempty
intersection. We say that \textit{$\Gamma$ acts without self-intersection\/} whenever  no wall is
self-intersecting under
$\Gamma$.
\end{defn}
If $X$ is a  square complexes and $\Gamma$ acts
freely, this definition is equivalent to the
fact that $\Gamma\backslash X$ has no
self-intersecting wall and no one-sided wall in
the sense of \cite{HaglundWise05}.

\begin{lem}\label{lem:virtclean}
   Let $\Gamma$ denote a discrete cocompact
group of automorphism of some
simply connected polygonal  complex
$X$. Assume that the stabilizers of walls are
separable in $\Gamma$.

Then there is a  finite index subgroup 
$\Gamma'\subset\Gamma$ acting without
self-intersection.
\end{lem}

\begin{proof}
 Let $n$ denote the maximal number of sides
of a polygon of the complex $X$.

  For each wall $M$ of $X$  and
each subgroup $G\subset{\rm Aut}(X)$ we
introduce a set of bad elements:
$$B(M,G)=\{g\in G, gM\neq M\ {\rm and}\
\exists\ v,v' \in V(M), \ d(v',gv)\le n\}$$
where $d$ denotes the combinatorial distance
in the 1--skeleton.

\begin{figure}[ht!]
\labellist
\small\hair 2pt
\pinlabel $|gM|$ [bl] at 406 362
\pinlabel $gv$ [l] at 405 313
\pinlabel $\leq n$ [l] at 413 200
\pinlabel $v'$ [bl] at 430 130
\pinlabel $|M|$ [r] at -5 65
\pinlabel $v$ [tr] at 5 -3
\endlabellist
\centering
\includegraphics[width=11cm]{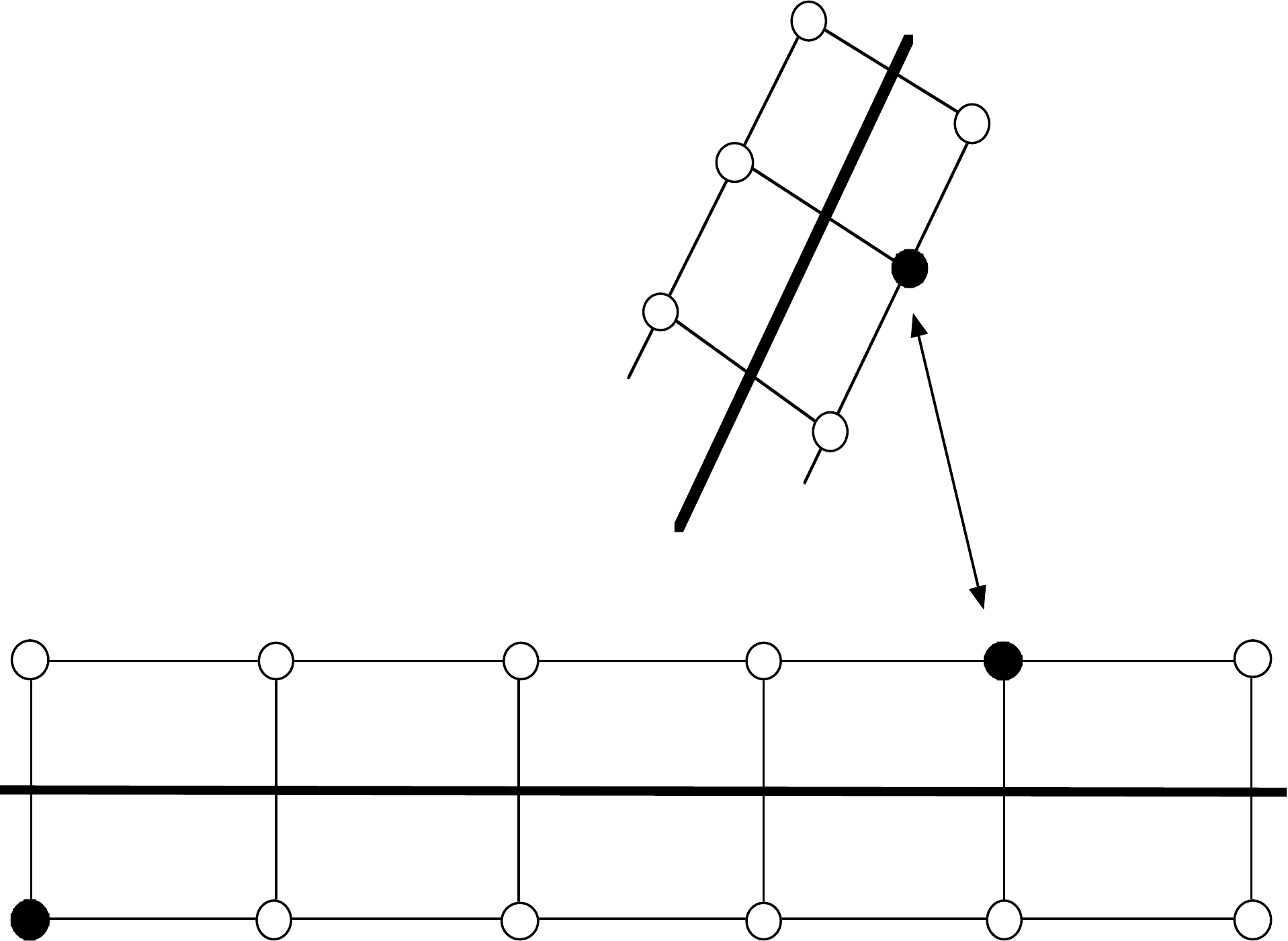}
\caption{An element of $B(G,M)$}
\end{figure}
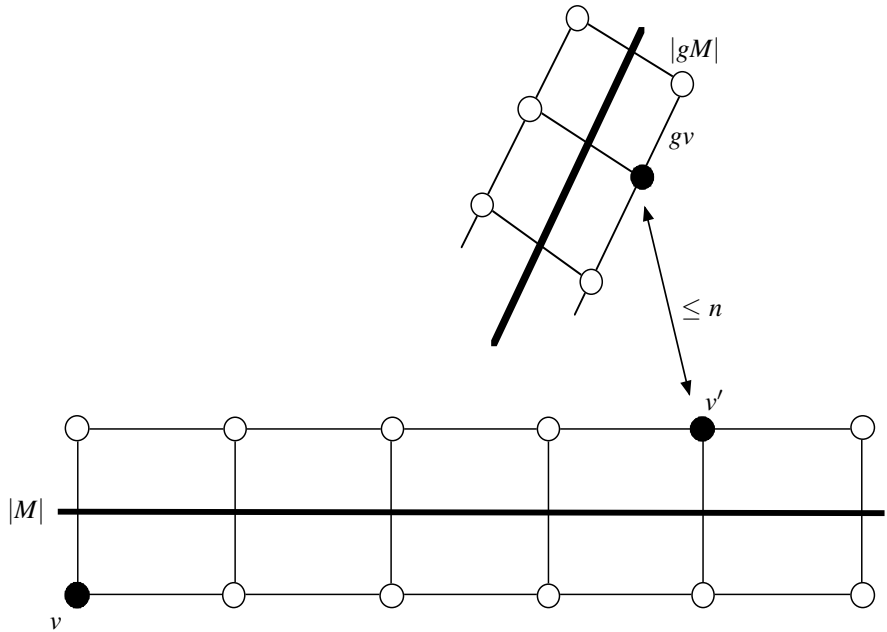

Let $G_M$ denotes the stabilizer
of $M$.
  Note that $G_M$ is the set  of $g\in G$ such
that there exists
$\ovae\in M$ satisfying 
$g(\!\ovae\!)\in M$. Observe that
$B(M,G)$ is right- and left-invariant under
$G_M$, and contains those $g\in G$ such that $gM\neq M$ but $g\size{M\!}\cap\size{M\!}\neq\emptyset$.

There are finitely many $\Gamma$--orbits of
oriented edges since
$\Gamma$ is cocompact. Thus 
$\Gamma_M\backslash M$ and hence
$\Gamma_M\backslash V(M)$ is finite.  Thus
there is a finite set
$\{b_1,\dots,b_k\}$ such that
$B(M,\Gamma)=\Gamma_M b_1 \Gamma_M\cup\cdots\cup \Gamma_M b_k \Gamma_M$,
since $X$ is locally compact and $\Gamma$ acts discretely.

By definition  $\{b_1,\dots,b_k\}\cap
\Gamma_M=\emptyset$. Since
$\Gamma_M$ is separable by assumption there  is
a finite index subgroup
$\Gamma'\subset
\Gamma$ such that $\Gamma_M\subset\Gamma'$  and
$\{b_1,\dots,b_k\}\cap
\Gamma'=\emptyset$. So for
$1\le i\le k$ we have
$\Gamma'\cap
\Gamma_Mb_i\Gamma_M=\emptyset$, hence
$B(M,\Gamma')=B(M,\Gamma)\cap 
\Gamma'=\emptyset$.

Let us consider any two translates 
$\gamma_1{M},\gamma_2{M}$ ($\gamma_i\in
\Gamma'$) whose geometric associated walls
have a nontrivial intersection. Then either
${\gamma_2}^{-1}\gamma_1$ stabilizes
$M$ or it belongs to $B(M,\Gamma')$. The 
latter possibility has been excluded thus
$\gamma_1{M}=\gamma_2{M}$.

So for the moment we have proved a weak 
version of the lemma: for each wall $M$ we
have found a finite index subgroup 
$\Gamma'\subset \Gamma$ such that any two
translates
$\gamma_1{M},\gamma_2{M}$  ($\gamma_i\in
\Gamma'$) either are equal,  or have disjoint
associated geometric walls.

If we consider a  finite index subgroup
$\Gamma''\subset \Gamma'$ such that $\Gamma''$
is normal in $\Gamma$, then we will have the
above property for each wall in the
$\Gamma$--orbit of $M$.

Since $\Gamma$ is cocompact there are  finitely
many $\Gamma$--orbits of walls. Thus taking a
finite intersection of finite index normal
subgroups $\Gamma''$ as above yields a finite
index subgroup with the desired property.
\end{proof}

When $X$ is   assumed to be a square complex,  
a result similar  to the previous lemma was
proven in \cite{HaglundWise05}.

\begin{rem}\label{rem:engulf}
Let $G$ denote a group and let $H\subset G$ denote a subgroup
all of whose finite index subgroups are separable in $G$. Then
each finite
index subgroup of $H$
is the intersection with $H$ of
a finite index subgroup in
$G$.

Indeed, let $K\subset H$ denote a finite
index subgroup. Write $H=\cap_i G_i$
where each $G_i$ is a finite index
subgroup of $G$. Then we also have $K=\cap_i
(G_i\cap H)$. All subgroups 
$G_i\cap H$ are intermediate
between $K$ and $H$. Since
$K$ is of finite index in $H$ there are finitely many distinct
intersections $G_i\cap H$. The
corresponding intersection $\bigcap G_i$
yields a finite index subgroup of $G$
whose intersection with $H$ is
$K$.

This applies in particular when $X$ is a simply connected polygonal
complex and $\Gamma$ denotes a
uniform lattice of ${\rm
Aut}(X)$ such that each
finite index subgroup of a wall stabilizer
in $\Gamma$ is separable: each
finite index subgroup of a wall-stabilizer in $\Gamma$
is the intersection with the wall-stabilizer of
a finite index subgroup in
$\Gamma$.
 \end{rem}

\subsection{Killing a 2--cocycle along a wall}

\begin{thm}\label{thm:killcocyclwall}  Let
$X$ denote some locally compact simply connected polygonal
complex satisfying $(\mathrm{C^2})$. Let $\Gamma^0$
denote a uniform lattice in ${\rm
Aut}(X)$. Assume that $\Gamma^0$ has no
self-intersection, and  that each finite index
subgroup of a wall-stabilizer  contains the
intersection with the wall-stabilizer of a
finite index subgroup in
$\Gamma^0$.

  Let
$\Gamma\subset\Gamma^0$ denote a finite index 
normal subgroup. Let
$c\co \Pi\to \field A$ denote a $\Gamma$--invariant
2--cocycle on $X$.

For any wall $M$  there
is a finite index normal subgroup
$\Gamma'\subset \Gamma^0$ and a 
$\Gamma'$--invariant 2-cocycle $c'$ on $X$  such
that
\begin{enumerate}
\item $c'-c=\delta u$, where $u$ is a 
$\Gamma'$--invariant 1--cochain
\item if a polygon $\pi$ is not separated by any  translate
$\gamma M$ ($\gamma\in\Gamma^0$)
then  $c'(\pi)=c(\pi)$ 
\item
   if a polygon $\pi$ is separated by some (necessarily unique) translate
$\gamma M$ ($\gamma\in\Gamma^0$)
then   $c'(\pi)=0$.
\end{enumerate}
\end{thm}

\begin{proof}
We first study the set $F=F(c,M)$ of 1--cochains
$u$ with the following properties:
\begin{enumerate}
\item ${\rm Supp}(u)$ is contained in the set
of edges which are dual to
$M$.
\item for any polygon $\pi$ separated by
$M$, if $\ovae$ and $\ovae\!'$ are the two oriented
edges of $\delta_v\pi$ such that $e$ and $e'$ are dual to $M$  then
$c(\pi)+u(e)\eta(\!\ovae\!)+u(e')\eta(\!\ovae\!')=0$.
(This condition
does not depend on the choice of the vertex $v\in\partial \pi$.)
\end{enumerate}

In view of the first condition, note that the
second condition reads $(c+\delta u)(\pi)=0$
for polygons $\pi$ that are separated by $\size{M\!}$. Note also that the pair
of oriented edges  $\ovae,\ovae\!'$ of the second
condition is well-defined because by \fullref{lem:walltree}, each
geometric wall intersects $\pi $ in a single diameter.

\begin{lem}
The set $F$ is finite, nonempty and invariant
under the stabilizer $\Gamma_M$ of $M$ in
$\Gamma$.
\end{lem}

\begin{proof}
To prove the two first assertions we
consider an edge $e$ of $X$ dual to $M$, with center $v$,
and we show that the map
$F\to \field A$ sending
$u\in F$ to $u(e)\in \field A$ is a bijection.

Consider a
combinatorial injective path
$c=(v_0=v,v_1,\dots,v_n=w)$ of the tree $T(M)$ whose first barycentric
subdivision is
$\size{M\!}$ (see \fullref{lem:walltree}). Any edge $(v_{i-1},v_{i})$ of
this path is contained in a single polygon
    $\pi _i$ of
$X$. We denote by $\overrightarrow{e_{i}}$ the oriented edge of $\delta_v\pi_i$ whose center is $v_{i-1}$, and we denote by $\overrightarrow{f_i}$ the oriented edge of $\delta_v\pi_i$ whose center is $v_{i}$ (this well-defined  independently of the choice of $v\in\partial \pi$). Observe that either $\overrightarrow{f_i}=\overrightarrow{e_{i+1}}$ or $\overrightarrow{f_i}={\overleftarrow{e_{i+1}}}$.

\begin{figure}[ht!]
\labellist
\small\hair 2pt
\pinlabel $v_0$ at 2 164
\pinlabel $\overrightarrow{e_1}$ at 0 107
\pinlabel $\overrightarrow{f_1}$ at 80 190
\pinlabel $\pi_1$ at 50 142
\pinlabel $v_1$ at 113 118
\pinlabel $\overrightarrow{e_2}$ at 99 80
\pinlabel $\pi_2$ at 165 99
\pinlabel $\overrightarrow{f_2}$ at 216 144
\pinlabel $v_2$ at 222 97
\pinlabel $\overrightarrow{e_3}$ at 272 141
\pinlabel $\pi_3$ at 314 99
\pinlabel $v_3$ at 367 104
\pinlabel $\overrightarrow{f_3}$ at 365 60
\pinlabel $\overrightarrow{e_4}$ at 417 139
\pinlabel $\pi_4$ at 460 96
\pinlabel $v_4$ at 512 161
\pinlabel $\overrightarrow{f_4}$ at 524 101
\endlabellist
\centering
\includegraphics[width=11cm]{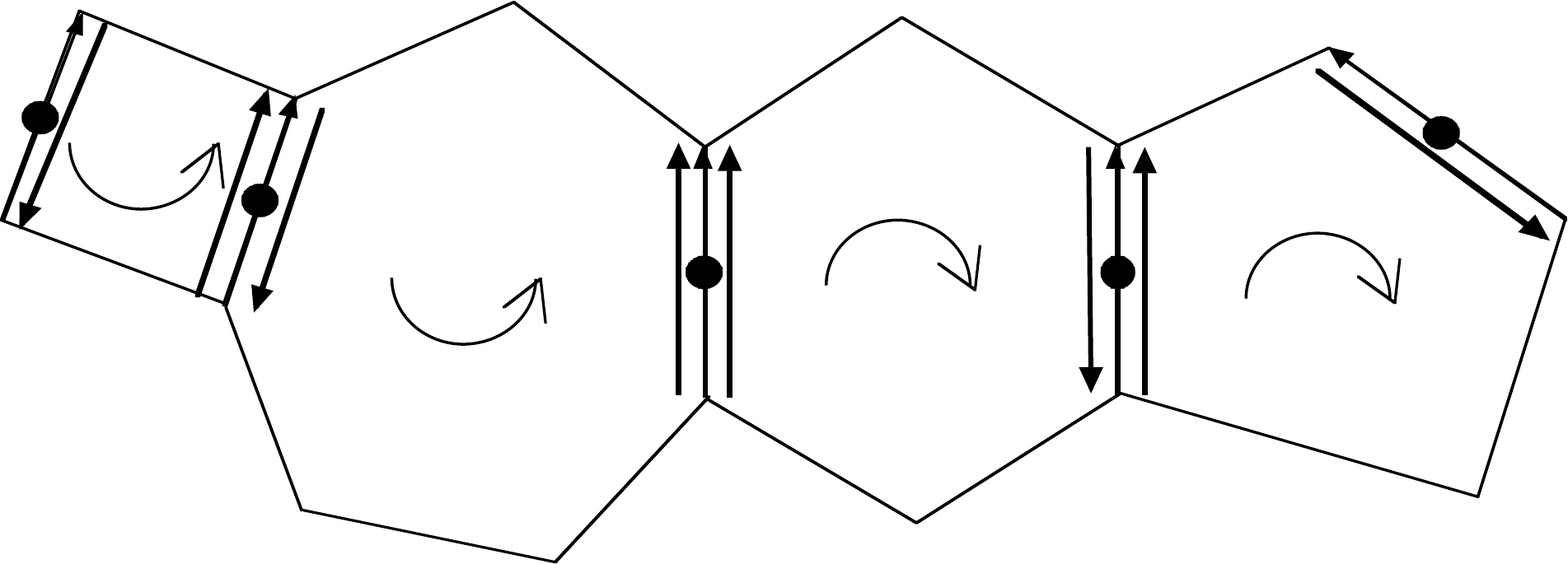}
\caption{A path in $T(M)$. Each vertex $v_i$ is indicated by $\bullet$ and each center 
of an oriented edge of the wall $M$ is indicated with an arrow. The oriented edges 
$\vec{e_{i}}$ and $\vec{f_{i}}$ 
are drawn inside the polygons $\pi_{i}$.}
\end{figure}
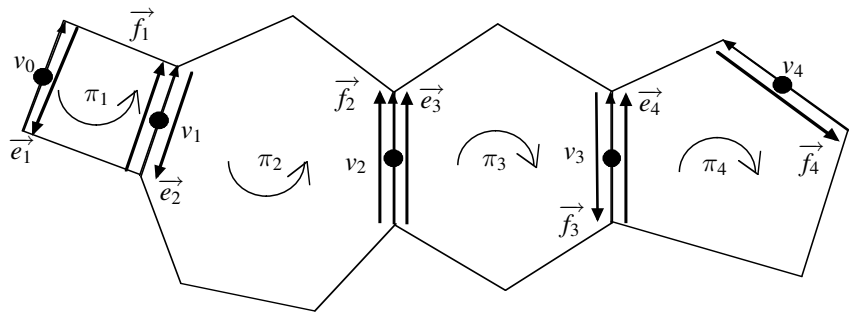

Then we claim that for any element $g$ of
$\field A$, there exists one and only one
sequence
$(g_0,\dots,
g_n)$ of elements of $\field A$ such that
\begin{enumerate}
\item $g_0=g$
\item for $1\le i\le n$ we have $c(\pi_i)+g_{i-1}\eta(\overrightarrow{e_i})+g_i\eta(\overrightarrow{f_i})=0$.
\end{enumerate}

Given  $g_{i-1}\in \field A$ the last equation is
always solvable in the unknown
$g_i\in \field A$ because $\eta$ takes its values in $\{-1,+1\}$.  So the
claim follows by induction on the length $n$
of the path $c$.

For any edge $e'$ of $X$ dual to ${M}$, with center $w$, let
$c=(v_0=v,v_1,\dots,v_n=w)$ denote the unique
injective combinatorial path in the tree
$T(M)$ from $v$ to $w$. Using the notation
above, we define
$u(e')$ by setting $u(e')={g_{n}}$.
For any edge $e''$ not dual to ${M}$
we set $u(e'')=0$. We thus obtain the unique
preimage of $g\in \field A$ in $F$ under the
restriction map
$F\to \field A$ (at the edge $e$).

It remains to check that $F$ is invariant
under $\Gamma_M$. For $u\in F$ and
$\gamma\in\Gamma_M$ clearly $\gamma u$ vanishes outside the set edges
dual to ${M}$.  Since    $c$
is $\Gamma$--invariant, the calculations made  in
\fullref{exmp:calculor} show  that the second relation defining
elements of $F$ is still satisfied by $\gamma u$.
\end{proof}

We construct $c',\Gamma'$ using
the action of $\Gamma_M$ on the sets $F(c,M)$.

\smallskip
Since $F(c,M)$ is finite and $\Gamma_M$ acts on
it, a finite index subgroup $\Lambda\subset
\Gamma_M$ fixes $F(c,M)$ pointwise.  Since $[\Gamma^0_M:\Gamma_M]<+\infty$, there exists by
assumption a finite index subgroup
$\Gamma(c,M)\subset 
\Gamma^0$ such that
$\Gamma(c,M)\cap\Gamma_M
\subset\Lambda$, that is the wall-stabilizer
$\Gamma(c,M)_M$ acts trivially  on
$F$. We may even assume that $\Gamma(c,M)$ is
normal in
$\Gamma^0$ and contained in $\Gamma$.

Consider all  $\Gamma$--invariant 2--cocycles 
$d$. Observe that there are finitely many such
$d$, because of
$\Gamma$--invariance ($\Gamma$ is cocompact). 
For each such $d$ we consider as above a finite
index subgroup
$\Gamma(d,M)\subset \Gamma$, which is normal  in $\Gamma^0$,  such that
the stabilizer of
$M$ in
$\Gamma(d,M)$ acts trivially on $F(d,M)$.
Then the intersection of all such 
$\Gamma(d,M)$ is a finite index normal
subgroup $\Gamma'$ of
$\Gamma^0$, $\Gamma'\subset
\Gamma$, such that ${\Gamma'}_M$ acts 
trivially on all
sets $F(d,M)$.

 \smallskip
   Choose elements
$\gamma_1=1,\gamma_2,\dots,\gamma_m\in
\Gamma^0$ such that any wall $\gamma M$ (with
$\gamma\in\Gamma^0$) is equivalent modulo
$\Gamma'$ to a unique $\gamma_i M=M_i$. We now
define a $\Gamma'$--invariant 1--cochain
$u$.
\smallskip

Pick  elements $u_i\in
F(c,M_i)$. We first check  that $u_i$ is invariant
under ${\Gamma'}_{M_i}$.  Using   the calculations made in \fullref{exmp:calculor}, we see that ${\gamma_i}$ sends
$F({\gamma_i}^{-1}c,M)$ to
$F(c,{\gamma_i}M)$. Observe  that
${\gamma_i}^{-1}c$ is invariant under
$\Gamma$ (because $\Gamma$ is normal in $\Gamma^0$). So by construction ${\Gamma'}_{M}$
fixes $F({\gamma_i}^{-1}c,M)$. Thus
${\gamma_i}{\Gamma'}_{M}{\gamma_i}^{-1}$ fixes
$F(c,{\gamma_i}M)$. But $\Gamma'$ is normal in
$\Gamma^0$, so ${\gamma_i}{\Gamma'}_{M}{\gamma_i}^{-1}=
{\Gamma'}_{M_i}$.

For every $\Gamma^0$--translate $\gamma M$, $\gamma\in\Gamma^0$, we
now choose some
$\gamma'\in\Gamma'$ such that $\gamma'
\gamma_i M$ equals $\gamma M$ with $\gamma'=1$ if
$\gamma M$ is $\gamma_i M$ for some $i$.

For an edge $e$ not dual to a wall of
the form  $\gamma{M}$ for
$\gamma\in\Gamma^0$, we set $u(e)=0$.

For an edge $e$  dual to $\gamma_i {M}$
we set $u(e)=u_i(e)$. Observe that by the
absence of self-intersection $e$ is not dual
to another
$\gamma_j {M}$.

More generally if the  edge $e$ is dual to a 
wall of the form  $\gamma  {M}$ for
$\gamma\in\Gamma^0$ where $\gamma M=\gamma'\gamma_iM$,  we set
$u(e)=(\gamma'u_i)(e)$. Again
$e$ is not dual to $\gamma''{M}$ not equal to $\gamma M$.

The 1--cochain $u$ is $\Gamma'$--invariant because each $
{\Gamma'}_{M_i}$ fixes $u_i$.

We set $c'=c+\delta u $: this new 2--cocycle is
$\Gamma'$--invariant ($\Gamma'\subset\Gamma$).

We now compute  $c'(\pi)$ for some polygon 
$\pi$.

If $\pi$ is not separated by a wall of  the
form $\gamma M$ ($\gamma
\in \Gamma^0$) then for each  edge $e\subset\partial \pi$ we have
$u(e)=0$. Thus $c'(\pi)=c(\pi)$.

Assume now that $\pi$ is separated by at least one $\Gamma$--translate
of $M$. Since distinct $\Gamma$--translates of
$M$ have disjoint associated geometric walls,
in  fact there exists a unique translate
$\gamma M$ which separates
$\pi$.   On the set of edges dual
to $\gamma M$ the 1--cochain $u$ coincides with an element of
$F(c,\gamma M)$. Thus
$c'(\pi)=0$.

The 2--cocycle $c'$ is trivial on a  
$\Gamma^0$--invariant set of polygons, and
$c'=c$ on the  $\Gamma^0$--invariant complement
of  this set. Thus $c'$ is $\Gamma$--invariant
(but it is cohomologous to $c$ only by the
coboundary of a $\Gamma'$--invariant 1--cochain).
\end{proof}

\subsection{Killing a 2--cocycle in a finite cover}

\begin{thm}\label{thm:killcocycl}
Let $X$ denote a compact polygonal complex
satisfying the nonpositive curvature 
condition $(\mathrm{C^2})$.  Assume that finite index
subgroups of wall stabilizers are separable in
$\Gamma=\pi _1(X)$.   Then for any 2--cocycle 
$c$ on
$X$, there is a (regular) finite cover 
$p\co X'\to X$ such that the lift $p^*(c)$ is
cohomologous to zero.
\end{thm}

\begin{proof}
First note that each basic oriented edge or polygon of $X$ 
lifts to the universal cover $\tilde
p\co \widetilde X\to X$. We thus endow  $\widetilde
X$ with a natural structure of polygonal
complex such that  $\tilde p$ is combinatorial
and in fact sends basic  oriented edges or
polygons  to basic  oriented edges or polygons.
Thus the deck transformation group $\Gamma$
also sends basic  oriented edges or polygons 
to basic  oriented edges or polygons. (In
particular for any $\gamma\in\Gamma$ and any
edge $e$ or any polygon $\pi$ we have
$\eta(\gamma,e)=\eta(\gamma,\pi)=1$.)

First by
\fullref{lem:virtclean}  there is a
finite index subgroup of $\Gamma$  acting
without self-inter\-section. So we may and will
assume that
$\Gamma$ acts
without self-intersection.

We denote by $\tilde c$  the 2--cocycle on
$\widetilde X$ preimage of $c$ under the
universal cover $\widetilde X\to X$.

By construction $\tilde c$ is 
$\Gamma$--invariant.

We choose walls $M_1,\dots,M_n$ such  that any
wall of $\widetilde X$ is a
$\Gamma$--translate of a unique wall  of the
family $\{M_1,\dots,M_n\}$.

We first set $\Gamma^0=\Gamma$. Using \fullref{rem:engulf}, we see
that we may apply \fullref{thm:killcocyclwall}  to the
wall $M_1$. Thus we find a finite index normal
subgroup
$\Gamma^1\subset\Gamma$, a 1--cochain $u^1$ and
a 2--cocycle $c^1$ such that $u^1,c^1$ are
$\Gamma^1$--invariant,  $c^1=\tilde c+\delta
u^1$, and for every polygon $\pi$, either $\pi$
is separated by no $\Gamma$--translate and
$c^1=\tilde c$ on
$\pi$, or  $\pi$ is separated by some 
$\Gamma$--translate
$\gamma M$ and then
$c^1(\pi)=0$.

We apply again 
\fullref{thm:killcocyclwall} to the
subgroup
$\Gamma^1$, the 2--cocycle $c^1$  and the wall $M_2$, thus  getting a finite
index normal subgroup $\Gamma^2\subset \Gamma$ with $\Gamma^2\subset
\Gamma^1$, together with a 1--cochain $u^2$ and a 2--cocycle
$c^2=c^1+\delta u^2$  which is trivial along
$\Gamma$--translates of $M_2$.  If we repeat
this process until we reach the
$\Gamma$--translates of $M_n$, we obtain a
decreasing sequence of normal subgroups
$\Gamma^1,\dots,\Gamma^n$ (each  of finite
index in $\Gamma$),  together with  sequences of
1--cochains,
$u^1, \dots,u^n$ and 2--cocycles $c^1,
\dots,c^n$  with $u^i,c^i$ invariant under
$\Gamma^i$, and $c^{i+1}=c^i+\delta u^{i+1}$.

We claim that $c^n=0$.

Indeed consider a polygon $\pi$ with ${k}$ 
edges. The walls separating $\pi$ are in the
$\Gamma$--orbit of
$M_{i_1},\dots,M_{i_{k}}$ with $i_1<
\cdots<i_{k}$. We denote by ${N}_j$ the
$\Gamma$--translate of $M_{i_j}$ which separates
$\pi$.

   For $i\not\in\{i_1,\dots,i_{k}\}$ we know that
$c^i=c^{i-1}$ on  $\pi$  (by convention
  $c^0=\tilde c$), and for $i\in\{i_1,\dots,i_{k}\}$ we know that
$c^i(\pi)=0$. Thus $c^n(\pi)=0$.

Let  $p\co X'\to X$ denote the  (regular) covering
corresponding to the normal subgroup
$\Gamma^n\subset \pi_1(X)$. The lift $p^*(c)$
lifts to $\tilde c$ in $\widetilde X$.   Note
that $\tilde c,u^1,\cdots,u^n$ are all
invariant under $\Gamma^n$. Thus each $c^i$ is 
$\Gamma^n$--invariant  and defines a 2--cocycle
$c_i$ on $X'$. Furthermore $c_i$ is
cohomologous to $c_{i-1}$ on $X'$. As a
conclusion $p^*(c)=c_0$ is cohomologous to
$c_n=0$.
\end{proof}

\begin{thm}\label{thm:virttrivextbis}
Let $X$ denote a compact polygonal complex 
satisfying  the nonpositive curvature
condition $(\mathrm{C^2})$.  Assume that finite index
subgroups of wall-stabilizers are separable in
$\Gamma=\pi _1(X)$.

Then
for any extension $1\to G\to 
\wbar{\Gamma}\to \Gamma\to 1$ where $G$  is
finite, the group 
$\wbar{\Gamma}$ is commensurable with 
${\Gamma}$.
\end{thm}

\begin{proof}
Let  $\pi\co \wbar{\Gamma}\to \Gamma$ denote a
morphism whose finite kernel is $G$.

Let us denote by $Z$ the centralizer of $G$ in
$\wbar{\Gamma}$. Since $G$ is normal and finite $Z$ is
a finite index subgroup of $\wbar{\Gamma}$.
So $\pi(Z)$ is a finite index subgroup of
$\Gamma$ and the kernel of $\pi\co  Z\to \pi(Z)$
is central in $Z$. This means that we may
restrict ourselves to central extension $1\to \field A\to
\wbar{\Gamma}\to \Gamma\to 1$ 
($G=\field A$ is now a finite abelian group and $\field A$ is
contained in the  center of
$\wbar{\Gamma}$).

 Central extensions of $ \Gamma$ by $\field A$ correspond, up to
isomorphisms of extensions, to cohomology
classes of $H^2( \Gamma,\field A)$ ($ \Gamma$ acts trivially on $\field A$) (see Brown \cite[Theorem 3.12, page 93]{Brown82}).
  Since $\widetilde X$ is $\CAT(0)$ it is contractible, $X$ is an Eilenberg--MacLane space for $ \Gamma$. Thus  $H^2( \Gamma,\field A)\simeq H^2( X,\field A)$, the group 
  of 2--cocycles on $X$ modulo the subgroup of coboundaries of 1--cocycles, and the product
extension $1\to \field A\to \field A\times {\Gamma}\to
\Gamma\to 1$ corresponds to the trivial
cohomology class. For convenience we recall here very briefly the identification between the set of central extensions of $ \Gamma$ by $\field A$ and the group $H^2(X,\field A)$.

Suppose first we are given some central extension $1\to \field A\to
\wbar{\Gamma}\to \Gamma\to 1$.   The inclusion $(X^1,x)\to (X,x)$
induces a surjective map $F=\pi_1(X^1,x)\to \Gamma=\pi_1(X,x)$.  The
kernel $R$ of this map is the subgroup generated by conjugates of elements
$\gamma_{\pi}$ in $\pi_1(X^1) = \sigma_{\pi}\cdot\delta_v\cdot\pi{\sigma_{\pi}}^{-1}$
where $\pi$ denote any polygon of $X$, and $\sigma_{\pi}$ is a path in
$X^1$ from the basepoint $x$ to some chosen vertex $v$ of $\partial
\pi$. Since $F$ is free and $\wbar{\Gamma}\to \Gamma$ is onto,  the
quotient morphism $F\to \Gamma$ has a lift to $\wbar{\Gamma}$. Choose
any lift $c\co F\to\wbar{\Gamma}$ of $F\to \Gamma$. Any element $r\in R$
maps to an element $c(r)\in \field A$ under this lift. In particular we may
associate to any polygon $\pi$ of $X$ an element $c(\pi)=c(\gamma_{\pi})\in
\field A$ (independent of the choice of $\sigma_{\pi}$ because $\field A$ is central). A
different choice of lift leads to a cohomologous 2--cocycle. Note that
$c=0$ means that the morphism $c\co F\to \wbar{\Gamma}$ is trivial on
$K$, thus defines a section of $\wbar{\Gamma}\to \Gamma$. Then
$\wbar{\Gamma}$ is a semidirect product, hence a direct product
since $\field A$ is central.

Conversely let $c$ denote some 2--cocycle on $X$ with values in $\field A$. In the free product $\field A*F$ 
consider the normal subgroup $N$ generated by the elements
$c(\gamma_{\pi})^{-1}\gamma_{\pi}$ and the  
commutators $[g,f], g\in \field A,f\in F$. Set $\wbar{\Gamma}=F/N$. Extend
trivially on $\field A$ the natural map $F\to \Gamma$, then $\field A*F\to \Gamma$
defines a surjective morphism $\wbar{\Gamma}\to \Gamma$. Furthermore
the natural inclusion $\field A\to \field A*F$ defines an inclusion $\field A\to
\wbar{\Gamma}$ whose image is central. We have an exact sequence
$1\to \field A\to \wbar{\Gamma}\to \Gamma\to 1$, whose 2--cocycle is $c$. A
cohomologous 2--cocycle $c'$ leads to an isomorphic central extension.

Given a central extension $1\to
\field A\to \wbar{\Gamma}\to \Gamma\to 1$, let $c$ denote the associated 2--cocycle  on $X$.
 By \fullref{thm:killcocycl}, there is
a regular finite cover $p\co X'\to X$ such that
$p^*(c)$ is cohomologous to 0. Let $\Gamma'$
denote the finite index normal subgroup of
$\Gamma$ corresponding to the covering $X'\to
X$. Let $\overline{\Gamma'}$ denote the
preimage of $\Gamma'$ under
$\wbar{\Gamma}\to \Gamma$. Then $p^*(c)$ is
the 2--cocycle on $X'$ corresponding to the central
extension $1\to \field A\cap \overline{\Gamma'} \to
\overline{\Gamma'}\to \Gamma'\to 1$. This
latter extension is trivial, hence
$\overline{\Gamma'}$ contains a finite index
subgroup mapping isomorphically onto
 $\Gamma'$.
\end{proof}

\begin{proof}[Proof of
\fullref{thm:virttrivext}]
It suffices to note that if $X$ satisfies
$(\mathrm{C^2})$ the geometric walls are convex subtrees
by \fullref{lem:walltree}. Thus the wall-stabilizers  and their finite
index subgroups are convex subgroups, hence separable by assumption. It
remains to apply \fullref{thm:virttrivextbis}.
\end{proof}

\part{III\qua Commensurability in some Coxeter--Davis\newline complexes}

In the first part of this paper we studied type-preserving
uniform lattices of the right-angled building associated with a graph product of finite groups.
In this part we consider a two-dimensional Coxeter group, and we are interested in the uniform lattices
of the polygonal complex associated with the Coxeter group. The results of this part are quite similar to the results of part 1, the method is rather analogous to the method used in part 2.

In all of this part we will denote by $L$ some
simplicial graph. We denote by $I$ the set of
vertices of
$L$ and by $E$ the set of its edges. We let
$\mu$ denote the girth of $L$, that is the least
length of a simple closed combinatorial path in
$L$. We consider a map $m$ assigning to each
edge $e$ of $L$ an integer $m(e)\ge 2$.

We consider the Coxeter presentation
$$\<s_i,i\in I\,\vert\,  {s_i}^2=1,\ \forall i\in I;(s_is_j)^{m(e)}=1,\
\forall e\in E\ {\rm with}\ \partial e=\{i,j\!\}\>$$ and denote by
$(W,S)$ the associated Coxeter system where $W$ is the group and 
$S$ is the generating set $\{s_i,\,i\in I\}$.  For $i,j\in J,\ i\neq j$ we set $m_{ij}=m(e)$ if there exists
an edge $e\in E$ with  $\partial e=\{i,j\!\}$, and otherwise we set $m_{ij}=\infty$.
For any subset $J\subset
I$  we denote by $W_J$ the subgroup of $W$ generated by
$\{s_j\}_{j\in J}$. We recall that the natural map 
$$\<s_i,i\in J\,\vert\, 
{s_i}^2=1,\ \forall i\in J;(s_is_j)^{m(e)}=1,\
\forall e\in E\ {\rm with}\ \partial e=\{i,j\!\},i\in J,j\in J\>\to W_J$$ is
an isomorphism, so $(W_J,J)$ is also a Coxeter system. For this and other
general results on Coxeter systems we refer to Bourbaki \cite{Bourbaki68}.

\section{Coxeter--Davis complexes}\label{sec:coxdavcomplex}

In this section we associate to the Coxeter system a geometric realization. As a
building block we first introduce a simplicial complex of dimension
two: $B_*=x_**L'$. So $B_*$ is obtained as the join of a single vertex $x_*$ with the first
barycentric subdivision of $L$. We set $\partial B_*=L'\subset B_*$.
\textit{The facets of $B_*$\/} are the subcomplexes $\phi_i$, stars of $i$
in $\partial B_*$. Note that the facets cover $\partial B_*$.
For each point $p\in B_*$ let $\tau(p)$ denote the
smallest simplex of $B_*$ containing $p$. We then define $t(p)\subset
I$ as the set of $i\in I$ such that $\tau(v)\subset \phi_i$.
    Now we define an equivalence relation $\sim$ on $W\times B_*$ by
declaring $(w,p)\sim (w',p')\iff p=p'\in \partial B_*$ and
$w^{-1}w'\in W_{t(p)}$

We will denote by $A=A(W,B_*)$ the quotient of
$W\times B_*$ by $\sim$ and we will denote by $[w,p]$ the image in $A$ of
$(w,p)$. Note the analogies with the first part: ${\mathcal G}\leftrightarrow \nobreak L,{\Gamma}\leftrightarrow\nobreak W, C_*\leftrightarrow B_*,\Delta\leftrightarrow A$.

In the  following proposition we list some well-known properties
of $A$ \cite{HaglundReseaux} and we introduce some
notation.

\begin{prop}\label{prop:propapart}
\begin{enumerate}
\item The left action of $W$ on $W\times B_*$ induces an action of $W$ on
$A$.
\item The second projection $W\times B_* \to B_*$ induces a
$W$--invariant map $\rho\co A\to B^*$, and $\rho(p)=\rho(q)$ if and only if
$p,q$ belong to the same $W$--orbit. 

For $p\in A$ we set $t(p)=t(\rho(p))$ (the type of $p$)
and ${\rm rk}(p)=\size{t(p)}$ (the rank of $p$).

\item The map $j$ sending $p\in B^*$ to $[1,p]$ satisfies $\rho j={\rm
id}_{B^*}$. Thus we will identify $B^*$ with its image $j(B^*)$, and
$B^*$ is then a strict fundamental domain for the action of $W$ on $A$
(meaning that each $W$--orbit intersects $B^*$ in a single point).

\item The quotient $A$ admits a unique structure of  simplicial complex
such that the natural map  $W\times B_*\to A$ is combinatorial. The
$W$--translates $wB^*$ are {\rm the blocks
of $A$}, and for any facet $\phi_i$ of $B^*$ with center $v$ and any $w\in
W$ we call $w\phi_i$ a {\rm facet of $A$} with center $wv$. The vertex
$wx_*\in wB_*$ is called the {\rm center} of the block $wB_*$.

For $J\subset \{0,1,2\}$ we will denote by $A^J$ the full subcomplex of
$A$ whose vertices $v$ satisfy $t(v)\in J$.

\item For each $p\in B_*$ the stabilizer of $p$ in $W$ is $W_{ t(p)}$. In
particular $W$ acts simply transitively on the blocks of $A$. If two
distinct blocks share a rank--1 vertex $v$, then their intersection is the
facet with center $v$.

\item Let $v$ denote a vertex of $A$ of type $J$ and rank $r$. Then $r\le
2$. If $r=0$ then $v$ is the center of a     unique block of $A$, and ${\rm
link}(v,A)$ is isomorphic to $\partial B^*$. If $r=1$ then $v$ is the
center of a unique facet of
$A$, and ${\rm link}(v,A)$ is the suspension of the boundary of the
facet.  If
$r=2$ then
$J=\{i,j\!\}$ with
$m_{ij}<\infty$ and
${\rm link}(v,A)$ is a cyclic graph with $m_{ij}$ vertices of type
$\{i\!\}$,
$m_{ij}$ vertices of type $\{j\}$  and $2m_{ij}$ vertices of type
$\emptyset$.

\item There exists a unique regular polygonal complex $X=X(W,B_*)$ whose
bary\-centric subdivision is $A$ in such a way that the rank of the center
of a face of $X$ is the dimension of the face, that is, the rank in $A$
coincides with the rank of barycentric subdivisions. For any vertex $x$
of $X$ we will denote by $B(x)$ the block of $A$ with center $x$.

\item $A$ is simply connected.
\end{enumerate}
\end{prop}

\begin{rem}
By \fullref{rem:autreg}, the group  ${\rm Aut}_{\rm rk}(A)$ of
combinatorial automorphisms of $A$ preserving the rank in fact coincides
with ${\rm Aut}(X)$.
\end{rem}

\begin{defn}
The Coxeter system $(W,S)$ is  \textit{two-dimensional\/} whenever $W_J$ is
infinite for each subset $J\subset I$ with ${\size J}\ge 3$.
\end{defn}

\begin{lem}\label{lem:condnpc}
\begin{enumerate}
\item If the   Coxeter system $(W,S)$ is {two-dimensional} then the
polygonal complex $X$ satisfies the nonpositive curvature condition
$(\mathrm{C^2})$.

\item If the   Coxeter system $(W,S)$ is { two-dimensional} and $I$ does
not contain any subset $\{i,j,\ell\}$ with $m_{ij}\le 3,m_{j\ell}=3,m_{\ell
i}<+\infty$, then $X$ satisfies the  condition ($\mathrm{C^4}$).

\item If the  girth $\mu$ of $L$ is at least $4$ then $(W,S)$ is
two-dimensional and $X$ satisfies  condition (Q).
\end{enumerate}
\end{lem}

\begin{proof}
(1)\qua In the case of $X$,  the nonpositive curvature condition $(\mathrm{C^2})$
reads as follows:
$$\text{If }
c=(\overrightarrow{e_1},\dots,
\overrightarrow{e_n}) \text{ is a cycle of } L \text{ then } 
\sum_i \Big(\frac{1}{2}-\frac{1}{2m_{e_i}}\Big)\ge 1.
$$ Since $m\ge 2$ this condition is obviously  fulfilled for any cycle
of length $n\ge 4$. So assume that $c$ has length 3. Let $J$ denote the
set of vertices of $L$ belonging to one of the edge of $c$. The subgroup
$W_J$ must be infinite by the condition on dimension. But $W_J$ is a
triangle group, so being infinite implies
$\unfrac{1}{m(e_1)}+\unfrac{1}{m(e_2)}+\unfrac{1}{m(e_3)}\le 1$, which is
equivalent to
$(\mathrm{C^2})$ in this particular case.

\medskip (2)\qua The  condition ($\mathrm{C^4}$) reads as follows:
$$\text{If }
c=(\overrightarrow{e_1},\dots,
\overrightarrow{e_n}) \text{ is a cycle of } L \text{ then }
\sum_i \Big(\frac{1}{2}-\frac{1}{4f_{e_i}}\Big)\ge 1
$$ where $4f(e_i)= 2{m(e_i)}$ if $m(e_i)$ is even, and 
$4f(e_i)={2m(e_i)-2}$ otherwise.

Again this relation is obviously fulfilled   for any cycle of length
$n\ge 4$. So assume that $c=(\overrightarrow{e_1},\overrightarrow{e_2},\overrightarrow{e_3})$ has length 3. Let $J$ denote the set of
vertices of $L$ belonging to one of the edges of $c$. 
We may assume that $m(e_1)\le m(e_2)\le m(e_3)$. Necessarily $m(e_3)>3$, otherwise the group $W_J$ is either finite commutative or has one $m(e_i)$ equal to $3$, another $m(e_i)$ at least $3$ and the last one is finite. Similarly $m(e_2)>2$, otherwise the group $W_J$ is finite, and in fact  $m(e_2)>3$ by the hypotheses. Then 
$\sum_i (\unfrac{1}{2}-\unpfrac{1}{4f_{e_i}})\ge (\unfrac{1}{2}-\unfrac{1}{4})+(\unfrac{1}{2}-\unfrac{1}{8})+(\unfrac{1}{2}-\unfrac{1}{8})=1$. 

\medskip \noindent (3)\qua If the girth $\mu$ of $L$ is $\ge 4$ then in any subset $\{i,j,\ell\}\subset I$ one of the quantities  $m_{ij},m_{j\ell},m_{\ell i}$ is infinite, hence $(W,S)$ is two-dimensional.  Condition (Q) is trivially fulfilled.
\end{proof}

\begin{exmp}\label{exmp:mconstant}
Assume that the function $m\co I\to \{2,3,\dots\}$ is constant. Then the
polygonal complex $X$ in \fullref{prop:propapart} is the complex
$X(m,L)$ of \fullref{intro}. The condition  ``$m\ge 3$ or $m=2$ and the
girth of $L$ is at least $4$'' is equivalent to $(W=W(m,L),S)$ being two-dimensional.

When this condition is fulfilled by \fullref{lem:condnpc}, we see that $X$ satisfies the condition
$(\mathrm{C^2})$, hence also the (more natural) nonpositive curvature condition
(C): equipped with a piecewise Euclidean length metric turning each 
polygon to a regular Euclidean polygon, $X$ is $\CAT(0)$.

Observe that the condition ``$X(m,L)$ is negatively curved'' of \fullref{intro}, that is,
$m\ge 4$, or $m\ge 3$ and the girth of $L$ is at
least $4$, or $m\ge 2$ and the girth of $L$ is
at least $5$, is stronger than the two-dimensionality condition.  Furthermore it is clear that under these hypotheses, $I$ does
not contain any subset $\{i,j,\ell\}$ with $m_{ij}\le 3 ,m_{j\ell}=3,m_{\ell
i}<+\infty$. So condition ($\mathrm{C^4}$) is satisfied.
 In fact the condition is  equivalent to condition ($\mathrm{C^4}$') for complexes $X(m,L)$). Thus by  \fullref{lem:condcat0} the terminology 
 ``negatively curved'' is justified.
\end{exmp}

\section{Reflections, walls and $e$--walls}\label{sec:ewalls}

   \begin{defn}[Reflections] The reflections  of
$W$ are the conjugate of the generating
involutions $s_i,i\in I$. The fixed point set
of a reflection $r=ws_iw^{-1}$ on $A$ is \textit{the geometric wall of $r$\/}, denoted by $M(r)$.
\end{defn}

\begin{lem}
   The geometric wall of a reflection $r$ is
the geometric wall associated to some wall of the
polygonal complex $X$ (see \fullref{defn:wall}).
The geometric wall of a reflection $r$ is a separating
totally geodesic subtree of $A$.
    \end{lem}

\begin{proof}
Since $W$ acts simply transitively on vertices of $X$, the fixed point
set $M(r)$ is contained in the graph $A^{\{1,2\}}$. Let us check that for
any polygon $\pi $ of $X$ either $\pi \cap M(r)=\emptyset$ or $\pi \cap
M(r)$ is a diameter of $\pi $ (see \fullref{defn:wall}). Note that
$M(r)$ is the fixed point set of an isometry of $C$, hence it is
a convex subcomplex of the $\CAT(0)$ complex
$A$ \cite[Proposition 6.2]{BridsonHaefliger}. Thus $M(r)$ is a subtree of
$A^{\{1,2\}}$. Each generating reflection $s_i\in S$ fixes the unique
vertex of type $\{i\!\}$ in $B_*$, and $s_i$ preserves each polygon of $X$
containing this vertex. In particular
$M(r)$ contains a vertex of rank~1.

Assume that $\pi \cap M(r)\neq\emptyset$. Then since $M(r)$ is connected 
and contains a rank--1 vertex, the intersection $\pi \cap
M(r)$ also contains a rank--1 vertex.  Observe that there is a reflection
$r'$ which fixes the same rank--1 vertex and furthermore preserves $\pi $.
Necessarily $r=r'$, thus $r$ preserves 
$\pi$. The restriction of
$r$ to
$\pi $ is an affine automorphism preserving an edge and fixing no vertex:
thus $r$ acts on $\pi $ like a usual reflection, and  $\pi \cap M(r)$ is
a diameter of $\pi $.

It follows easily that $M(r)$ is union of geometric walls, hence a single
geometric wall by connectedness. The rest of the lemma follows by \fullref{lem:walltree}.
\end{proof}

   For later purposes we have to consider other
types of walls. We want walls to be (separating)
$\CAT(0)$--subtrees of $A^{\{1,2\}}$, but we also
want a parity condition that the wall of a
reflection might fail to satisfy.

\begin{defn}\label{def:eparallel}
Let $X$ denote some  polygonal complex,   and
let $\pi$ denote some polygon in $X$,  not a
triangle. We say that two oriented edges
$\overrightarrow{a},\overrightarrow{a}'$
contained in $\pi$ are \textit{even-parallel in
$\pi$} whenever there are  disjoint connected
subgraphs
$\sigma,\sigma'$  of
$\partial\pi$ satisfying the following
holds:
\begin{enumerate}
\item The boundary $\partial\pi$ is  the
union $a\cup \sigma\cup a'\cup \sigma'$.
\item If $\ell,\ell'$ denote the numbers of edges of
$\sigma,\sigma'$, then $\ell$ is odd and
$\ell\le\ell'< \ell+4$.
\item The terminal vertices of  the oriented
edges
$\overrightarrow{a},\overrightarrow{a}'$
belong to $\sigma$ (so their initial vertices
belong to $\sigma'$).
\end{enumerate}

\begin{figure}[ht!]
\labellist
\small\hair 2pt
\pinlabel $\sigma$ at 136 539
\pinlabel $\sigma'$ at 135 395
\pinlabel $\sigma$ at 235 351
\pinlabel $\sigma'$ at 248 152
\pinlabel $\sigma$ at 349 543
\pinlabel $\sigma'$ at 350 377
\pinlabel $\sigma$ at 518 355
\pinlabel $\sigma'$ at 518 152
\pinlabel $\sigma$ at 618 545
\pinlabel $\sigma'$ at 620 363
\endlabellist
\centering
\includegraphics[width=11cm]{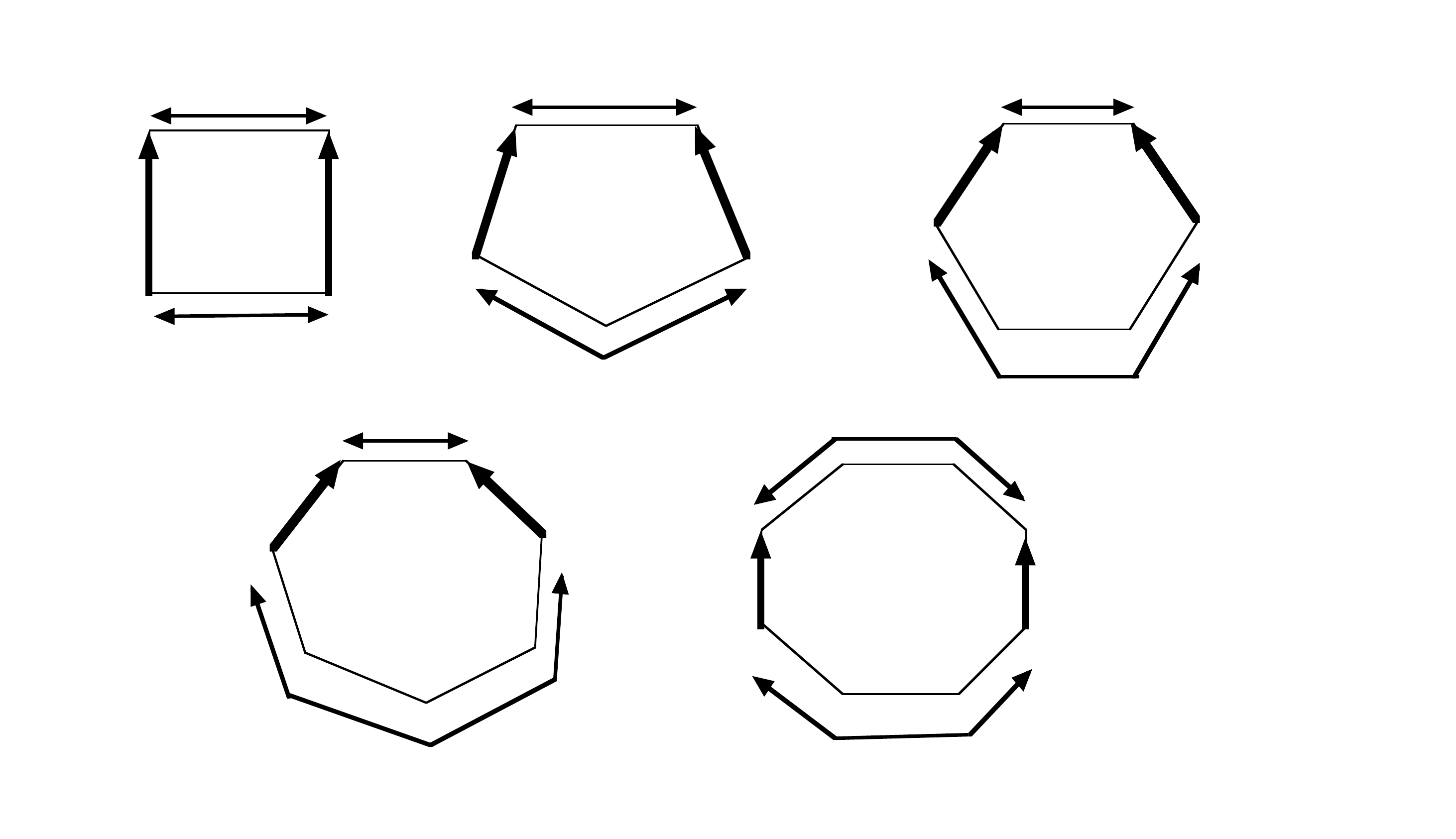}
\caption{Even-parallel oriented edges}
\end{figure}
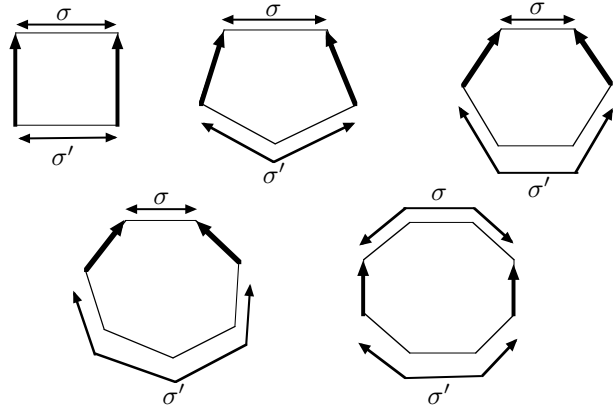

Observe that if such subgraphs exist then they
are unique. Hence any oriented  edge
$\overrightarrow{a}$ of a polygon $\pi$ is
even-parallel to exactly one other oriented
edge $\overrightarrow{a}'$ of $\pi$. We denote
by $\parallel_e$ the equivalence relation on
oriented edges generated by even-parallelism
inside polygons: that is we write
$\overrightarrow{a}\parallel_e
\overrightarrow{a}'$ whenever there is a
sequence of polygons $\pi_1,\dots,\pi_n$ and a
sequence of oriented edges
$\overrightarrow{a}_0=\overrightarrow{a},
\dots,\overrightarrow{a}_n=\overrightarrow{a}'$
such that $\overrightarrow{a}_i$ and
$\overrightarrow{a}_{i+1}$ are even-parallel
in $\pi_{i+1}$. Note that the group ${\rm
Aut}(X)$ preserves the relation
$\parallel_e$.
\end{defn}

\begin{exmp}
   Let $k\ge 4$ denote the number of edges of
the polygon $\pi$.

   If $4\vert k$ then two oriented edges
$\overrightarrow{a},\overrightarrow{a}'$
contained in $\pi$ are { even-parallel in
$\pi$} whenever they induce opposite
orientations on $\partial\pi$ and the
associated geometric edges $a,a'$ are opposite
on $\partial\pi$.

More generally if $k=4f+r$ with $0\le r<4$
then $\ell=2f-1$ and $\ell'=2f-1+r$.

Note that parallelism and even-parallelism
inside a polygon coincide  if the number of
sides of this polygon is of the form $4f$ or
$4f+1$.  For example in a square complex
$\parallel\ =\ \parallel _e$.
\end{exmp}

\begin{defn}[$e$--walls]\label{defn:ewall}
Let $X$ denote a polygonal complex. \textit{An
$e$--wall of $X$} is an equivalence class of
$\parallel _e $.

We say that \textit{an $e$--wall $M$
passes through an oriented edge
$\overrightarrow a$}  whenever
$\overrightarrow a\in M$. We say that \textit{a
wall
$M$ separates a polygon $\pi $ of $X$}
whenever $M$ passes through oriented edges
$\overrightarrow a,\overrightarrow{a}'$ with
$a\subset \partial\pi,a'\subset \partial\pi$ and $\overrightarrow
a,\overrightarrow{a}'$ are even-parallel inside $\pi $.

\textit{An e-diameter of a polygon $\pi$\/} of $X$ is the union of two edges
$[v,p],[v',p]$ of $X'$ such that $p$ is the center of $\pi $, and there
are distinct oriented edges $\overrightarrow a,\overrightarrow{a}'\subset
\partial\pi$ which are even-parallel in $\pi $, and whose centers are $v,v'$.
The e-diameter is then said to be \textit{dual to $\overrightarrow
a,\overrightarrow{a}'$}.

\textit{The geometric $e$--wall associated to an $e$--wall
$M$} is the subcomplex $\size{M\!}$ of the
first barycentric subdivision $X'$, union of those e-diameters of polygons
which are dual to $\overrightarrow
a,\overrightarrow{a}'$, with $\overrightarrow
a,\overrightarrow{a}'\in M$.

  If some $e$--wall $M$ passes through an oriented
edge
$\overrightarrow a$, we will also  say that
\textit{$M$ is dual to
$\overrightarrow a$}, and that \textit{$M$ is dual to
$a$}. For any $e$--wall $M$ we will
denote by $V(M)$ the set of vertices $v$
belonging to an edge $a$  to which $M$ is dual.

  Note that a combinatorial map $f\co X\to Y$
maps an $e$--wall of $X$ inside an $e$--wall of $Y$.
The same is true for  the
associated geometric $e$--walls.
\end{defn}

Changing walls to $e$--walls, the proof of the following lemma is exactly
the same as that of \fullref{lem:walltree}.

\begin{lem}\label{lem:ewalltree}
 Assume that $X$ is locally compact, satisfies condition  ($\mathrm{C^4}$)  and is
simply connected. Equip $X$ with the
($\mathrm{C^4}$)--metric. Then:
\begin{enumerate}
\item Each geometric $e$--wall is a totally geodesic subtree.

\item Each geometric $e$--wall separates $X$ into two connected components.

\item Two opposite oriented edges are not even-parallel.

\item The intersection of a geometric $e$--wall and a polygon is either empty
or a single diameter. In particular a geometric $e$--wall $\size{M\!}$ is the
first barycentric subdivision of a tree whose vertices correspond to rank
1 vertices of $\size{M\!}$, and whose center of edges  correspond to rank
2 vertices of $\size{M\!}$.
\end{enumerate}
\end{lem}

\section{Systems of local reflections}\label{sec:localrefl}
\begin{defn}[Edge neighbourhoods]
Let $v$ denote some rank--1 vertex of $A$.  Then
$v$ is adjacent to exactly two vertices $x,y$
of rank 0.  \textit{The edge neighbourhood of $A$
with center $v$} is $U(v)=B(x)\cup B(y)$. Note
that
$U(v)$ is the first regular neighbourhood in $X$ of the edge
with center $v$.
\end{defn}

Let us define a \textit{weight\/} $m\co A^{\{2\}}\to
\{2,3,\dots\}$ by the rule $m(p)=m(e)$, where $e$ is the edge of
$L=\partial B_*$ whose center is $\rho(p)$. In other words $m(p)$
is half the number of sides of the polygon of $X$ with center $p$. For
any subcomplex $K\subset A$ we will denote by ${\rm
Aut}_{\rm rk}(K,m)$ the group of automorphisms of $K$ preserving the rank
and the weight $m$ on $K$. Clearly ${\rm
Aut}_{\rm rk}(A,m)={\rm
Aut}_{\rm rk}(A)$.

The following notions were introduced in \cite{HaglundExUniHom} for
arbitrary polygonal complexes all of whose polygons have the same number
of sides.

\begin{defn}\label{defn:hol}
Let $v$ denote some rank--1 vertex of $A$; let $x,y$ denote its
two adjacent rank--0 vertices.
\textit{A local reflection at $v$\/} is an automorphism $\sigma$ of the
edge neighbourhood $U(v)$ that exchanges $x$ and $y$, preserves the rank
and the weight
$m$, and fixes pointwise the facet of center $v$.

\textit{A system of local reflections on $A$\/} is a choice
$(\sigma_v)_v$ for every rank--1 vertex $v$ of a local reflection
$\sigma_v$ at $v$.

Let $\sigma$ denote some system of local reflections. Let $\varphi\in
{\rm Aut}_{\rm rk}(A)$. For any vertex $v $ of rank 1 consider
the local reflection at $v$ given by ${\varphi}\circ \sigma_{
{\varphi}^{-1}(v)}\circ {\varphi}^{-1}$. This defines a new system of
local reflections, denoted by $\varphi_*(\sigma)$. We thus obtain an
action of ${\rm Aut}_{\rm rk}(A)$ on the set of systems of
local reflections. Let ${\rm Aut}_{\rm rk}(A,\sigma)$
denote the stabilizer of $\sigma$.

For any triangle $\tau$ of $A$, let $\pi$ denote the polygon
of $X$ whose first barycentric subdivision contains $\tau$.
Let $(\overrightarrow{a_1},\dots,\overrightarrow{a_{2m}})$ denote the
combinatorial edge-path of $X$ winding once around $\pi$, such
that the first barycentric subdivision of $a_1$ contains an edge of
$\tau$. Then \textit{the holonomy of a system of local reflections
$\sigma$ at $\tau$} is the composition
$h_{\sigma}(\tau)=\sigma_{v_{2m}}\circ\dots\circ \sigma_{v_{1}}$,
where $v_j$ denotes the center of $a_j$. Note that $h_{\sigma}(\tau)$
is an automorphism of the block $B(\tau)$ containing $\tau$. It is
straightforward that $h_{\sigma}(\tau)$ fixes the rank--2 vertex $p$ of
$\tau$ ($p$ is the center of the polygon $\pi$). Let $\tau'$ be the
other triangle of $B(\tau)$ containing $p$: then
$h_{\sigma}(\tau')={h_{\sigma}(\tau)}^{-1}$. In fact $h_{\sigma}(\tau)$
fixes $\tau$ and $\tau'$.

We say that \textit{a system $\sigma$ of local reflections has no
holonomy} if for every triangle $\tau$ we have $h_{\sigma}(\tau)=1$.
\end{defn}

\begin{exmp}\label{exmp:systemcox}
For each rank--1 vertex $v=wv_i$ the stabilizer of $v$ in $W$ is the
conjugate $w\{1,s_i\}w^{-1}$. Thus we may restrict the global
reflection $ws_i w^{-1}$ to the edge neighbourhood $U(v)$. This
yields a system of local reflections on $A$, which we will
denote by $\sigma^W$.

We observe that  $\sigma^W$ has no holonomy because $W$ acts 
freely on blocks.
\end{exmp}

\begin{prop}\label{prop:conjugsystem}
Let $\sigma,\sigma'$ denote two systems of local reflections on
$A$ with no holonomy. Let $B,B'$ denote any two blocks of
$A$. Let $f\co B\to B'$ denote any {\rm germ}, that is, any isomorphism of
blocks preserving the rank and preserving the weight $m$ on rank--2
vertices. Then $f$ extends to a unique automorphism $\fbar\in {\rm
Aut}_{\rm rk}(A)$ such that $\fbar_*(\sigma)=\sigma'$.
\end{prop}

\begin{proof}
The proof is exactly parallel to the proof of
\fullref{prop:conjugatlas}. We just give a skeleton, details are
left to the reader.

Define a \textit{ gallery\/} to be a sequence of blocks such that two
consecutive  blocks are \textit{adjacent\/} in the sense that they are
equal or they share a facet.

Then to each pair $(f,G)$ where $f\co B\to B'$ is a germ and $G=(B_0=B,B_1,\dots,B_n)$ 
a gallery, we associate a gallery
$f(G)=(B'_0=B',B'_1,\dots,B'_n)$ and a sequence
$G(f)=(f_0,f_1,\dots,f_n)$ of germs $f_i\co B_i\to B'_i$ such that
\begin{enumerate}
\item if
$B_{i+1}=B_i$ then $B'_{i+1}=B'_i,f_{i+1}=f_i$
\item if $B_{i+1}\cap B_i$
is a facet with center $v_i$ then $B'_{i+1}$ is the block of $A$ whose
intersection with $B'_i$ is the facet of center $v'_i=f_i(v_i)$, and
$f'_{i+1}=\sigma'_{v'_i}\circ f_i\circ \sigma_{v_i}$.
\end{enumerate}
If $f$  extends to an automorphism  $\fbar\in{\rm
Aut}_{\rm rk}(A) $ conjugating
$\sigma$ with $\sigma'$ then for each gallery $G$ we have $\fbar(G)=f(G)$ and $\omega(G(f))={\fbar}\vert_{\omega(G)}$ with
$\omega(G)=B_n$ and $\omega(G(f))=f_n$. This proves uniqueness.

If $G$ is closed, then since $\sigma,\sigma'$ have no holonomy the last
germ of $G(f)$ is $f$ (it suffices to check this when $G$ is a gallery
winding once around a polygon of $X$).

Thus we have $\omega(G(f))=\omega(G'(f))$, for any two galleries
$G$ and $G'$ from
$B$ to some fixed  block. This way we
associate to each block a germ defined on this block:
these germs fit together and define a combinatorial map $\fbar\co A\to
A$.

We then see that $\fbar$ is a rank-preserving
automorphism conjugating $\sigma$ with $\sigma'$.
\end{proof}

\begin{cor}\label{cor:conjugsystem}
Let $\sigma$ denote some system  of local reflections on $A$
with no holonomy. Then ${\rm Aut}_{\rm rk}(A,\sigma)$ is
conjugate with ${\rm Aut}_{\rm rk}(A,\sigma^W)$ inside ${\rm Aut}_{\rm rk}(A)$.

The stabilizer of a
block $B$  in ${\rm Aut}_{\rm rk}(A,\sigma)$ is isomorphic (by
restriction) to the group of all germs $f\co B\to B$. In particular ${\rm
Aut}_{\rm rk}(A,\sigma)$ is a uniform lattice of $A$.

The group ${\rm Aut}_{\rm rk}(A,\sigma^W)$ is the semidirect product of
$W$ and the finite group of those automorphisms of the graph $L$ which
preserve the weight $m$ (this group acts naturally on $W\!$).
\end{cor}

\begin{proof}
 The proof is similar to the proof of 
\fullref{cor:conjugatlas}. The first assertion follows by
\fullref{prop:conjugsystem} with
$\sigma'=\sigma^W$.
For the second assertion use \fullref{prop:conjugsystem}
with $\sigma'=\sigma$.

The last affirmation is a consequence of the previous one, since $W$ is
already simply-transitive on blocks and for any block $B$ we have ${\rm
Aut}_{\rm rk}(B,m)\simeq {\rm Aut}(L,m)$.
\end{proof}

\begin{proof}[Proof of \fullref{thm:commholcox}]
Assume first the uniform lattice
$\Gamma$ of  $X(m,L)$ is commensurable to
$W(m,L)$ in ${\rm Aut} X(m,L))$. Let $f$ denote an automorphism of
$X(m,L)$ conjugating a finite index subgroup of $W$ onto a finite index
subgroup $\Gamma'\subset \Gamma$. Then $\Gamma'\subset {\rm Aut}_{\rm
rk}(A,\sigma)$ where $\sigma$ is the system of local reflections without
holonomy obtained by conjugating $\sigma^W$ by $f$.

Conversely assume
$\Gamma$ admits a finite index subgroup $\Gamma'$
preserving a system of local reflections $\sigma$ whose
holonomy is trivial. Then by \fullref{cor:conjugsystem}, 
there exists an automorphism $f\in {\rm Aut}_{\rm rk}(A)$ which
conjugates  ${\rm Aut}_{\rm rk}(A,\sigma)$ 
 with ${\rm Aut}_{\rm rk}(A,\sigma^W)$, and furthermore $W$ is of finite
index in ${\rm Aut}_{\rm rk}(A,\sigma^W)$. It follows that $f$ conjugates
a finite index subgroup of $\Gamma'$ onto a finite index subgroup of $W$.
Hence $f$ commensurates $\Gamma$ with $W$.
\end{proof}

\section{Killing the holonomy}\label{sec:killholcox}

\begin{lem}\label{lem:preservsystem}
Let $\Gamma'\subset {\rm Aut}_{\rm rk}(A)$ 
denote any subgroup acting freely on rank--1
vertices.
Then $\Gamma'$ preserves some system of
local reflections.
\end{lem}

\begin{proof}
Let $(v_t)_{t\in T}$  denote a set of
representatives of rank--1 vertices under the
action of $\Gamma'$. Choose arbitrarily a
local reflection $\sigma_{v_t}$ at each $v_t$, for example use
$\sigma^W$.
Then extend $\sigma$ using conjugation by
elements of $\Gamma'$.
\end{proof}

\subsection{Modification of a system of local reflections}

\begin{defn}[\rm Fields]\label{defn:fields}
    For each triangle  $\tau$ of $A$ we denote by
$B(\tau)$ the block containing $\tau$, by
$\phi^+(\tau)$ the facet of $B$ meeting $\tau$
along an edge, and by $F^+(\tau)$ the subgroup
of  ${\rm Aut}_{\rm rk}(B(\tau),m)$ fixing
$\phi^+(\tau)$ pointwise. We also denote by
$\phi^-(\tau)$ the other facet of $B(\tau)$
meeting $\tau$ (at the vertex of rank $2$).
We introduce the subgroup $F^-(\tau)$ of
${\rm Aut}_{\rm rk}(B(\tau),m)$ fixing
$\phi^-(\tau)$ pointwise.

    A \textit{semi-edge of $X$\/} is an edge $e$ of $A$
whose vertices have rank 0 and 1. Given a
semi-edge $e$, we will always denote by $\bar
e$ the unique other semi-edge of $X$ meeting
$e$. Note that the union of $e$ and $\bar e$
is the barycentric subdivision of an edge of
$X$.

    For each semi-edge $e$ of $X$ we denote by
$B(e)$ the block containing $e$, by
$\phi(e)$ the facet of $B$ meeting $e$, and
by $F(e)$ the subgroup of  ${\rm Aut}_{\rm
rk}(B(e),m)$ fixing $\phi(e)$ pointwise.

    \textit{A field of rank 2 on $A$\/} is a map $f$
assigning to each triangle $\tau$ of $A$ an
element $f(\tau)$ of  $F^+(\tau)$. Similarly
a map $f$ assigning to each semi-edge $e$ of
$X$ an element $f(e)$ of  $F(e)$ is \textit{a
field of rank 1 on $A$}.

    The set of fields of given rank has a natural
structure of group: for example in rank 1 we
set $f_1f_2(e)=f_1(e)f_2(e)$.

The \textit{support\/} of a rank--1 field $f$ is the set ${\rm Supp}(f)$
of semi-edges $e$ of $X$ such that $f(e)\neq
1$. Similarly the \textit{\/} of a rank--2 field $\varphi$ is the set
${\rm Supp}(\varphi)$
of triangles $\tau$ of $A$ such that $\varphi(\tau)\neq
1$.

The group ${\rm Aut}_{\rm rk}(A)$ acts on fields by conjugation:
for $\gamma\in {\rm Aut}_{\rm rk}(A)$ and $f$ a field we define $\gamma f$
by $\gamma f(x)=\gamma \circ f(\gamma^{-1}x)\circ \gamma^{-1}$.
\end{defn}

\begin{defn}
    Let $\sigma$ denote a system of local
reflections. Let $f$ denote some field of rank
1. We say that \textit{$f$ is symmetric with respect to
$\sigma$} (or \textit{$\sigma$--symmetric\/} for short) whenever for every rank--1 vertex $v$
and for every semi-edge $e$ of $X$ containing
$v$, we have ${\sigma_v}\vert_{B(e)} \circ
f(e) = f(\bar e)^{-1}\circ {\sigma_v}\vert_{B(e)} $.

    Assume $f$ is $\sigma$--symmetric.
Consider some rank--1 vertex $v$; let
$e_1,e_2$ be the two semi-edges of $X$
containing $v$. We define an automorphism
$\sigma'_v$ of $U(v)$ as follows:
\begin{enumerate}
\item for $p\in B(e_1)$ we set
$\sigma'_v(p)=\sigma_v(f(e_1)p)$
\item for $p\in B(e_2)$ we set
$\sigma'_v(p)=\sigma_v(f(e_2)p)$.
\end{enumerate}

    This is well-defined since $B(e_1)\cap
B(e_2)=\phi(v)$ and $f(e_i),\sigma_v$ fix
$\phi(v)$ pointwise. Clearly $\sigma'_v$
exchanges $e_1$ and $e_2$. Furthermore using
invariance, we have for $p\in B(e_1)$
$${\sigma'_v}^2(p)=
\sigma_v(f(e_2)\sigma_v(f(e_1)p)) =
\sigma_v(f(e_2)f(e_2)^{-1} {\sigma_v}(p)
)=p.$$ 
Thus $\sigma'_v$ is an involution.  We
have thus obtained a new system of local
reflections, which we will denote by $\sigma
f$. 
\end{defn}

\begin{rem}\label{rem:fieldtrans}
    For any two systems of local reflections
$\sigma,\sigma'$ there exists one and only one
 $\sigma$--symmetric rank--1 field $f$ such that $\sigma'=\sigma f$.
This field is defined by $f(e)p={\sigma_v}^{-1}(\sigma'_v(p))$ where $v$ is the rank
0 vertex of the semi-edge $e$.
\end{rem}

We are going to modify systems of local reflections by applying
convenient symmetric rank--1 field, in such a way that the holonomy of the
resulting system of local reflections is smaller.

\begin{lem}[Holonomy of $\sigma
f$]\label{lem:holcompos} Let $\sigma$ denote some system
of local reflections. Let $f$ denote some 
$\sigma$--symmetric rank--1 field. Let $\tau$ denote some
triangle of $A$. We introduce a sequence
${\underline g}={\underline
g}(\tau,\sigma,f)=(g_0,\dots,g_{2m})$ of
automorphisms of $B(\tau)$ defined as follows.

Let $\pi$
denote the polygon of $X$ containing $\tau$,
let
$(\overrightarrow{a_1},
\dots,\overrightarrow{a_{2m}})$ denote the
combinatorial edge-path of $X$ winding once
around $\pi$ such that the first  barycentric
subdivision of $a_1$ contains an edge of
$\tau$ (see \fullref{defn:hol}), and let $v_i$ denote the rank  1
vertex of $A$ at the center of $a_i$. Let
$e_i$ denote the initial semi-edge of
$\overrightarrow{a_i}$.  We set
$$ g_i=g_i(\tau,\sigma,f)=
({\sigma_{v_{1}}}\vert_{B(e_2)}\circ\cdots\circ{\sigma_{v_{i-1}}}\vert_{B(e_i)})\circ f(e_i) \circ
({\sigma_{v_{i-1}}}\vert_{B(e_{i-1})}\circ\cdots\circ
{\sigma_{v_{1}}}\vert_{B(e_1)} ).$$
Then
$g_i=f(e_1)\in F^+(\tau)$ for $i$ odd, $g_i\in
F^-(\tau)$ for $i$ even, and
furthermore
$$h_{\sigma f}(\tau)= h_{\sigma }(\tau)\circ
g_{2{m}}\circ \cdots\circ g_1,$$ 
where $2{m}$ denotes the number of vertices of the polygon $\pi$.
\end{lem}

\begin{proof}
By definition we have 
$$h_{\sigma f}(\tau)=(\sigma f)_{v_{2m}}\circ\dots\circ (\sigma
f)_{v_{1}} = \sigma
_{v_{2m}}\circ f(e_{2m})\circ \sigma
_{v_{2m-1}}\circ f(e_{2m-1})\circ \dots\circ \sigma
_{v_{1}}\circ f(e_1).$$ 
We also have $$ \sigma
_{v_{2m}}\circ f(e_{2m})\circ \sigma
_{v_{2m-1}}\circ f(e_{2m-1})=\sigma
_{v_{2m}}\circ\sigma
_{v_{2m-1}}\circ {f(e_{2m})}^{\sigma
_{v_{2m-1}}}\circ f(e_{2m-1}).$$ 
If we make successively each
$\sigma_{v_i}$ go from the right to the left we get the desired formula
$h_{\sigma f}(\tau)= h_{\sigma }(\tau)\circ g_{2{m}}\circ \cdots\circ g_1
$.

Now local reflections preserve rank and weight. So conjugation by the
inverse of the composition
$t_i={\sigma_{v_{1}}}\vert_{B(e_2)}\circ\cdots\circ{\sigma_{v_{i-1}}}\vert_{B(e_i)}$ sends $F(e_{i})$ onto
$F(t_i(e_{i}))\subset F(\tau)$. When $i$ is odd, $F(t_i(e_{i}))=F^+(\tau)$, and $F(t_i(e_{i}))=F^-(\tau)$ otherwise.
\end{proof}

 The following result already appears in
\cite{HaglundExUniHom}.

\begin{cor}\label{cor:decomposhol}
For every fixed triangle $\tau$,  the set of
all possible holonomies $h_{\sigma}(\tau)$ (as
$\sigma$ varies in the set of systems of local
reflections) is the whole product subgroup
$F^+(\tau)F^-(\tau)=F^-(\tau)F^+(\tau)$.
\end{cor}

\begin{proof}
By \fullref{rem:fieldtrans}, any system of local reflection $\sigma$ is
of the form $\sigma^W f$, for $f$ some $\sigma^W$--symmetric rank--1 field.
Thus by \fullref{lem:holcompos}, the set of
all possible holonomies $h_{\sigma}(\tau)$ is the set of all possible
products $g_{2m}\circ \dots\circ g_1$.

Note that the subgroups $F^+(\tau)$ and $F^-(\tau)$ normalize each other. Thus
the group they generate consists in elements of the form $g^+g^-$ with
$g^+\in F^+(\tau)$ and $g^-\in F^-(\tau)$. So we always have
$h_{\sigma}(\tau)\in F^+(\tau)F^-(\tau)$.

Conversely choose a $\sigma^W$--symmetric rank--1 field which takes an arbitrary value on
$e_1$ and $e_2$,  and is trivial outside the edges
$e_1,e_2,\overline e_1$ and $\overline e_2$.  Then we see that
$h_{\sigma}(\tau)$ is any element of
$F^-(\tau)F^+(\tau)=F^+(\tau)F^-(\tau)$.
\end{proof}

\subsection{Groups acting cleanly}

\begin{defn}[\rm Self-intersecting $e$--wall] Let
$X$ be some simply connected polygonal
complex and let $\Gamma$ denote some subgroup
of ${\rm Aut}(X)$. We say that an $e$--wall $M$
is \textit{self-intersecting under $\Gamma$\/} whenever
there exists $\gamma\in \Gamma$ such that
$M$ and $\gamma{M}$ are distinct,  but the  geometric
$e$--walls associated to $M$ and $\gamma{M}$ have nonempty
intersection. We say that \textit{$\Gamma$ acts without e-self-intersection\/} whenever  no $e$--wall is
self-intersecting under
$\Gamma$.

We say that $\Gamma\subset {\rm Aut}(X)$ is \textit{clean\/} whenever it acts without e-self-intersection
and moreover it acts freely on $X$.
\end{defn}

\begin{lem}\label{lem:virteclean}
   Let $\Gamma$ denote a discrete cocompact
group of automorphism of some
simply connected polygonal  complex
$X$. Assume that the stabilizers of $e$--walls are
separable in $\Gamma$.

Then there is a  finite index subgroup 
$\Gamma'\subset\Gamma$ acting without
e-self-intersection.
\end{lem}

The proof is the same as that of \fullref{lem:virtclean}; details are
left to the reader.

\begin{cor}\label{cor:virteclean}
Assume $(W,S)$ is two-dimensional.
 Let $\Gamma$ denote a residually finite
uniform lattice of ${\rm Aut}_{\rm rk}(A)={\rm
Aut}(X)$ all of whose $e$--wall stabilizers are
separable. Then there is a finite index
subgroup $\Gamma'\subset\Gamma$ such that
$\Gamma'$ is clean.
\end{cor}

\begin{proof}
By \fullref{lem:virteclean}, it suffices  to
prove that $\Gamma$  virtually acts freely on
$A$.

Since $A$ admits a $\CAT(0)$ length metric by \fullref{lem:condnpc},  any
finite order element of
$\Gamma$ has a fixed point in $A$. Thus, up to conjugation,  there are
finitely many finite order elements in $\Gamma$. Since $\Gamma$ is
residually finite there is a finite index torsion free subgroup
$\Gamma'\subset\Gamma$. Since $\Gamma$ acts discretely we see that
$\Gamma'$ has to act freely.
\end{proof}
In fact the separability of $e$--wall stabilizers imply the
residual finiteness of $\Gamma$.

\begin{rem}\label{rem:coxclean}
Assume $A$ is two-dimensional and all $m_{ij}$ are even, for example, $W$
is right-angled. Then $e$--walls are walls, and geometric walls are the
fixed point sets of reflections of
$W$. Thus an $e$--wall stabilizer is the centralizer of a reflection. Since
Coxeter groups are residually finite the centralizers of their
involutions are separable. So in this case \fullref{cor:virteclean}
applies: $W$ has a clean finite index subgroup.

In \cite{HaglundWise06} we prove that in every Gromov-hyperbolic
Coxeter group, each quasiconvex subgroup is separable. It follows that
in this case also $W$ has a clean finite index subgroup.

The general two-dimensional case remains open.
\end{rem}

\subsection{Killing half of the holonomy along an $e$--wall}

 Until the end of this subsection we assume that $W$ is two-dimensional,
and we consider a fixed clean uniform lattice $\Gamma^0$ of
$X$.

\begin{defn}[\rm $(i,j)$--adjacency]
Let $i$ and $j$ denote two distinct integers of $\{0,1,2\}$.
We say that
two triangles $\tau_1$ and $\tau_2$ are
\textit{$(i,j)$--adjacent\/} whenever they share an edge
whose endpoints are vertices $x$ and $y$ of ranks $i$ and $j$, respectively. We
say that two triangles are \textit{adjacent\/}  if they are $(i,j)$--adjacent for
some pair $(i,j)$.

The $(i,j)$--adjacency condition is an equivalence relation on
the set of triangles. Every $(0,2)$--adjacency class
contains two triangles. For each triangle $\tau$ we will denote by
$\tau'$ the second triangle
of the $(0,2)$--adjacency class of $\tau$. Note that $\tau$ and $\tau'$ are
contained in the same polygon.
\end{defn}

\begin{defn}
Since $\Gamma^0$ preserves rank and acts
freely on rank--2 vertices, it acts freely on the set of
polygons of $X$, and we may find \textit{a
transversal set
${\mathcal T}$ of triangles} with the
following properties:

\begin{enumerate}
\item ${\mathcal T}$ is $\Gamma^0$--invariant
\item for each polygon $\pi$, the set of triangles of
${\mathcal T}$ contained
in $\pi$ consists of a $(0,2)$--adjacency class denoted by ${\mathcal T}_{\pi}$.
\end{enumerate}

From now on we choose such a set ${\mathcal T\/}$. We will use it 
to compute the holonomy of various systems of
local reflections.

Let $\pi$ denote a polygon of $X$. Since $\partial \pi$ has an even
number of edges
there is a labeling of these edges in $\{+,-\}$ such that two
adjacent edges receive distinct labels. This labeling defines  a
labeling on the set of rank--1 vertices of $\pi$.
For any $e$--wall $M$ separating the polygon $\pi$  the two edges of
$\partial\pi$ dual to $M$ receive the same label, say $\varepsilon$.
Then the rank--1 vertex of exactly one of the triangles of  ${\mathcal
T}_{\pi}$ has the same label $\varepsilon$: we will denote this
triangle by $\tau(\pi,M)$. Note that  ${\mathcal
T}_{\pi}=\{\tau(\pi,M),\tau(\pi,M)'\}$.

Observe that parallel oriented edges inside $\pi$  receive the same label
if and only if they are also even-parallel. This is why we need the 
notion of $e$--walls,  justifying
our terminology of even-parallelism.
\end{defn}

\begin{defn}
Let $\sigma$ denote any system of local
reflections. \textit{A decomposition of the
holonomy
$h_{\sigma}$} is a rank--2 field $\varphi$ with support contained in
${\mathcal T}$ such
that for any triangle $\tau\in{\mathcal T}$
we have
$h_{\sigma}(\tau)=\varphi(\tau')^{-1}
\varphi(\tau)$.
\end{defn}

\begin{rem}\label{rem:existdecomp}
Let $\sigma$ denote any system of local
reflections. Then $\sigma$ always admits a decomposition $\varphi$ by
\fullref{cor:decomposhol}. Furthermore if $\sigma$ is invariant
under a subgroup $\Gamma\subset{\rm Aut}_{\rm rk}(A)$ acting freely
on vertices then the decomposition $\varphi$ of $\sigma$ may be
chosen to be $\Gamma$--invariant.
\end{rem}

\begin{thm}\label{thm:killholwall} Assume $(W,S)$ is two-dimensional
and $I$ does
not contain any subset $\{i,j,\ell\}$ with $m_{ij}\le 3,m_{j\ell}=3,m_{\ell
i}<+\infty$ (so that $X$ satisfies condition ($\mathrm{C^4}$) by \fullref{lem:condnpc}).

Let
$\Gamma^0$ denote a clean uniform lattice in ${\rm Aut}_{\rm rk}(A)$.
where each finite index subgroup of a wall-stabilizer in
$\Gamma^0$ contains the intersection with the wall-stabilizer of a
finite index subgroup in $\Gamma^0$.  Consider such a finite index normal subgroup
$\Gamma\subset\Gamma^0$, and let 
$\sigma$ denote a $\Gamma$--invariant
system of local reflections and 
$\varphi$ denote a rank--2 field invariant
under $\Gamma$ which is a
decomposition of $h_{\sigma}$.

For any $e$--wall $M$  there
is a system of local reflections $\sigma'$,  a
rank--2 field
$\varphi'$ and a finite index normal subgroup
$\Gamma'\subset \Gamma^0$ such that:
\begin{enumerate}
\item $\sigma'$ and $\varphi'$ are invariant
under $\Gamma'$
\item
$\varphi'$ is a decomposition of
$h_{\sigma'}$
\item $\varphi'(\tau)=\varphi(\tau)$ on
triangles $\tau\in {\mathcal T}_{\pi}$ if a polygon $\pi$ is not separated by any  translate
$\gamma M$ for $\gamma\in\Gamma^0$
\item $\varphi'(\tau(\pi,\gamma M))=1$ and $\varphi'(\tau(\pi,\gamma
M)')=\varphi(\tau(\pi,\gamma M)')$
   if a polygon $\pi$ is separated by some (necessarily unique) translate
$\gamma M$ for $\gamma\in\Gamma^0$.
\end{enumerate}
\end{thm}
The theorem says that up to passing to a finite index subgroup
$\Gamma'$ we may simplify the decomposition of the holonomy of an
invariant system of local reflections along $\Gamma^0$--translates of
the wall $M$.
\begin{proof}
We first study the set $F=F(\sigma,\varphi,M)$ of rank--1 fields
$f$ with the following properties:
\begin{enumerate}
\item $f$ is $\sigma$--symmetric

\item any semi-edge of ${\rm Supp}(f)$ is contained in an edge  dual to
$M$
\item for any polygon $\pi$ separated by
$M$, using the notation of
\fullref{lem:holcompos} we have
$$g_{i_2(\pi,M)}(\tau(\pi,M),\sigma,f)
g_{i_1(\pi,M)}(\tau(\pi,M),\sigma,f)={\varphi(\tau(\pi,M))}^{-1}.$$
\end{enumerate}
We must explain the definition of the integers
$i_1(\pi,M)$ and $i_2(\pi,M)$. The triangle
$ \tau(\pi,M)$ defines a sequence of semi-edges
$(e_1,\dots,e_{2m_{\pi}})$ winding once around $\pi$ with $e_1\subset
\tau(\pi,M)$, as in \fullref{lem:holcompos}. Then $i_1(\pi,M)$ is the first integer $1\le i\le
2m_{\pi}$ such that $e_i$ meets $\size{M\!}$ and $i_2(\pi,M)$ is the second
such integer. Note that $\size{M\!}$ meets only two of the semi-edges $e_i$.

If $m_{\pi}$ is even then $i_2(\pi,M)=i_1(\pi,M)+m_{\pi}$, and in contrast if $m_{\pi}$ is
odd then $i_2(\pi,M)=i_1(\pi,M)+m_{\pi}-1$.

\begin{lem}
The set $F$ is finite, nonempty and invariant
under the stabilizer $\Gamma_M$ of $M$ in
$\Gamma$.
\end{lem}

\begin{proof}
To prove the two first assertions we
consider a rank--1 vertex $v$ in $\size{M\!}$ together
with a semi-edge $e$ of $X$ containing $v$,
and we show that the map
$F\to F(e)$ sending
$f\in F$ to $f(e)\in F(e)$ is a bijection.

Consider a
combinatorial injective path
$c=(v_0=v,v_1,\dots,v_n=w)$ of the tree $T(M)$ whose first barycentric
subdivision is
$\size{M\!}$ (see \fullref{lem:condnpc} and \fullref{lem:ewalltree}).
Any edge
$(v_{i-1},v_{i})$ of this path is contained in a single polygon
    $\pi _i$ of
$X$. With the above definition (using notation of \fullref{lem:holcompos})
the semi-edges of the sequence $(e_1,\dots,e_{2m_{\pi}})$ winding once
around $\pi_i$ with $e_1\subset \tau(\pi_i,M)$ meet  $\size{M\!}$ exactly
at  the indices $i_1:=i_1(\pi_i,M)$ and $i_2:=i_2(\pi_i,M)$. So if $w_j$ is the
rank--1 vertex of $e_j$ we have
$\{w_{i_1(\pi_i,M)},w_{i_2(\pi_i,M)}\}=\{v_{i-1},v_{i}\}$. To
simplify notation we set $\tau(\pi_i,M)
=\tau_i$. We also set
$a_i=e_{i_1}$ and $\wbar a_i=e_{i_2}$ if $v_{i-1}\in e_{i_1}$, and
otherwise we set $a_i=e_{i_2}$ and $\wbar a_i=e_{i_1}$.

As in
\fullref{lem:holcompos}, we define
two isomorphisms from the block
$B(\tau_i)$ onto $B(e_{i_1})$ and  from the block
$B(\tau_i)$ onto $B(e_{i_2})$, by
composing local reflections. More precisely
we define
${s_{i,1}}={\sigma_{w_{i_1-1}}}\vert_{B(e_{i_1-1})}
\circ\cdots\circ {\sigma_{w_{1}}}\vert_{B(e_{1})} $ which is the identity if $i_1=1$, and
 ${s_{i,2}}={\sigma_{w_{i_2-1}}}\vert_{B(e_{i_2-1})}
\circ\cdots\circ {\sigma_{w_{1}}}\vert_{B(e_{1})} $.

We claim that for any element $f$ of
$F(e)$ and any semi-edge $e'$ of $X$
containing the endpoint $w$ of $c$, if we set $\wbar a_0=e$ and
$a_{n+1}=e'$, then there exists one and only one sequence
$\bigl(\fbar_0,f_1,\fbar_1,f_2,\dots,
f_n,\fbar_n,f_{n+1}\bigr)$ such that
\begin{enumerate}
\item $\fbar_0=f$; for $1\le i\le
n$, $f_i\in
F(a_i)$, $\fbar_i\in
F(\wbar a_i)$;  $f_{n+1}\in F(e')$
\item for $0\le i\le
n$, if
$a_{i+1}=\wbar a_i$ then $f_{i+1}=\fbar_i$ and otherwise
$f_{i+1}=\sigma_{v_{i}}\circ
\fbar_i\circ
\sigma_{v_{i}}$

\item for $1\le i\le
n$, if $a_i=e_{i_1}$ then
$({s_{i,2}}^{-1}\circ \fbar_i\circ s_{i,2})\circ
({s_{i,1}}^{-1}\circ {f_i}\circ s_{i,1})=\varphi(\tau_i)^{-1}$ and otherwise
$({s_{i,2}}^{-1}\circ f_i\circ s_{i,2})\circ
({s_{i,1}}^{-1}\circ  \fbar_i\circ s_{i,1})=\varphi(\tau_i)^{-1}$.
\end{enumerate}

Given  $f_i\in F(a_i)$ the last equation is
always solvable in the unknown
$\fbar_i\in F(\wbar a_i)$ because
the rank--1 vertices of $e_1,e_{i_1},e_{i_2}$ are mutually at even
distance on $\partial \pi _i$, so we see that the isomorphism $s_{i,1}$
conjugates $F(e_1)=  F^+(\tau_i)$ with
$F(a_i)$ or  $F(\wbar a_i)$ (and similarly $s_{i,2}$
conjugates $F^+(\tau_i)$ with $F(\wbar a_i)$
or $F(a_i)$). The other
conditions are trivial to satisfy. So the
claim follows by induction on the length $n$
of the path $c$.

\begin{figure}[ht!]
\labellist
\small\hair 2pt
\pinlabel $a_4=e'$ [l] at 435 348
\pinlabel $v_3$ at 443 318
\pinlabel {$\bar a_3$} at 401 313
\pinlabel $\tau'_3$ at 482 278
\pinlabel $\tau_3$ at 497 249
\pinlabel {$\bar a_2$} [r] at 349 227
\pinlabel $\pi_3$ at 452 218
\pinlabel $a_3$ at 405 200
\pinlabel $v_2$ at 377 191
\pinlabel {$e=\bar a_0=a_1$} [r] at 57 166
\pinlabel $v=v_0$ [tr] at 59 137
\pinlabel $\tau'_1$ at 90 167
\pinlabel $\tau_1$ at 73 152
\pinlabel $v_1$ at 152 152
\pinlabel $a_2$ at 153 117
\pinlabel {$\bar a_1$} at 126 121
\pinlabel $\pi_1$ at 90 121
\pinlabel $\pi_2$ at 346 108
\pinlabel $\tau'_2$ at 260 35
\pinlabel $\tau_2$ at 298 35
\endlabellist
\centering
\includegraphics[width=13cm]{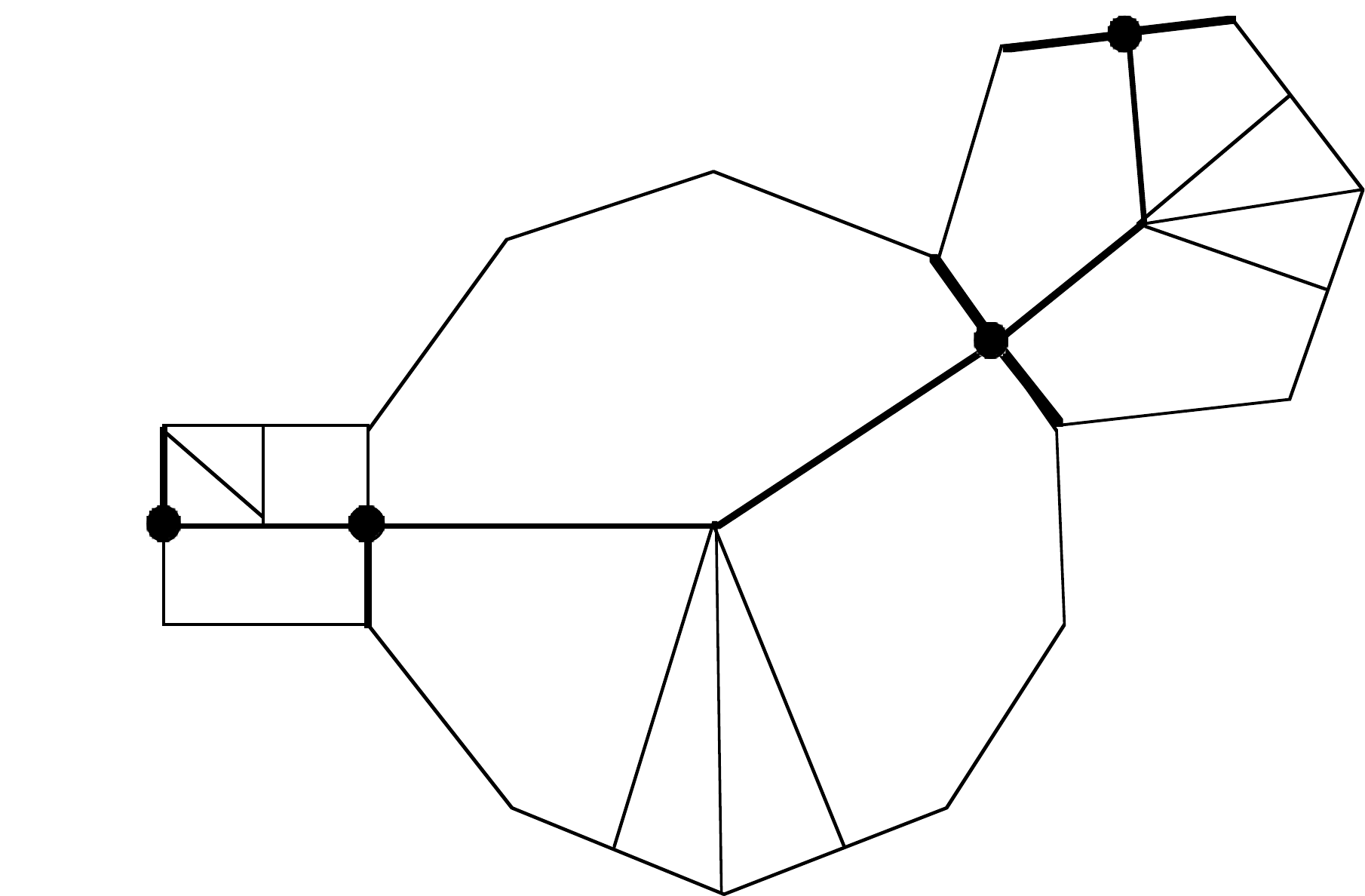}
\caption{Extension of a rank--$1$ field along an $e$--wall}
\end{figure}
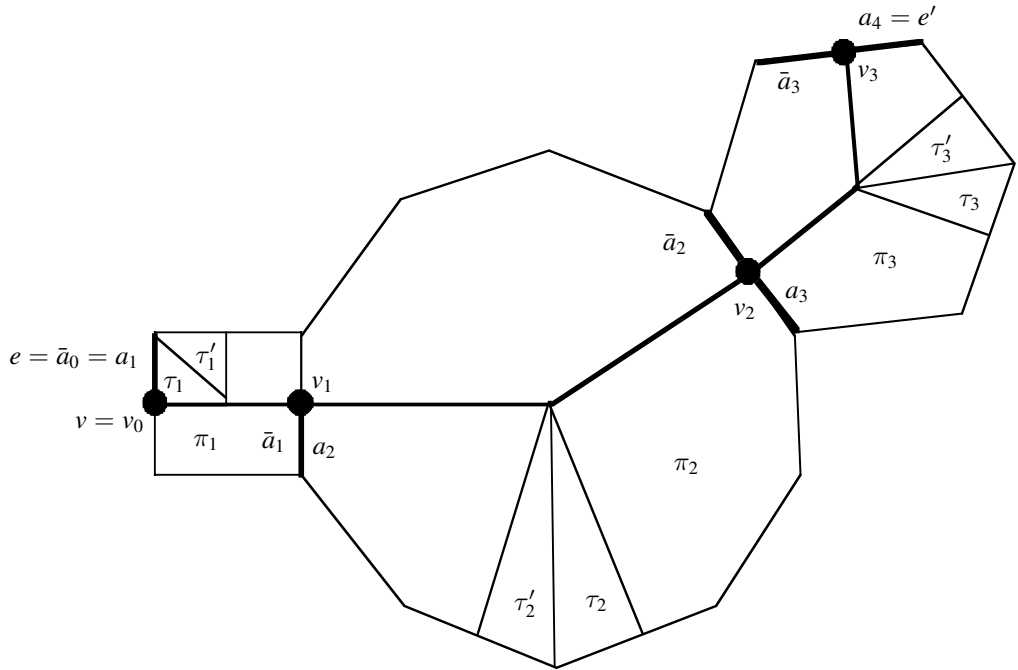

Let $e'$ be a semi-edge $e'$ of $X$ containing a
rank--1 vertex $w$ of $\size{M\!}$, and denote by $c=(v_0=v,v_1,\dots,v_n=w)$ the unique
injective combinatorial path in the tree
$T(M)$ from $v$ to $w$.  Using the notation
above, we define
$f(e')$ by setting $f(e')={f_{n+1}}$.
For any semi-edge $e''$ not adjacent to $\size{M\!}$
we set $f(e'')=1$. We thus obtain the unique
preimage of $f\in F(e)$ in $F$ under the
restriction map
$F\to F(e)$.

It remains to check that $F$ is invariant
under $\Gamma_M$. For $f\in F$ and
$\gamma\in\Gamma_M$ clearly $\gamma f$ is
trivial outside the set of semi-edges
adjacent to $\size{M\!}$. Also $\gamma f$ is $\sigma$--symmetric, because $\sigma$ is invariant
under $\Gamma$. The third relation defining
elements of $F$ is still satisfied by $\gamma
f$ because $\varphi$ is $\Gamma$--invariant
and $g_i(\gamma\tau,\gamma\sigma,\gamma f)$
is the conjugate by $\gamma$ of
$g_i(\tau,\sigma,f)$.
\end{proof}

We construct $\sigma',\varphi'$ and $\Gamma'$ using
the action of $\Gamma_M$ on the sets $F(c,M)$.

Since $F(\sigma,\varphi,M)$ is finite and $\Gamma_M$ acts on
it, some finite index subgroup of 
$\Gamma_M$ fixes $F(\sigma,\varphi,M)$ pointwise. By the hypotheses, there
exists a finite index subgroup
$\Gamma'(\sigma,\varphi)\subset \Gamma^0$ such that
$\Gamma'(\sigma,\varphi)_M$ acts trivially on $F$. We may
even assume that $\Gamma'(\sigma,\varphi)$ is normal in
$\Gamma^0$ and contained in $\Gamma$.

Consider all pairs $(\rho,\psi)$ with
a system $\rho$ of local reflections and a rank--2 field $\psi$ 
decomposing the holonomy
$h_{\rho}$ where $\rho$ and $\psi$ are invariant under  $\Gamma$.
Observe that there are finitely many such pairs, because of
$\Gamma$--invariance ($\Gamma$ is cocompact). For each
pair we consider a finite index normal subgroup
$\Gamma'(\rho,\psi)\subset \Gamma^0$ with $\Gamma'(\rho,\psi)\subset \Gamma$ such that the stabilizer of
$M$ in
$\Gamma'(\rho,\psi)$ acts trivially on $F(\rho,\psi,M)$.
Then the intersection of all such $\Gamma'(\rho,\psi)$ is a
finite index normal subgroup $\Gamma'$ of $\Gamma^0$ where $\Gamma'\subset
\Gamma$ and ${\Gamma'}_M$ acts trivially on all
the sets $F(\rho,\psi,M)$.

    Choose elements
$\gamma_1=1,\gamma_2,\dots,\gamma_k\in
\Gamma^0$ such that any $e$--wall $\gamma M$ with
$\gamma\in\Gamma^0$ is equivalent modulo
$\Gamma'$ to a unique $\gamma_i M=M_i$. We now
define a $\Gamma'$--invariant rank--1 field
$f$.

Pick  elements $f_i\in
F(\sigma,\varphi,M_i)$. We first check  that $f_i$ is invariant
under ${\Gamma'}_{M_i}$.  Conjugation
by ${\gamma_i}$ sends
$F({\gamma_i}^{-1}\sigma,{\gamma_i}^{-1}\varphi,M)$ to
$F(\sigma,\varphi,{\gamma_i}M)$. Observe  that
${\gamma_i}^{-1}\sigma$ and ${\gamma_i}^{-1}\varphi$ are invariant under
$\Gamma$ (because $\Gamma$ is normal in $\Gamma^0$). So by construction ${\Gamma'}_{M}$
fixes $F({\gamma_i}^{-1}\sigma,{\gamma_i}^{-1}\varphi,M)$. Thus
${\gamma_i}{\Gamma'}_{M}{\gamma_i}^{-1}$ fixes
$F(\sigma,\varphi,{\gamma_i}M)$. But $\Gamma'$ is normal in
$\Gamma^0$, so ${\gamma_i}{\Gamma'}_{M}{\gamma_i}^{-1}=
{\Gamma'}_{M_i}$.

For every translate $\gamma M$, $\gamma\in\Gamma^0$, we
now choose some
$\gamma'\in\Gamma'$ such that $\gamma'
\gamma_i M=\gamma M$ with $\gamma'=1$ if
$\gamma M$ is $\gamma_i M$ for some $i$.

For a semi-edge $e$ not adjacent to a geometric $e$--wall of
the form  $\gamma \size{M\!}$ with
$\gamma\in\Gamma^0$, we set $f(e)=1$.
For a semi-edge $e$ adjacent to $\gamma_i \size{M\!}$
we set $f(e)=f_i(e)$. Observe that by cleanliness $e$ is not adjacent to
an other $\gamma_j \size{M\!}$.

More generally if the semi-edge $e$ is adjacent to a geometric
$e$--wall of the form  $\gamma \size{M\!}$ with
$\gamma\in\Gamma^0$ and  $\gamma M=\gamma'\gamma_iM$, we set
$f(e)=\gamma'\circ
f_i({\gamma'}^{-1}e)\circ {\gamma'}^{-1}$. Again, by cleanliness $e$ is
not adjacent to  $\gamma''\size{M\!}$ with $\gamma M\neq \gamma''M$.

The rank--1 field $f$ is $\Gamma'$--invariant because each $
{\Gamma'}_{M_i}$ fixes $f_i$.
Since $\Gamma'\subset \Gamma$ each of the chosen elements $\gamma'$
preserves $\sigma$. Then by construction $f$ is $\sigma$--symmetric.
We set $\sigma'=\sigma f$. This new system $\sigma'$ of local reflections is also
$\Gamma'$--invariant.

We now compute the holonomy of $\sigma'$ at some triangle
$\tau\in{\mathcal T}$ and show how to decompose it. Let us denote by
$\pi$ the polygon of $X$ containing $\tau$. We have either
$\tau(\pi,M)=\tau$ or $\tau(\pi,M)=\tau'$.

If $\pi$ is not separated by an $e$--wall of the form $\gamma M$, $\gamma
\in \Gamma^0$, then for each semi-edge $e\subset\partial \pi$ we have
$f(e)=1$. We then have $h_{\sigma
'}(\tau)=h_{\sigma}(\tau)=\varphi(\tau')^{-1}\varphi(\tau)$, by \fullref{lem:holcompos}. We thus
set $\varphi'(\tau)=\varphi(\tau)$ and $\varphi'(\tau')=\varphi(\tau')$.

Assume now that $\pi$ is separated by at least one $\Gamma$--translate
of $M$. Since cleanliness implies that distinct $\Gamma$--translates of $M$ have
disjoint associated geometric $e$--walls, there exists a unique
translate
$\gamma M$ which separates
$\pi$. By \fullref{lem:holcompos}, we  have $h_{\sigma
'}(\tau(\pi,\gamma M))=h_{\sigma }(\tau(\pi,\gamma M))\circ
g_{2m_{\pi}}\circ\cdots\circ g_1$.  On the set of semi-edges adjacent
to $\gamma M$ the field $f$ coincides with an element of
$F(\sigma,\varphi,\gamma M)$. Thus the product
$g_{2m_{\pi}}\circ\cdots\circ g_1$ reduces to $g_{i_2}\circ g_{i_1}$,
and this latter is ${\varphi(\tau(\pi,\gamma M))}^{-1}$. We thus get
$h_{\sigma '}(\tau(\pi,\gamma M))=h_{\sigma }(\tau(\pi,\gamma
M))\circ\varphi(\tau(\pi,\gamma M))^{-1} =\varphi(\tau(\pi,\gamma
M)')^{-1}$, and then set $\varphi'(\tau(\pi,\gamma M))=1$ and
$\varphi'(\tau(\pi,\gamma M)')=\varphi(\tau(\pi,\gamma
M)')$.

If we extend $\varphi'$ outside ${\mathcal T}$ by setting
$\varphi'=1$ we obtain  a $\Gamma'$--invariant rank--2 field that
decomposes the holonomy of $\sigma'$ and has the required properties.
\end{proof}

\subsection{Finding a virtually invariant system of local reflections without holonomy}

\begin{thm}\label{thm:killehol}
Assume $(W,S)$ is two-dimensional
and $I$ does
not contain any subset $\{i,j,\ell\}$ with $m_{ij}\le 3,m_{j\ell}=3,m_{\ell
i}<+\infty$ (so that $X$ satisfies condition ($\mathrm{C^4}$) by \fullref{lem:condnpc}).

Let $\Gamma$ denote a residually finite uniform lattice of ${\rm
Aut}_{\rm rk}(A)$ in which the finite index subgroups of $e$--wall 
stabilizers are  separable. Then $\Gamma$ has a finite index subgroup
preserving a system of local reflections without holonomy.
\end{thm}

\begin{proof}
First by
\fullref{cor:virteclean},  there is a finite index subgroup of
$\Gamma$  acting as a clean uniform lattice. So we may and will assume
that $\Gamma$ is clean. By \fullref{lem:preservsystem}, the group 
$\Gamma$ preserves a system of local reflections $\sigma$. By
\fullref{rem:existdecomp} the system $\sigma$ admits a
$\Gamma$--invariant decomposition $\varphi$.

We choose $e$--walls $M_1,\dots,M_n$ such that any $e$--wall of $A$ is a
$\Gamma$--translate of a unique $e$--wall of the family $\{M_1,\dots,M_n\}$.

We first set $\Gamma^0=\Gamma,\varphi^0=\varphi$ and $\sigma^0=\sigma$. We apply
\fullref{thm:killholwall} to the $e$--wall $M_1$, using
\fullref{rem:engulf}. Thus we find a finite index normal subgroup
$\Gamma^1\subset\Gamma$, a system of local reflections $\sigma^1$ and
a rank--2 field $\varphi^1$ such that $\sigma^1$ and $\varphi^1$ are
$\Gamma^1$--invariant, $\varphi^1$ is a decomposition of the
holonomy of $\sigma^1$ and for every polygon $\pi$ either $\pi$ is
separated by no $\Gamma$--translate and $\varphi^1=\varphi^0$ on
${\mathcal T}_{\pi}$ or  $\pi$ is separated by some $\Gamma$--translate
$\gamma M$ and then
$\varphi^1(\tau(\pi,M))=1,\varphi^1(\tau(\pi,M)')=\varphi^0(\tau(\pi,M
)')$.

We apply again \fullref{thm:killholwall} to the subgroup
$\Gamma^1$, the system of local reflections $\sigma^1$, its
decomposition $\varphi^1$  and the $e$--wall $M_2$, thus  getting a finite
index normal subgroup $\Gamma^2\subset \Gamma$ with $\Gamma^2\subset
\Gamma^1$, together with a system $\sigma^2$ and a decomposition
$\varphi^2$ of the holonomy of $\sigma^2$ which is trivial along
$\Gamma$--translates of $M_2$. If we repeat this procedure until we
reach the $\Gamma$--translates of $M_n$, we obtain sequences
$\Gamma^1\supset\cdots\supset\Gamma^n$ (each normal of finite index in
$\Gamma$),
$\sigma^1, \dots,\sigma^n$ systems of local reflections, $\varphi^1,
\dots,\varphi^n$  with $\sigma^i,\varphi^i$ invariant under
$\Gamma^i$, $\varphi^i$ decomposition of the holonomy of $\sigma^i$.

 We claim that $\varphi^n=1$, so that
$\sigma^n$ has no holonomy.

 Indeed consider a polygon $\pi$ with ${2m}$
edges. The $e$--walls separating $\pi$ are in the
$\Gamma$--orbit of
 $M_{i_1},\dots,M_{i_{2m}}$ with
$i_1<\cdots<i_{2m}$. We denote by ${N}_j$ the
$\Gamma$--translate of $M_{i_j}$ which separates
$\pi$.

   For $i\not\in\{i_1,\dots,i_{2m}\}$ we know that
$\varphi^i=\varphi^{i-1}$ on  ${\mathcal T}_{\pi}$.

   To simplify notation set $\tau=\tau(\pi,N_1)$. Then we know that
$\varphi^{i_1}(\tau')=\varphi^{i_1-1}(\tau')$ and
$\varphi^{i_1}(\tau)=1$. Half of the $e$--walls $N$ that separate $\pi$
satisfy $\tau(\pi,N)=\tau$, and the other half satisfy
$\tau(\pi,N)=\tau'$. So there is a least positive integer $\ell\ (\le
{m}+1)$ such that $\tau(\pi,N_{\ell})=\tau'$. For every integer
$1\le j<\ell$ we have $\tau(\pi,N_j)=\tau$ and in fact
$\varphi^{i_j}=\varphi^{i_1}$ on  ${\mathcal T}_{\pi}$.

The new decomposition $\varphi^{i_{\ell}}$ is trivial on $\tau'$, and
takes the same value as $\varphi^{i_1}$ on $\tau$. This means that
$\varphi^\ell=1$ on  ${\mathcal T}_{\pi}$.

For $j>\ell$ the decompositions $\varphi^j$ keep the value of  the
previous $\varphi^{j-1}$  on one (or both) of the triangles of
${\mathcal T}_{\pi}$ and puts the value 1 on the other. So
$\varphi^n=\varphi^\ell=1$ on  ${\mathcal T}_{\pi}$.
\end{proof}

\begin{defn}
Let $X$ denote some polygonal complex.
 Assume $X$ is equipped with a piecewise
Euclidean nonpositively curved length metric $d$. A
subgroup $\Lambda$ of ${\rm Aut} (X)$
is \textit{$d$--convex\/} whenever there is a convex
subcomplex $C\subset X'$ of the first barycentric subdivision such that
$C$ is invariant under $\Lambda$ and $\Lambda$ is
discrete cocompact on $C$.

 For example if $X$ is the polygonal complex of a two-dimensional Coxeter
system $(W,S)$, and if $X$ is endowed with the
($\mathrm{C^4}$)-- or the $(\mathrm{C^2})$--metric, we will
consider ($\mathrm{C^4}$)--convex subgroups or
$(\mathrm{C^2})$--convex subgroups.
\end{defn}

\begin{exmp}\label{exmp:exconvex}
Assume that a polygonal complex $X$ satisfies  ($\mathrm{C^4}$).
Then by \fullref{lem:ewalltree}, any
geometric $e$--wall is a convex subcomplex of
$A$, endowed with the ($\mathrm{C^4}$)--metric. The stabilizer of an $e$--wall $M$ in
a uniform lattice $\Gamma$ of $X$ has finitely many orbits in $M$. Indeed
there are finitely many orbits of oriented edges of $M$ under $\Gamma$,
and if an element $\gamma\in\Gamma$ sends
$\overrightarrow e\in M$ inside $M$, then in fact $\gamma M=M$.

It follows that the stabilizer of an $e$--wall is cocompact on the
associated geometric $e$--wall, hence is
($\mathrm{C^4}$)--convex, and so is any finite index subgroup of the $e$--wall
stabilizer.

Similarly if $X$ satisfies  $(\mathrm{C^2})$ then any finite index
subgroup of the stabilizer of a wall in a uniform lattice of $X$ is
a $(\mathrm{C^2})$--convex subgroup.

Note that the trivial subgroup  $\{1\}$ is
$d$--convex for any metric $d$.
\end{exmp}

\begin{thm}\label{thm:commseparcox}
Assume $(W,S)$ is two-dimensional
and $I$ does
not contain any subset $\{i,j,\ell\}$ with $m_{ij}\le 3,m_{j\ell}=3,m_{\ell
i}<+\infty$, so that $X$ satisfies  ($\mathrm{C^4}$).

Let $\Gamma$ denote a   uniform lattice of
${\rm Aut}_{\rm rk}(A)$ in which each
($\mathrm{C^4}$)--convex subgroup is  separable.  Then
$\Gamma$ is commensurable in    ${\rm
Aut}_{\rm rk}(A)$ with $W$.
\end{thm}

\begin{proof}
 The condition implies that $\Gamma$ is residually finite and that the
finite index subgroups of $e$--wall stabilizers are separable (see
\fullref{exmp:exconvex}). By \fullref{thm:killehol}, the group
$\Gamma$ has a finite index subgroup $\Gamma'$ preserving a system of
local reflections $\sigma$ without holonomy. Now by 
\fullref{cor:conjugsystem}, there is an automorphism $\varphi$ of
${\rm Aut}_{\rm rk}(A)$  sending $\sigma$ onto
$\sigma^W$. Also, conjugation by $\varphi$ identifies ${\rm Aut}_{\rm
rk}(A,\sigma)$  with ${\rm Aut}_{\rm rk}(A,\sigma^W)$, an extension of $W$
of finite index. Since
$\Gamma'$ is a uniform lattice, it has  finite index in the uniform
lattice
${\rm Aut}_{\rm rk}(A,\sigma)$. Thus conjugation by $\varphi$ maps
$\Gamma'$ onto a group $\Gamma''$ such that $\Gamma''\cap W$ is of finite
index in both $\Gamma''$ and $W$. The commensurability follows.
\end{proof}

\begin{proof}[Proof of \fullref{thm:commseparbis}]
The negative curvature condition insures that $(W,S)$ is two-dimensional
and  that $I$ contains no subset $\{i,j,\ell\}$ with $m_{ij}\le 3,m_{j\ell}=3, m_{\ell
i}<+\infty$, so that $X$ satisfies  ($\mathrm{C^4}$) (see \fullref{exmp:mconstant}).

Clearly ($\mathrm{C^4}$)--convex subgroups are quasiconvex subgroups. So if
all quasiconvex subgroups are separable 
\fullref{thm:commseparcox} applies.
\end{proof}

\bibliographystyle{gtart}
\bibliography{link}

\end{document}